\documentclass[10pt]{amsart}

\usepackage{latexsym,exscale,enumerate}
\usepackage{mathtools,amsmath,mathscinet,amsfonts,amssymb,amsthm,bbm,euscript}
\usepackage{amscd, stmaryrd}
\usepackage[normalem]{ulem}
\usepackage[all]{xy}
\SelectTips{cm}{}
\usepackage{graphicx}
\usepackage{wrapfig}
\usepackage{xcolor}
\usepackage{bm}

\RequirePackage{color}
\definecolor{references}{rgb}{.7,.1,.6}

\RequirePackage[pdftex,
colorlinks = true,
urlcolor = references, 
citecolor = references, 
linkcolor = references, 
]
{hyperref}

\usepackage{tikz}
\usepackage{tikz-cd}
\usetikzlibrary{math}
\usetikzlibrary{decorations.markings}
\usetikzlibrary{decorations.pathreplacing}
\usetikzlibrary{arrows,shapes,positioning}
\tikzstyle directed=[postaction={decorate,decoration={markings,
    mark=at position #1 with {\arrow{>}}}}]
\tikzstyle rdirected=[postaction={decorate,decoration={markings,
    mark=at position #1 with {\arrow{<}}}}]

\tikzset{anchorbase/.style={baseline={([yshift=-0.5ex]current bounding box.center)}},
    tinynodes/.style={font=\tiny,text height=0.75ex,text depth=0.15ex},
    smallnodes/.style={font=\scriptsize,text height=0.75ex,text depth=0.15ex},
    >={Latex[length=1mm, width=1.5mm]}
  }
\tikzcdset{arrow style=tikz, diagrams={>=stealth}}

  \newcommand{\pu}{to [out=90,in=270]}
  \newcommand{\pr}{to [out=0,in=180]}

\allowdisplaybreaks

\colorlet{green}{black!30!green}
\newcommand{\BLUE}[1]{\textcolor[rgb]{0.00,0.00,0.80}{#1}}
\newcommand{\GREEN}[1]{\textcolor[rgb]{0.00,0.70,0.00}{#1}}

\newcommand{\GRAY}[1]{\textcolor[rgb]{0.50,0.50,0.50}{#1}}
\definecolor{CQG}{RGB}{0,153,76}
\definecolor{FS}{RGB}{0,76,153} 
\def\CQGsgn#1{{\color{CQG}#1}}

\newcommand{\arxiv}[1]{\href{https://arxiv.org/abs/#1}{\small  arXiv:#1}}

\newcommand{\googlebooks}[1]{(preview at \href{https://books.google.com/books?id=#1}{google books})}

\newcommand{\numdam}[1]{}

\def\emph#1{{\sl #1\/}}
\def\ie{{\sl i.e.\;}}
\def\eg{{\sl e.g.\;}}

\def\cal#1{\mathcal{#1}}

\def\mf#1{\mathfrak{#1}}

\let\hat=\widehat
\let\tilde=\widetilde

\renewcommand{\to}{\rightarrow}

\newcommand{\scs}{\scriptstyle}

\def\B{\mathbb{B}}

\def\D{\mathbb{D}}
\def\Eb{\mathbb{E}}

\def\N{\mathbb{N}}

\def\X{\mathbb{X}}

\def\Z{\mathbb{Z}}
\def\K{\mathbb{Q}}

\renewcommand{\aa}{\bm{a}}
\newcommand{\bb}{\bm{b}}

\newcommand{\AS}{\EuScript A}

\newcommand{\CS}{\EuScript C}

\newcommand{\KS}{\EuScript K}

\newcommand{\US}{\EuScript U}
\newcommand{\USd}{\dot{\EuScript U}}

\renewcommand{\a}{\alpha}
\renewcommand{\b}{\beta}

\renewcommand{\d}{\delta}
\newcommand{\e}{\varepsilon}
\newcommand{\lbd}{\lambda}

\renewcommand{\phi}{\varphi}
\renewcommand{\theta}{\vartheta}

\newcommand{\slm}{\mf{sl}_m}
\newcommand{\slnn}[1]{\mf{sl}_{#1}}

\newcommand{\glm}{\mf{gl}_m}
\newcommand{\glnn}[1]{\mf{gl}_{#1}}

\newcommand{\SBim}{\mathrm{SBim}}
\newcommand{\SSBim}{\mathrm{SSBim}}

\newcommand{\symg}{\mathfrak{S}}

\newcommand{\E}{{\sf{E}}}

\newcommand{\F}{{\sf{F}}}

\newcommand{\Id}{\mathrm{id}}

\newcommand{\one}{\mathbbm{1}}
\newcommand{\oone}{\mathbf{1}}
\newcommand{\ooone}{\mathsf{1}}

\newcommand{\Br}{\operatorname{Br}}

\newcommand{\cone}{\operatorname{cone}}

\newcommand{\End}{\operatorname{End}}

\newcommand{\Hom}{\operatorname{Hom}}

\newcommand{\largewedge}{\mbox{\Large $\wedge$}}

\newcommand{\rk}{\operatorname{\rk}}

\newcommand{\Sym}{\operatorname{Sym}}

\newcommand{\tw}{\operatorname{tw}}

\newcommand{\wt}{\operatorname{wt}}

\newcommand{\inv}{^{-1}}

\newcommand{\qdeg}{\mathbf{q}}
\newcommand{\tdeg}{\mathbf{t}}

\newcommand{\M}{\mathbb{M}}
\newcommand{\Fr}{\mathbb{F}}
\newcommand{\leftX}{\mathbb{X}}
\newcommand{\rightX}{\mathbb{X}'}

\newcommand{\leftM}{\mathbb{M}}
\newcommand{\rightM}{\mathbb{M}'}

\newcommand{\I}{I}

\newcommand{\MCS}{\mathrm{MCS}}
\newcommand{\MCSmin}{\mathsf{MCS}}

\newcommand{\KMCS}{\mathrm{KMCS}}
\newcommand{\KMCSmin}{\mathsf{KMCS}}

\newcommand{\MCCS}{\mathrm{MCCS}}
\newcommand{\MCCSmin}{\mathsf{MCCS}}

\newcommand{\hComp}{\star}

\newcommand{\qbinom}[2]{\genfrac[]{0pt}{2}{#1}{#2}}

\newcommand{\Schur}{\mathfrak{s}}

\newcommand{\cre}{\mathbf{cr}}
\newcommand{\col}{\mathbf{col}}
\newcommand{\zip}{\mathbf{zip}}
\newcommand{\un}{\mathbf{un}}

\newcommand{\CQGbox}[1]{
\begin{tikzpicture}
\node[draw,  fill=white,rounded corners=4pt,inner sep=3pt] (X) at (0,.75) {$\scs#1$};
\end{tikzpicture}}      
\newcommand{\CQGbbox}[1]{
\begin{tikzpicture}
\node[draw,  fill=white,rounded corners=4pt,inner sep=4pt] (X) at (0,.75) {$\scs#1$};
\end{tikzpicture}}

\theoremstyle{plain}
\newtheorem{thm}{Theorem}[section]
\newtheorem{prop}[thm]{Proposition}
\newtheorem{proposition}[thm]{Proposition}
\newtheorem{cor}[thm]{Corollary}

\newtheorem{lem}[thm]{Lemma}
\newtheorem{lemma}[thm]{Lemma}

\theoremstyle{definition}
\newtheorem{defi}[thm]{Definition}
\newtheorem{definition}[thm]{Definition}

\newtheorem{rem}[thm]{Remark}
\newtheorem{remark}[thm]{Remark}

\newtheorem{exa}[thm]{Example}

\newtheorem{conv}[thm]{Convention}
\newtheorem*{exa-nono}{Example}

\begin{document}

\author{Matthew Hogancamp}

\address{Department of Mathematics, Northeastern University, 360 Huntington Ave, Boston,
MA 02115, USA}
\email{m.hogancamp@northeastern.edu}

\author{David~E.~V.~Rose}

\address{Department of Mathematics, University of North Carolina, 
Phillips Hall, CB \#3250, UNC-CH, 
Chapel Hill, NC 27599-3250, USA
\href{https://davidev.web.unc.edu/}{davidev.web.unc.edu}}
\email{davidrose@unc.edu}

\author{Paul Wedrich}

\address{P.W.: Max Planck Institute for Mathematics,
Vivatsgasse 7, 53111 Bonn, Germany 
AND Mathematical Institute, University of Bonn,
Endenicher Allee 60, 53115 Bonn, Germany
\href{http://paul.wedrich.at}{paul.wedrich.at}}
\email{p.wedrich@gmail.com}

\title{A skein relation for singular Soergel bimodules}

\begin{abstract} We study the skein relation that governs the HOMFLYPT invariant
	of links colored by one-column Young diagrams. 
	Our main result is a categorification of this colored skein relation.
	This takes the form of a homotopy equivalence between two one-sided twisted
	complexes constructed from Rickard complexes of singular Soergel bimodules
	associated to braided webs. Along the way, we prove a conjecture
	of Beliakova--Habiro relating the colored 2-strand full twist complex with 
	the categorical ribbon element for quantum $\slnn{2}$.
\end{abstract}

\maketitle



\section{Introduction}
\label{s:introone}
The HOMFLYPT polynomial is an invariant of framed oriented links 
that is determined by the skein relation
\begin{equation}\label{eq:HOMFLYskein}
	\left[
\begin{tikzpicture}[rotate=90,scale=.5,smallnodes,anchorbase]
	\draw[very thick,->] (1,-1) to [out=90,in=270] (0,1);
	\draw[line width=5pt,color=white,->] (0,-1) to [out=90,in=270] (1,1);
	\draw[very thick,->] (0,-1) to [out=90,in=270] (1,1);
\end{tikzpicture}
\right]
-
\left[
\begin{tikzpicture}[rotate=90,scale=.5,smallnodes,anchorbase]
	\draw[very thick,->] (0,-1) to [out=90,in=270] (1,1);
	\draw[line width=5pt,color=white,->] (1,-1) to [out=90,in=270] (0,1);
	\draw[very thick,->] (1,-1) to [out=90,in=270] (0,1);
\end{tikzpicture}
\right]
= 
(q-q\inv) 
\left[
\begin{tikzpicture}[scale=.5,smallnodes,rotate=90,anchorbase]
	\draw[very thick,->] (0,-2) to (0,0);
	\draw[very thick,->] (1,-2) to (1,0);
\end{tikzpicture}
\right]
\end{equation}
together with its behavior under framing change and disjoint union, and its value on the
unknot. Algebraically, the HOMFLYPT polynomial can be
obtained from the following two-step process.  
First, one considers the type $A$ Hecke algebra $\mathrm{H}_n$, 
i.e. the quotient of the (group algebra of the) $n$-strand braid
group $\Br_n$ by the relation \eqref{eq:HOMFLYskein}. As such, any 
$n$-strand braid $\beta$ determines a well-defined element $[\beta] \in \mathrm{H}_n$. 
Second, there exists a linear map 
$\mathrm{H}_n \rightarrow \Z[q,q\inv, 
	\frac{\mathrm{a} - \mathrm{a}\inv}{q - q\inv}]$, 
known as the \emph{Jones-Ocneanu trace}, 
which gives a Markov trace on the braid group. Applying the latter to
the element of $\mathrm{H}_n$ assigned to a braid gives the HOMFLYPT polynomial
of the braid closure.

The triply-graded Khovanov--Rozansky homology \cite{MR2421131, MR2339573} is a
categorification of the HOMFLYPT polynomial, which can be constructed using a
similar framework. First, the category $\SBim_n$ of type $A_{n-1}$ 
Soergel bimodules provides a categorical analogue of the Hecke algebras $\mathrm{H}_n$. 
Paralleling the relation between $\Br_n$ and $\mathrm{H}_n$ is Rouquier's construction
\cite{0409593,MR2258045}, which associates to each braid (word) $\b$ a complex
$\llbracket \beta \rrbracket$ of Soergel bimodules. In particular, the skein
relation \eqref{eq:HOMFLYskein} is promoted to a homotopy equivalence:
\begin{equation}
	\label{eq:uncoloredskein}
 \cone\left(
		\left\llbracket
\:\begin{tikzpicture}[rotate=90,scale=.5,smallnodes,anchorbase]
	\draw[very thick,->] (1,-1) to [out=90,in=270] (0,1);
	\draw[line width=5pt,color=white,->] (0,-1) to [out=90,in=270] (1,1);
	\draw[very thick,->] (0,-1) to [out=90,in=270] (1,1);
\end{tikzpicture}
	\:\right\rrbracket
	\xrightarrow{f}
	\left\llbracket\:
\begin{tikzpicture}[rotate=90,scale=.5,smallnodes,anchorbase]
	\draw[very thick,->] (0,-1) to [out=90,in=270] (1,1);
	\draw[line width=5pt,color=white,->] (1,-1) to [out=90,in=270] (0,1);
	\draw[very thick,->] (1,-1) to [out=90,in=270] (0,1);
\end{tikzpicture}
	\:\right\rrbracket
	\right)
	\simeq 
	\cone \left(
	\qdeg
	\left\llbracket\:
	\begin{tikzpicture}[scale=.5,smallnodes,rotate=90,anchorbase]
		\draw[very thick,->] (0,-2) to (0,0);
		\draw[very thick,->] (1,-2) to (1,0);
	\end{tikzpicture}
	\:\right\rrbracket
	\xrightarrow{g}
	\qdeg\inv
	\left\llbracket\:
	\begin{tikzpicture}[scale=.5,smallnodes,rotate=90,anchorbase]
		\draw[very thick,->] (0,-2) to (0,0);
		\draw[very thick,->] (1,-2) to (1,0);
	\end{tikzpicture}
	\:\right\rrbracket
	\right)
\end{equation}
for appropriate chain maps $f$ and $g$. Finally a categorical
analogue of the Jones--Ocneanu trace is provided by the Hochschild (co)homology
functor.
\medskip

In recent years, it has proven to be increasingly important to consider not just
categorifications of the HOMFLYPT polynomial, but also its colored variants,
especially those where the coloring consists of 1-column Young
diagrams\footnote{The specialization of the thus colored HOMFLY polynomial at
$\mathrm{a}=q^m$ recovers the $\glm$ Reshetikhin--Turaev invariant with colorings by
fundamental representations, a.k.a. exterior powers of the defining
representation.}. The two relevant algebraic structures in the decategorified
story are the \emph{colored braid groupoid} and the \emph{Hecke algebroid}. Both
can be considered as categories whose objects are finite sequences
\emph{colors}, i.e. natural numbers encoding the numbers of boxes in one-column
Young diagrams, such as $\aa=(a_1,\dots, a_r)$ and $\bb=(b_1,\dots, b_s)$. 

In the colored braid groupoid $\mathbf{Br}$, morphisms from $\aa$ to $\bb$
exist only if $r=s$, in which case they are braids $\beta\in \Br_r$
whose strands connect equal colors $b_{\beta(i)}=a_i$. In the Hecke algebroid
$\mathbf{H}$, morphisms from $\aa$ to $\bb$ exist only if $|\aa|=|\bb|=n$, in
which case they are given by $e_{\bb} \mathrm{H}_n e_{\aa}$, where $e_{\aa}\in
\mathrm{H}_n$ (and similarly $e_{\bb}$) is a certain partially antisymmetrizing
idempotent, modeled on the Young antisymmetrizer for 
$\symg_{a_1}\times\cdots \times\symg_{a_r}$. 
The maps $[-]\colon \Br_n \to \mathrm{H}_n$ now induce a functor 
\[[-] \colon \mathbf{Br} \to \mathbf{H}\] given by sending a colored braid
$\beta_{\aa}$ to the Hecke algebra element obtained by cabling the strands of
$\beta$ with multiplicities specified by $\aa$, and then composing with the
idempotent $e_{\aa}$. 

Computations in the Hecke algebroid are facilitated by a diagrammatic calculus
of braided webs that goes back to Murakami--Ohtsuki--Yamada~\cite{MOY}, 
and can be understood as the $m \to \infty$ limit of the 
web calculus from \cite{CKM}.
For example, 
the decategorification of Theorem~\ref{thm:intro skein} below
gives the following identity in $\mathbf{H}$, 
which to our knowledge is new:
\[
\sum_{s=0}^{b} (-q^{b-1})^s
		\left[
			\begin{tikzpicture}[scale=.35,smallnodes,rotate=90,anchorbase]
				\draw[very thick] (1,-1) to [out=150,in=210] (0,1); 
				\draw[line width=5pt,color=white] (0,-2) to [out=90,in=270] (1,2);
				\draw[very thick] (0,-2) node[right]{$a$}to [out=90,in=270] (1,2) node[left]{$a$};
				\begin{scope} 
				\clip (1,-.5) rectangle (0,.75) ;
				\draw[line width=5pt,color=white] (1,-1) to [out=30,in=330] (0,1); 
				\end{scope}
				\draw[very thick] (1,-1) to [out=30,in=330] node[above,yshift=-2pt]{$s$} (0,1); 
				\draw[very thick] (1,-2) node[right]{$b$} to (1,-1); 
				\draw[very thick] (0,1) to (0,2) node[left]{$b$};
			\end{tikzpicture}
		\right]
		 = (-1)^b q^{-b} \prod_{i=1}^{b} (1-q^{2i}) 
		 \left[
\begin{tikzpicture}[smallnodes,rotate=90,anchorbase,scale=.525]
	\draw[very thick] (0,.25) to [out=150,in=270] (-.25,1) node[left,xshift=2pt]{$b$};
	\draw[very thick] (.5,.5) to (.5,1) node[left,xshift=2pt]{$a$};
	\draw[very thick] (0,.25) to node[right,yshift=-1pt]{$a{-}b$} (.5,.5);
	\draw[very thick] (.5,-1) node[right,xshift=-2pt]{$b$} to [out=90,in=330] (.5,.5);
	\draw[very thick] (-.25,-1)node[right,xshift=-2pt]{$a$} to [out=90,in=270] (0,.25);
\end{tikzpicture}
 \right] \, .
\]

For technical reasons we will actually be mostly interested in the following
(equivalent) relation:
\begin{equation}
	\label{eq:colskeindecat}
\sum_{s=0}^{b} (-q^{b-1})^s
\left[
\begin{tikzpicture}[scale=.4,smallnodes,rotate=90,anchorbase]
	\draw[very thick] (1,-1) to [out=150,in=270] (0,0); 
	\draw[line width=5pt,color=white] (0,-2) to [out=90,in=270] (.5,0) to [out=90,in=270] (0,2);
	\draw[very thick] (0,-2) node[right=-2pt]{$a$} to [out=90,in=270] (.5,0) 
		to [out=90,in=270] (0,2) node[left=-2pt]{$a$};
	\draw[very thick] (1,1) to (1,2) node[left=-2pt]{$b$};
	\draw[line width=5pt,color=white] (0,0) to [out=90,in=210] (1,1); 
	\draw[very thick] (0,0) to [out=90,in=210] (1,1); 
	\draw[very thick] (1,-2) node[right=-2pt]{$b$} to (1,-1); 
	\draw[very thick] (1,-1) to [out=30,in=330] node[above=-2pt]{$s$} (1,1); 
\end{tikzpicture}
\right]
= (-1)^b q^{b(a-b-1)} \prod_{i=1}^{b} (1-q^{2i}) 
\left[
\begin{tikzpicture}[scale=.4,smallnodes,anchorbase,rotate=270]
	\draw[very thick] (1,-1) to [out=150,in=270] (0,1) to (0,2) node[right=-2pt]{$b$}; 
	\draw[line width=5pt,color=white] (0,-2) to (0,-1) to [out=90,in=210] (1,1);
	\draw[very thick] (0,-2) node[left=-2pt]{$b$} to (0,-1) to [out=90,in=210] (1,1);
	\draw[very thick] (1,1) to (1,2) node[right=-2pt]{$a$};
	\draw[very thick] (1,-2) node[left=-2pt]{$a$} to (1,-1); 
	\draw[very thick] (1,-1) to [out=30,in=330] node[below=-1pt]{$a{-}b$} (1,1); 
\end{tikzpicture}
\right] \, .
\end{equation}
\medskip

The goal of this paper is to prove a categorical analog of the colored
skein relation \eqref{eq:colskeindecat}, which takes the form of a homotopy
equivalence of complexes constructed from \emph{Rickard
complexes} of \emph{singular Soergel bimodules}. We now discuss these
ingredients in turn. 
\smallskip

Singular Soergel bimodules~\cite{Williamson-thesis} in type $A$ form a monoidal
2-category $\SSBim$, which provides a categorification of the Hecke algebroid
$\mathbf{H}$ in the same sense in which $\SBim_n$ categorifies the Hecke
algebra $\mathrm{H}_n$. Moreover, $\SSBim$ is obtained as the idempotent
completion of a monoidal 2-category of so-called \emph{singular Bott--Samelson
bimodules}---composites of induction and restriction bimodules between partially
symmetric polynomial rings, modeled on planar webs as drawn above; see
Section~\ref{ss:ssbim} for details.
\smallskip

Rickard complexes can be considered as generalizations of the Rouquier complexes
for Artin generators to the colored setting. They entered higher representation
theory in the seminal work of Chuang--Rouquier~\cite{MR2373155} in the context
of $\slnn{2}$-actions on categories. Closer to our setting, Rickard
complexes of singular Soergel bimodules were proposed as the basic ingredient
for a colored version of triply-graded Khovanov--Rozansky homology by
Mackaay--Sto\v{s}i\'{c}--Vaz~\cite{MR2746676}, a proposal that was subsequently
implemented by Webster--Williamson~\cite{MR3687104}. We will describe these in
detail in Section~\ref{sec:colbraid}. Just like Rouquier complexes, we denote
the Rickard complexes of colored braids\footnote{We will also adopt 
this notation for certain complexes of singular Bott--Samelson bimodules 
that are most-easily described as braided webs.} $\beta$ by $\llbracket \beta\rrbracket$. 
\smallskip

The ``right-hand'' side of our categorified colored skein relation involves the
following complex: 
\[
\MCSmin_{a,b} := \left( 
		\begin{tikzpicture}[smallnodes,rotate=90,baseline=.1em,scale=.525]
			\draw[very thick] (0,.25) to [out=150,in=270] (-.25,1) node[left,xshift=2pt]{$a$};
			\draw[very thick] (.5,.5) to (.5,1) node[left,xshift=2pt]{$b$};
			\draw[very thick] (0,.25) to  (.5,.5);
			\draw[very thick] (0,-.25) to (0,.25);
			\draw[dotted] (.5,-.5) to [out=30,in=330] node[above=-2pt]{$0$} (.5,.5);
			\draw[very thick] (0,-.25) to  (.5,-.5);
			\draw[very thick] (.5,-1) node[right,xshift=-2pt]{$b$} to (.5,-.5);
			\draw[very thick] (-.25,-1)node[right,xshift=-2pt]{$a$} to [out=90,in=210] (0,-.25);
		\end{tikzpicture}
		\to  
		\begin{tikzpicture}[smallnodes,rotate=90,baseline=.1em,scale=.525]
			\draw[very thick] (0,.25) to [out=150,in=270] (-.25,1) node[left,xshift=2pt]{$a$};
			\draw[very thick] (.5,.5) to (.5,1) node[left,xshift=2pt]{$b$};
			\draw[very thick] (0,.25) to  (.5,.5);
			\draw[very thick] (0,-.25) to (0,.25);
			\draw[very thick] (.5,-.5) to [out=30,in=330] node[above=-2pt]{$1$} (.5,.5);
			\draw[very thick] (0,-.25) to  (.5,-.5);
			\draw[very thick] (.5,-1) node[right,xshift=-2pt]{$b$} to (.5,-.5);
			\draw[very thick] (-.25,-1)node[right,xshift=-2pt]{$a$} to [out=90,in=210] (0,-.25);
		\end{tikzpicture}
		\to 
\begin{tikzpicture}[smallnodes,rotate=90,baseline=.1em,scale=.525]
			\draw[very thick] (0,.25) to [out=150,in=270] (-.25,1) node[left,xshift=2pt]{$a$};
			\draw[very thick] (.5,.5) to (.5,1) node[left,xshift=2pt]{$b$};
			\draw[very thick] (0,.25) to  (.5,.5);
			\draw[very thick] (0,-.25) to (0,.25);
			\draw[very thick] (.5,-.5) to [out=30,in=330] node[above=-2pt]{$2$} (.5,.5);
			\draw[very thick] (0,-.25) to  (.5,-.5);
			\draw[very thick] (.5,-1) node[right,xshift=-2pt]{$b$} to (.5,-.5);
			\draw[very thick] (-.25,-1)node[right,xshift=-2pt]{$a$} to [out=90,in=210] (0,-.25);
\end{tikzpicture}
\to \cdots \to
		\begin{tikzpicture}[smallnodes,rotate=90,baseline=.1em,scale=.525]
			\draw[very thick] (0,.25) to [out=150,in=270] (-.25,1) node[left,xshift=2pt]{$a$};
			\draw[very thick] (.5,.5) to (.5,1) node[left,xshift=2pt]{$b$};
			\draw[dotted] (0,.25) to  (.5,.5);
			\draw[very thick] (0,-.25) to (0,.25);
			\draw[very thick] (.5,-.5) to [out=30,in=330] node[above=-2pt]{$b$} (.5,.5);
			\draw[dotted] (0,-.25) to  (.5,-.5);
			\draw[very thick] (.5,-1) node[right,xshift=-2pt]{$b$} to (.5,-.5);
			\draw[very thick] (-.25,-1)node[right,xshift=-2pt]{$a$} to [out=90,in=210] (0,-.25);
		\end{tikzpicture}
		\right).
\]
Here the webs represent certain singular Bott-Samelson bimodules, 
and (for now) we omit all degree shifts.
In Proposition \ref{prop:topological shifted rickard}, 
we show that
\[
\MCSmin_{a,b}\simeq
\left\llbracket
\begin{tikzpicture}[scale=.4,smallnodes,anchorbase,rotate=270]
\draw[very thick] (1,-1) to [out=150,in=270] (0,1) to (0,2) node[right=-2pt]{$b$}; 
\draw[line width=5pt,color=white] (0,-2) to (0,-1) to [out=90,in=210] (1,1);
\draw[very thick] (0,-2) node[left=-2pt]{$b$} to (0,-1) to [out=90,in=210] (1,1);
\draw[very thick] (1,1) to (1,2) node[right=-2pt]{$a$};
\draw[very thick] (1,-2) node[left=-2pt]{$a$} to (1,-1); 
\draw[very thick] (1,-1) to [out=30,in=330] node[below=-1pt]{$a{-}b$} (1,1); 
\end{tikzpicture}
\right\rrbracket \, .
\]
This holds for all integers $a,b$ provided we interpret the right-hand side as
zero when $a<b$.  Note that $\MCSmin_{a,b}$ (like any complex of singular
Soergel bimodules) is a complex of modules over an appropriate ring of partially
symmetric functions $\Sym(\X_1|\X_2|\X_1'|\X_2')$, so we may form the tensor
product
\[
K(\MCSmin_{a,b}):=\MCSmin_{a,b}\otimes_{ \Sym(\leftX_2|\rightX_2)} K,
\]
where $K$ is the Koszul complex
\[
K := \Sym(\X_2|\X_2')\otimes\largewedge[\xi_1,\ldots,\xi_b]\, , \quad
\d(\xi_k) = \sum_{i+j=k} (-1)^j h_i(\leftX_2)e_j(\rightX_2)
\]
(see Definition \ref{def:Koszul cx}).
The colored skein relation then takes the following form.  
\begin{thm}\label{thm:intro skein} 
The complex $K(\MCSmin_{a,b})$ is homotopy
equivalent to the following one-sided twisted complexes: 
\begin{equation}\label{eq:conv}
\left(
\left\llbracket
\begin{tikzpicture}[scale=.35,smallnodes,rotate=90,baseline=.4em]
	\draw[very thick] (1,-1) to [out=150,in=270] (0,0); 
	\draw[line width=5pt,color=white] (0,-2) to [out=90,in=270] (.5,0) 
		to [out=90,in=270] (0,2);
	\draw[very thick] (0,-2) node[right=-2pt]{$a$}to [out=90,in=270] (.5,0) 
		to [out=90,in=270] (0,2) node[left=-2pt]{$a$};
	\draw[line width=5pt,color=white] (0,0) to [out=90,in=210] (1,1); 
	\draw[very thick] (0,0) to [out=90,in=210] (1,1); 
	\draw[very thick] (1,-2) node[right=-2pt]{$b$} to (1,-1); 
	\draw[dotted] (1,-1) to [out=30,in=330] node[above,yshift=-2pt]{$0$} (1,1); 
	\draw[very thick] (1,1) to (1,2) node[left=-2pt]{$b$};
\end{tikzpicture}
\right\rrbracket 
\to
\left\llbracket
\begin{tikzpicture}[scale=.35,smallnodes,rotate=90,baseline=.4em]
	\draw[very thick] (1,-1) to [out=150,in=270] (0,0); 
	\draw[line width=5pt,color=white] (0,-2) to [out=90,in=270] (.5,0) 
		to [out=90,in=270] (0,2);
	\draw[very thick] (0,-2) node[right=-2pt]{$a$}to [out=90,in=270] (.5,0) 
		to [out=90,in=270] (0,2) node[left=-2pt]{$a$};
	\draw[line width=5pt,color=white] (0,0) to [out=90,in=210] (1,1); 
	\draw[very thick] (0,0) to [out=90,in=210] (1,1); 
	\draw[very thick] (1,-2) node[right=-2pt]{$b$} to (1,-1); 
	\draw[very thick] (1,-1) to [out=30,in=330] node[above,yshift=-2pt]{$1$} (1,1); 
	\draw[very thick] (1,1) to (1,2) node[left=-2pt]{$b$};
\end{tikzpicture}
\right\rrbracket 
\to \cdots \to
\left\llbracket
\begin{tikzpicture}[scale=.35,smallnodes,rotate=90,baseline=.4em]
	\draw[dotted] (1,-1) to [out=150,in=270] (0,0); 
	\draw[line width=5pt,color=white] (0,-2) to [out=90,in=270] (.5,0) 
		to [out=90,in=270] (0,2);
	\draw[very thick] (0,-2) node[right=-2pt]{$a$}to [out=90,in=270] (.5,0) 
		to [out=90,in=270] (0,2) node[left=-2pt]{$a$};
	\draw[dotted] (0,0) to [out=90,in=210] (1,1); 
	\draw[very thick] (1,-2) node[right=-2pt]{$b$} to (1,-1); 
	\draw[very thick] (1,-1) to [out=30,in=330] node[above,yshift=-2pt]{$b$} (1,1); 
	\draw[very thick] (1,1) to (1,2) node[left=-2pt]{$b$};
\end{tikzpicture}
\right\rrbracket 
\right) 
\simeq 
K\left(\left\llbracket\begin{tikzpicture}[scale=.35,smallnodes,anchorbase,rotate=270]
	\draw[very thick] (1,-1) to [out=150,in=270] (0,1) to (0,2) node[right=-2pt]{$b$}; 
	\draw[line width=5pt,color=white] (0,-2) to (0,-1) to [out=90,in=210] (1,1);
	\draw[very thick] (0,-2) node[left=-2pt]{$b$} to (0,-1) to [out=90,in=210] (1,1);
	\draw[very thick] (1,1) to (1,2) node[right=-2pt]{$a$};
	\draw[very thick] (1,-2) node[left=-2pt]{$a$} to (1,-1); 
	\draw[very thick] (1,-1) to [out=30,in=330] node[below]{$a{-}b$} (1,1); 
\end{tikzpicture}\right\rrbracket \right).
\end{equation}
\end{thm}
Here, we have 
omitted all degree shifts as well as potentially longer arrows
pointing to the right.
For the precise statement, see Theorem \ref{thm:coloredskein}.

\begin{remark}
We prove Theorem \ref{thm:coloredskein} essentially by showing that
$K(\MCSmin_{a,b})$ has a filtration whose subquotients
are homotopy equivalent to the the complexes associated to ``threaded digons''
as shown on the left-hand side of \eqref{eq:conv}. 
\end{remark}

Composing with a negative crossing on the left (say) yields the following consequence.

\begin{cor}\label{cor:intro skein}
Let $K'$ denote the Koszul complex
\[
K' := \Sym(\leftX_1|\rightX_2)\otimes \largewedge[\xi_1',\ldots,\xi_b']
\, , \quad \d(\xi_k') = \sum_{i+j=k} (-1)^j h_i(\leftX_1)e_j(\rightX_2),
\] 
then we have
\[ 
\left(
\left\llbracket
\begin{tikzpicture}[scale=.35,smallnodes,rotate=90,anchorbase]
\draw[very thick] (1,-1) to [out=150,in=210] (0,1); 
\draw[line width=5pt,color=white] (0,-2) to [out=90,in=270] (1,2);
\draw[very thick] (0,-2) node[right=-2pt]{$a$}to [out=90,in=270] (1,2) node[left=-2pt]{$a$};
\begin{scope} 
\clip (1,-.5) rectangle (0,.75) ;
\end{scope}
\draw[dotted] (1,-1) to [out=30,in=330] node[above,yshift=-2pt]{$0$} (0,1); 
\draw[very thick] (1,-2) node[right=-2pt]{$b$} to (1,-1); 
\draw[very thick] (0,1) to (0,2) node[left=-2pt]{$b$};
\end{tikzpicture}
\right\rrbracket
\to
\left\llbracket
\begin{tikzpicture}[scale=.35,smallnodes,rotate=90,anchorbase]
\draw[very thick] (1,-1) to [out=150,in=210] (0,1); 
\draw[line width=5pt,color=white] (0,-2) to [out=90,in=270] (1,2);
\draw[very thick] (0,-2) node[right=-2pt]{$a$}to [out=90,in=270] (1,2) node[left=-2pt]{$a$};
\begin{scope} 
\clip (1,-.5) rectangle (0,.75) ;
\draw[line width=5pt,color=white] (1,-1) to [out=30,in=330] (0,1); 
\end{scope}
\draw[very thick] (1,-1) to [out=30,in=330] node[above,yshift=-2pt]{$1$} (0,1); 
\draw[very thick] (1,-2) node[right=-2pt]{$b$} to (1,-1); 
\draw[very thick] (0,1) to (0,2) node[left=-2pt]{$b$};
\end{tikzpicture}
\right\rrbracket
\to
\cdots
\to
\left\llbracket
\begin{tikzpicture}[scale=.35,smallnodes,rotate=90,anchorbase]
\draw[dotted] (1,-1) to [out=150,in=210] (0,1); 
\draw[line width=5pt,color=white] (0,-2) to [out=90,in=270] (1,2);
\draw[very thick] (0,-2) node[right=-2pt]{$a$}to [out=90,in=270] (1,2) node[left=-2pt]{$a$};
\begin{scope} 
\clip (1,-.5) rectangle (0,.75) ;
\draw[line width=5pt,color=white] (1,-1) to [out=30,in=330] (0,1); 
\end{scope}
\draw[very thick] (1,-1) to [out=30,in=330] node[above,yshift=-2pt]{$b$} (0,1); 
\draw[very thick] (1,-2) node[right=-2pt]{$b$} to (1,-1); 
\draw[very thick] (0,1) to (0,2) node[left=-2pt]{$b$};
\end{tikzpicture}
\right\rrbracket
\right)
\simeq 
\left\llbracket\begin{tikzpicture}[smallnodes,rotate=90,anchorbase,scale=.525]
	\draw[very thick] (0,.25) to [out=150,in=270] (-.25,1) node[left=-2pt]{$b$};
	\draw[very thick] (.5,.5) to (.5,1) node[left=-2pt]{$a$};
	\draw[very thick] (0,.25) to node[right=-2pt]{$a{-}b$} (.5,.5);
	\draw[very thick] (.5,-1) node[right=-2pt]{$b$} to [out=90,in=330] (.5,.5);
	\draw[very thick] (-.25,-1)node[right=-2pt]{$a$} to [out=90,in=270] (0,.25);
\end{tikzpicture}\right\rrbracket \otimes_{\Sym(\leftX_1|\rightX_2)} K' \, .
\]
\end{cor}

In the course of proving Theorem \ref{thm:intro skein}, we obtain explicit
descriptions of the chain complexes involved above. Of particular interest, we
compute the complex assigned to a colored full twist braid on two strands and
identify it with the image of the Beliakova--Habiro categorical ribbon element
\cite{BH}. This verifies a version\footnote{The original statement concerns the
homotopy category of categorified quantum $\slnn{2}$; 
our results show that it holds in any integrable quotient thereof.} 
of \cite[Conjecture 1.3]{BH}; see Theorem~\ref{thm:Q0 is FT}.

\begin{exa}(1-colored case) By composing the skein relation
	\eqref{eq:uncoloredskein} with a positive crossing, we obtain the following
	homotopy equivalence:
	\begin{equation}\label{eq:linksplit pure}
	\Bigg(
	\left\llbracket
	\begin{tikzpicture}[scale=.5,smallnodes,anchorbase,rotate=90]
			\draw[very thick] (1,0) to [out=90,in=270] (0,1.5);
			\draw[line width=5pt,color=white] (1,-1.5) to [out=90,in=270] (0,0) 
			   to [out=90,in=270] (1,1.5);
			\draw[very thick] (1,-1.5) to [out=90,in=270] (0,0) 
			   to [out=90,in=270] (1,1.5);
			\draw[line width=5pt,color=white] (0,-1.5) to [out=90,in=270] (1,0);
			\draw[very thick] (0,-1.5) to [out=90,in=270] (1,0);
	\end{tikzpicture} 
	\right\rrbracket
	\longrightarrow
	\tdeg
	\left\llbracket
	\begin{tikzpicture}[scale=.5,smallnodes,rotate=90,anchorbase]
		\draw[very thick] (0,-2) to (0,0);
		\draw[very thick] (1,-2) to (1,0);
	\end{tikzpicture}
	\right\rrbracket
	\Bigg)
	\simeq 
	\Bigg(
		\qdeg 
	\left\llbracket
	\begin{tikzpicture}[rotate=90,scale=.5,smallnodes,anchorbase]
		\draw[very thick] (1,-1) to [out=90,in=270] (0,1);
		\draw[line width=5pt,color=white] (0,-1) to [out=90,in=270] (1,1);
		\draw[very thick] (0,-1) to [out=90,in=270] (1,1);
	\end{tikzpicture}
	\right\rrbracket
	\longrightarrow 
	\qdeg\inv \tdeg 
	\left\llbracket
	\begin{tikzpicture}[rotate=90,scale=.5,smallnodes,anchorbase]
		\draw[very thick] (1,-1) to [out=90,in=270] (0,1);
		\draw[line width=5pt,color=white] (0,-1) to [out=90,in=270] (1,1);
		\draw[very thick] (0,-1) to [out=90,in=270] (1,1);
	\end{tikzpicture}
	\right\rrbracket
	\Bigg) 
	\end{equation}
	in which the map on the right is multiplication by
	$h_1(\leftX_2-\rightX_2)$. 
	This is the special case of \eqref{eq:conv}
	corresponding to $a=b=1$.
	More explicitly, the right-hand side of
	\eqref{eq:linksplit pure} is a complex of the form
	\begin{equation}\label{eq:uncoloredKoszul}
	\begin{tikzcd}
		\qdeg\;
	\begin{tikzpicture}[smallnodes,rotate=90,anchorbase,scale=.5]
		\draw[very thick] (0,.375) to [out=150,in=270] (-.5,1);
		\draw[very thick] (0,.375) to [out=30,in=270] (.5,1);
		\draw[very thick] (0,-.375) to (0,.375);
		\draw[very thick] (.5,-1) to  [out=90,in=330] (0,-.375);
		\draw[very thick] (-.5,-1) to [out=90,in=210] (0,-.375);
	\end{tikzpicture}
	\ar[rr,"\chi_0^+"] \ar[dd, "x_2 - x_2' "]
	&&
	\tdeg\;
	\begin{tikzpicture}[scale=.5,smallnodes,rotate=90,anchorbase]
			\draw[very thick] (0,-2) to (0,0);
			\draw[very thick] (1,-2) to (1,0);
	\end{tikzpicture}
	\ar[dd, "x_2-x_2'(=0)"]
	 \\ && \\
	 \qdeg\inv \tdeg\;
	\begin{tikzpicture}[smallnodes,rotate=90,anchorbase,scale=.5]
		\draw[very thick] (0,.375) to [out=150,in=270] (-.5,1);
		\draw[very thick] (0,.375) to [out=30,in=270] (.5,1);
		\draw[very thick] (0,-.375) to (0,.375);
		\draw[very thick] (.5,-1) to  [out=90,in=330] (0,-.375);
		\draw[very thick] (-.5,-1) to [out=90,in=210] (0,-.375);
	\end{tikzpicture}
	\ar[rr, "\chi_0^+"]
	&&
	\qdeg^{-2} \tdeg^2\;
	\begin{tikzpicture}[scale=.5,smallnodes,rotate=90,anchorbase]
			\draw[very thick] (0,-2) to (0,0);
			\draw[very thick] (1,-2) to (1,0);
	\end{tikzpicture}
	\end{tikzcd} 
	\end{equation}
	On the other hand, 
	there is a well-known homotopy equivalence
	\begin{equation}\label{eq:ft2}
	\left\llbracket
	\begin{tikzpicture}[scale=.5,smallnodes,anchorbase,rotate=90]
			\draw[very thick] (1,0) to [out=90,in=270] (0,1.5);
			\draw[line width=5pt,color=white] (1,-1.5) to [out=90,in=270] (0,0) 
			   to [out=90,in=270] (1,1.5);
			\draw[very thick] (1,-1.5) to [out=90,in=270] (0,0) 
			   to [out=90,in=270] (1,1.5);
			\draw[line width=5pt,color=white] (0,-1.5) to [out=90,in=270] (1,0);
			\draw[very thick] (0,-1.5) to [out=90,in=270] (1,0);
	\end{tikzpicture} 
	\right\rrbracket
	\simeq
	\left(\
	\qdeg 
	\begin{tikzpicture}[smallnodes,rotate=90,anchorbase,scale=.5]
		\draw[very thick] (0,.375) to [out=150,in=270] (-.5,1);
		\draw[very thick] (0,.375) to [out=30,in=270] (.5,1);
		\draw[very thick] (0,-.375) to (0,.375);
		\draw[very thick] (.5,-1) to  [out=90,in=330] (0,-.375);
		\draw[very thick] (-.5,-1) to [out=90,in=210] (0,-.375);
	\end{tikzpicture}
	\xrightarrow{x_2 - x_2'}
	\qdeg\inv \tdeg 
	\begin{tikzpicture}[smallnodes,rotate=90,anchorbase,scale=.5]
		\draw[very thick] (0,.375) to [out=150,in=270] (-.5,1);
		\draw[very thick] (0,.375) to [out=30,in=270] (.5,1);
		\draw[very thick] (0,-.375) to (0,.375);
		\draw[very thick] (.5,-1) to  [out=90,in=330] (0,-.375);
		\draw[very thick] (-.5,-1) to [out=90,in=210] (0,-.375);
	\end{tikzpicture}
	\xrightarrow{\chi_0^+}
	\qdeg^{-2} \tdeg^2
	\begin{tikzpicture}[scale=.5,smallnodes,rotate=90,anchorbase]
			\draw[very thick] (0,-2) to (0,0);
			\draw[very thick] (1,-2) to (1,0);
	\end{tikzpicture}\ 
	\right),
	\end{equation}
	thus \eqref{eq:ft2} can be extracted as a quotient of \eqref{eq:uncoloredKoszul}.
	We show that this
	remarkable fact extends to arbitrary colors.
	\end{exa}

\begin{exa}(2-colored case)
	The Rickard complex for a crossing between two 2-colored strands has the form
	\begin{equation*}
		C_{2,2}:=\left \llbracket
\begin{tikzpicture}[rotate=90,scale=.5,smallnodes,anchorbase]
	\draw[very thick] (1,-1) node[right,xshift=-2pt]{$2$} to [out=90,in=270] (0,1);
	\draw[line width=5pt,color=white] (0,-1) to [out=90,in=270] (1,1);
	\draw[very thick] (0,-1) node[right,xshift=-2pt]{$2$} to [out=90,in=270] (1,1);
\end{tikzpicture}
		\right\rrbracket
		=\MCSmin_{2,2} =
		\left(
		\begin{tikzpicture}[smallnodes,rotate=90,anchorbase,scale=.66]
			\draw[very thick] (0,.25) to [out=150,in=270] (-.25,1) node[left,xshift=2pt]{$2$};
			\draw[very thick] (.5,.5) to (.5,1) node[left,xshift=2pt]{$2$};
			\draw[very thick] (0,.25) to node[left,xshift=2pt,yshift=-1pt]{$2$} (.5,.5);
			\draw[very thick] (0,-.25) to (0,.25);
			\draw[dotted] (.5,-.5) to [out=30,in=330] (.5,.5);
			\draw[very thick] (0,-.25) to node[right,xshift=-2pt,yshift=-1pt]{$2$} (.5,-.5);
			\draw[very thick] (.5,-1) node[right,xshift=-2pt]{$2$} to (.5,-.5);
			\draw[very thick] (-.25,-1)node[right,xshift=-2pt]{$2$} to [out=90,in=210] (0,-.25);
		\end{tikzpicture}
		\to  \qdeg\inv \tdeg
		\begin{tikzpicture}[smallnodes,rotate=90,anchorbase,scale=.66]
			\draw[very thick] (0,.25) to [out=150,in=270] (-.25,1) node[left,xshift=2pt]{$2$};
			\draw[very thick] (.5,.5) to (.5,1) node[left,xshift=2pt]{$2$};
			\draw[very thick] (0,.25) to node[left,xshift=2pt,yshift=-1pt]{$1$} (.5,.5);
			\draw[very thick] (0,-.25) to (0,.25);
			\draw[very thick] (.5,-.5) to [out=30,in=330] (.5,.5);
			\draw[very thick] (0,-.25) to node[right,xshift=-2pt,yshift=-1pt]{$1$} (.5,-.5);
			\draw[very thick] (.5,-1) node[right,xshift=-2pt]{$2$} to (.5,-.5);
			\draw[very thick] (-.25,-1)node[right,xshift=-2pt]{$2$} to [out=90,in=210] (0,-.25);
		\end{tikzpicture}
		\to  \qdeg^{-2} \tdeg^2
		\begin{tikzpicture}[smallnodes,rotate=90,anchorbase,scale=.66]
			\draw[very thick] (-.25,-1) node[right,xshift=-2pt]{$2$} to (-.25,1) node[left,xshift=2pt]{$2$};
			\draw[very thick] (.5,-1) node[right,xshift=-2pt]{$2$} to (.5,1) node[left,xshift=2pt]{$2$};
		\end{tikzpicture}
		\right) \, .
	\end{equation*}
We denote the webs appearing in this complex as $W_2$, $W_1$ and $W_0$ respectively. 
After basis change in the exterior algebras, 
the twisted complex on the right-hand side of \eqref{eq:conv} 
has the following schematic form: 
\begin{equation*}
	\begin{tikzpicture}[anchorbase]
		\node[scale=1] at (5,-2.5){$\MCSmin_{2,2} \otimes \wedge$
		};
		\draw[->] (5.75,-2.25) to (5.75,-1.25);
		\draw[->] (4,-2.5) to (3,-2.5);
		\node[scale=.75] at (-3.5,3.5){
			$\left\llbracket
\begin{tikzpicture}[scale=.5,smallnodes,anchorbase,rotate=90]
		\draw[very thick] (1,0) to [out=90,in=270] (0,1.5)node[left,xshift=2pt]{$2$};
		\draw[line width=5pt,color=white] (1,-1.5) to [out=90,in=270] (0,0) 
		   to [out=90,in=270] (1,1.5);
		\draw[very thick] (1,-1.5)node[right,xshift=-2pt]{$2$}  to [out=90,in=270] (0,0) 
		   to [out=90,in=270] (1,1.5)node[left,xshift=2pt]{$2$};
		\draw[line width=5pt,color=white] (0,-1.5) to [out=90,in=270] (1,0);
		\draw[very thick] (0,-1.5)node[right,xshift=-2pt]{$2$}  to [out=90,in=270] (1,0);
\end{tikzpicture} 
\right\rrbracket$
		};
		\draw[->,gray] (-2,3.5) to (-1.5,3.5);
		\node[scale=.75, blue] at (0,3.5){
			$\left\llbracket
\begin{tikzpicture}[scale=.375,smallnodes,rotate=90,anchorbase]
\draw[very thick] (1,-1) to [out=150,in=270] (0,0); 
\draw[line width=5pt,color=white] (0,-2) to [out=90,in=270] (.5,0) to [out=90,in=270] (0,2);
\draw[very thick] (0,-2) node[right=-2pt]{$2$}to [out=90,in=270] (.5,0) 
	to [out=90,in=270] (0,2) node[left=-2pt]{$2$};
\draw[line width=5pt,color=white] (0,0) to [out=90,in=210] (1,1); 
\draw[very thick] (0,0) to [out=90,in=210] (1,1); 
\draw[very thick] (1,-2) node[right=-2pt]{$2$} to (1,-1); 
\draw[very thick] (1,-1) to [out=30,in=330] (1,1); 
\draw[very thick] (1,1) to (1,2) node[left=-2pt]{$2$};
\end{tikzpicture}
\right\rrbracket$
		};
		\draw[->,gray] (1.5,3.5) to (2,3.5);
		\node[scale=.75, green] at (3.5,3.5){
			$\left\llbracket
\begin{tikzpicture}[scale=.5,smallnodes,anchorbase,rotate=90]
	\draw[very thick] (0,-1.5) node[right,xshift=-2pt]{$2$} to (0,1.5)node[left,xshift=2pt]{$2$};
	\draw[very thick] (1,-1.5) node[right,xshift=-2pt]{$2$} to (1,1.5)node[left,xshift=2pt]{$2$};
\end{tikzpicture} 
\right\rrbracket$
		};
		\draw[dotted] (1.75,4) to (1.75,3) to [out=270,in=180] (4,1.75) to (5,1.75);
		\draw[dotted] (-1.75,4) to (-1.75,3) to [out=270,in=180] (2,-.75) to (5,-.75);
		\node[scale=1] at (0,0.5){
	\begin{tikzcd}[row sep=2em,column sep=-1.5em]
		& 
		W_{2}\otimes \zeta^{(2)}_1\zeta^{(2)}_2	
		\arrow[ddl]
		\arrow[dr ]
		\arrow[rrr,gray,] & & & 
		\BLUE{W_{1}\otimes \zeta^{(1)}_1\zeta^{(1)}_2}
		\arrow[dr,blue] 
		\arrow[rrr,gray] & & & 
		\GREEN{W_{0}\otimes \zeta^{(0)}_1\zeta^{(0)}_2}	
		\arrow[from=dr,dotted,gray,dash] & 
		\\
		& &
		W_{2}\otimes \zeta^{(2)}_2	
		\arrow[ldd]
		\arrow[dr,] 
		\arrow[rrr,gray] & & & 
		\BLUE{W_{1}\otimes \zeta^{(1)}_2}	
		\arrow[from=ldd,dotted,gray,dash] 
		\arrow[dr,blue] 
		\arrow[rrr,blue]  & & & 
		\BLUE{W_{0}\otimes \zeta^{(0)}_2}	
		\arrow[from=ldd,dotted,gray,dash]
		\\
		W_{2}\otimes \zeta^{(2)}_1	
		\arrow[dr]
		\arrow[rrr,crossing over] 
		& & & 
		W_{1}\otimes \zeta^{(1)}_1	
		\arrow[dr]
		\arrow[rrr,crossing over,gray] 
		\arrow[uur,dotted,gray,dash] 
		& & & 
		\BLUE{W_{0}\otimes \zeta^{(0)}_1}	
		\arrow[from=dr,dotted,gray,dash]
		\arrow[uur,dotted,gray,dash]
		 & & 
		\\
		&
		 W_{2}\otimes 1	 \arrow[rrr]& & & 
		 W_{1}\otimes 1	 \arrow[rrr]& & & 
		 W_{0}\otimes 1	 &
	\end{tikzcd}
		};
	\end{tikzpicture}
\end{equation*}
	The subquotients with respect to the filtration indicated by the dotted lines
	are homotopy equivalent to the complexes on the left-hand side of
	\eqref{eq:conv}. Additional details appear in Example~\ref{exa:KMCS}.
	\end{exa}

\begin{remark}
In this paper we focus on the objects associated to braids, 
and not closed link diagrams.
Paralleling the uncolored case, 
one obtains colored Khovanov-Rozansky homology by 
taking Hochschild (co)homology of the complex $\llbracket \beta \rrbracket$ 
assigned to a colored braid $\beta$, and then taking homology.
As such, our results have implications for (colored) Khovanov-Rozansky homology, 
but we do not explore them here.
However, in the companion paper \cite{HRW2}, 
we use curved deformations of Theorem~\ref{thm:coloredskein} 
to explore \emph{colored link splitting phenomena}.
Indeed, the results in this paper grew from the considerations in \cite{HRW2}.
We believe they are of general interest/utility, 
so we have independently packaged them here.
\end{remark}

\begin{remark}
	An expression for complexes associated to colored full twist braids on two
	strands, similar to the one implicit in Theorem~\ref{thm:intro skein}, was
	obtained in \cite[Section 4]{Wed1} and described in terms of certain winding
	diagrams inspired by Heegaard--Floer theory. It would be interesting to find
	an interpretation of the entire colored skein relation from Theorem
	\ref{thm:intro skein} in terms of suitable Fukaya categories (depending on
	the colors) associated with the 4-punctured sphere. 
\end{remark}

\begin{conv}
Throughout, we work over the field $\K$ of rational numbers 
for simplicity (e.g. in treating the background on symmetric functions);
however, our results hold over an arbitrary field. 
We further expect our results to hold over the integers, 
but certain statements 
(e.g. Lemma \ref{lem:Cautis+} and Proposition \ref{prop:UniqueYCrossing}) 
will require additional arguments in this setting.
\end{conv}

\subsection*{Acknowledgements}
This project was conceived during the conference 
``Categorification and Higher Representation Theory'' at the Institute Mittag-Leffler, 
and began in earnest during the workshop ``Categorified Hecke algebras, link
homology, and Hilbert schemes'' at the American Institute for Mathematics. We
thank the organizers and hosts for a productive working atmosphere. We would
also thank Eugene Gorsky and Lev Rozansky for many useful discussions.

\subsection*{Funding}

M.H.  was supported by NSF grant DMS-2034516.  D.R. and P.W. were supported in part by the National Science Foundation under
Grant No. NSF PHY-1748958 during a visit to the program ``Quantum Knot
Invariants and Supersymmetric Gauge Theories'' at the Kavli Institute for
Theoretical Physics. 
D.R. was partially supported by Simons Collaboration Grant 523992: 
``Research on knot invariants, representation theory, and categorification.''
P.W. was partially supported by the Australian Research
Council grants `Braid groups and higher representation theory' DP140103821 and
`Low dimensional categories' DP160103479 while at the Australian National
University during early stages of this project. P.W. was also supported by the
National Science Foundation under Grant No. DMS-1440140, while in residence at
the Mathematical Sciences Research Institute in Berkeley, California, during the
Spring 2020 semester.

\section{Webs, bimodules, and categorified quantum \texorpdfstring{$\glm$}{gl(m)}}
\label{s:webs etc}

In this section, we review background on singular Soergel bimodules and Rickard
complexes.

\subsection{Symmetric functions}
\label{ss:sym fns}

We begin with some preliminaries on symmetric functions, which play a
substantial role throughout.

\begin{defi} If $\X=\{x_1,\ldots,x_N\}$ is a finite alphabet with $N$ letters,
we let $\Sym(\X) = \K[\X]^{\mathfrak{S}_N}$ denote the ring of symmetric polynomials. 
The \emph{elementary} symmetric polynomials $e_j(\X)$, 
\emph{complete} symmetric polynomials $h_j(\X)$, 
and \emph{power sum} symmetric polynomials $p_j(\X)$
are each defined via their generating functions as follows:
\[
\begin{aligned}
E(\X,t) &:= \prod_{x\in \X} (1+x t) =: \sum_{j \geq 0} e_j(\X) t^j \\
H(\X,t) &:= \prod_{x\in \X} (1-x t)\inv =: \sum_{j\geq 0} h_j(\X) t^j \\
P(\X,t) & := \sum_{x \in \X} \frac{x t}{1-x t} =: \sum_{j\geq 1} p_j(\X) t^j \, .
\end{aligned}
\]
By convention, $e_0(\X) = h_0(\X) = 1$ and $p_0(\X)$ is undefined.
For pairwise disjoint alphabets $\X_1, \dots, \X_r$, we write
\[\Sym(\X_1|\cdots|\X_r) \cong \Sym(\X_1)\otimes \cdots \otimes \Sym(\X_r)\] for
the ring of polynomials in $\X_1\cup\cdots\cup\X_r$ that are separately
symmetric in the alphabets $\X_i$. 
\end{defi} 

The elementary and complete symmetric polynomials are related by the identity
\begin{equation}\label{eq:HE}
H(\X,t)E(\X,-t) = 1
\, , \quad \text{i.e.~}\quad
\sum_{i+j=k} (-1)^j h_i(\X) e_j(\X) = 0 \quad \forall k \geq 1\, ,
\end{equation}
and each are related to the power sum symmetric polynomials by the 
Newton identity:
\begin{equation}\label{eq:Newton}
\frac{t\frac{d}{dt}H(\X,t)}{H(\X,t)} = P(\X,t)
\, , \quad \text{i.e.~}\quad
H(\X,t) = \mathrm{exp}  \int P(\X,t) \frac{dt}{t} \, .
\end{equation}
We will establish identities involving symmetric polynomials
via the manipulation of generating functions.
For example, for disjoint alphabets $\X$ and $\X'$, the identity
\[
\sum_{i+j=k}(-1)^j h_i(\X\cup \X')e_j(\X) = h_k(\X')
\]
follows from the generating function identity
\[
H(\X \cup \X',t) E(\X,-t) = \frac{H(\X,t) H(\X',t)}{H(\X,t)} = H(\X',t) \, .
\]
In the following, when the parameter $t$ is understood, 
we shall omit it from the notation.
\smallskip

Let us now consider an alphabet $\X^N=\{x_1,\ldots,x_N\}$ on $N$ letters.
There is a map of graded algebras $\Sym(\X^{N+1})\rightarrow \Sym(\X^N)$
sending $x_{N+1}\mapsto 0$. 
By definition, the \emph{ring of symmetric functions} 
in infinitely many variables $\X^\infty = \{x_1,x_2,\ldots\}$ is 
the inverse limit
\[
\Sym(\X^\infty) := \lim_{\longleftarrow} \Sym(\X^N) \, .
\]
The symmetric functions $e_i(\X^N),h_i(\X^N),p_i(\X^N) \in \Sym(\X^N)$ 
are stable with respect
to the projections $\Sym(\X^N)\rightarrow \Sym(\X^{N-1})$, 
hence determine well-defined elements of $\Sym(\X^\infty)$. 
When we do not wish to commit ourselves to a particular alphabet, 
we will utilize the following notation.

\begin{definition}
Let $\Lambda$ denote the ring $\Sym(\X^\infty)$ of symmetric functions. 
The elementary, complete, and power sum symmetric functions 
$e_k(\X^\infty), h_k(\X^\infty)$, and $p_k(\X^\infty)$
are denoted as $e_k, h_k, p_k \in \Lambda$, respectively.
As an algebra, we have
$\Lambda \cong \K[e_1,e_2,\ldots ] \cong \K[h_1,h_2,\ldots] \cong \K[p_1,p_2\ldots]$. 
\end{definition}

Our considerations necessitate working with unions of disjoint alphabets, 
as well as differences of alphabets. 
These operations can be placed on equal footing by considering 
\emph{formal linear combinations of alphabets}.

\begin{definition}\label{def:LinCombAlph}
Let $\X_1,\ldots,\X_r$ be alphabets and let $a_1,\ldots,a_r\in \K$ be scalars.
For $f \in \Lambda$, define
\[
f(a_1\X_1+\cdots+a_r\X_r)\in \Sym(\X_1)\otimes\cdots\otimes \Sym(\X_r)
\]
as follows.
If $f=p_k$ is a power sum symmetric function, then set
\[
p_k(a_1\X_1+\cdots+a_r\X_r)=a_1p_k(\X_1)+\cdots+a_rp_k(\X_r) \, .
\]
This extends to all of $\Lambda$ by linearity:
\[
(f+g)(a_1\X_1+\cdots+a_r\X_r) = f(a_1\X_1+\cdots+a_r\X_r) + g(a_1\X_1+\cdots+a_r\X_r)
\]
and multiplicativity:
\[
(fg)(a_1\X_1+\cdots+a_r\X_r) = f(a_1\X_1+\cdots+a_r\X_r)g(a_1\X_1+\cdots+a_r\X_r) \, .
\]
\end{definition}

If $\X_1$ and $\X_2$ are disjoint alphabets, then
$p_k(\X_1+\X_2) = p_k(\X_1)+p_k(\X_2)= p_k(\X_1\cup \X_2)$,
thus Definition \ref{def:LinCombAlph} implies that 
\[
f(\X_1+\X_2) = f(\X_1\cup \X_2)
\]
for every symmetric functions $f$.
Similarly, formal subtraction of alphabets behaves as expected:
if $\X_1,\X_2$ are alphabets and $\X_0$ is disjoint from both, then
\[
f\Big((\X_1\cup \X_0) - (\X_2\cup \X_0)\Big) = f(\X_1 - \X_2)
\]
Again, this identity need only be checked in the special case that $f=p_k$
and there it is immediate.

Next, we evaluate elementary and complete symmetric functions 
on formal linear combinations of alphabets. 
For a power series $F(t)\in A[[t]]$ with coefficients in 
a $\K$-algebra $A$ and $a\in \K$, we write
\[
F(t)^a := \exp( a\ln (F(t))),
\]
where $\exp$ and $\ln$ are the obvious operators acting on power series.

\begin{lem}
	\label{lem:PHEonsum}
On the level of generating functions, we have
\begin{align*}
P(a_1\X_1+a_2\X_2,t) &= a_1P(\X_1,t)+ a_2 P(\X_2,t),\\
H(a_1\X_1+a_2\X_2,t) &= H(\X_1,t)^{a_1} H(\X_2,t)^{a_2},\\
E(a_1\X_1+a_2\X_2,t) &= E(\X_1,t)^{a_1} E(\X_2,t)^{a_2}
\end{align*}
for all $a_1,a_2\in \K$.
\end{lem}

\begin{proof}
The statement for $P(\X,t)$ is immediate from Definition \ref{def:LinCombAlph}.
The remaining statements follow via equation \eqref{eq:Newton}.
For example,
\begin{align*}
H(a_1\X_1+a_2\X_2,t)
&= \mathrm{exp}\int P(a_1\X_1+a_2\X_2,t) \frac{dt}{t}\\
&= \mathrm{exp}\int (a_1P(\X_1,t)+a_2P(\X_2,t))\frac{dt}{t}\\
&= \mathrm{exp}\left(a_1\int P(\X_1,t)\frac{dt}{t}\right)\mathrm{exp}\left( a_2\int P(\X_2,t)\frac{dt}{t}\right)\\
&= \mathrm{exp}\Big( a_1 \ln(H(\X_1,t))\Big)\mathrm{exp}\Big( a_2 \ln(H(\X_2,t))\Big)\\
&= H(\X_1,t)^{a_1} H(\X_2,t)^{a_2} \, . 
\qedhere
\end{align*}
\end{proof}


It follows that this notational convention for formal 
addition and subtraction of alphabets is consistent with that in \cite{Las}.
Useful special cases of Lemma~\ref{lem:PHEonsum} include
\[
H(-\X,t) = H(\X,t)\inv = E(\X,-t) \, ,
\]
and
\[
\begin{aligned}
H(\X_1+\X_2) = H(\X_1)H(\X_2) 
\, &, \quad
H(\X_1-\X_2) = \frac{H(\X_1)}{H(\X_2)} \\
E(\X_1+\X_2) = E(\X_1)E(\X_2) 
\, &, \quad
E(\X_1-\X_2) = \frac{E(\X_1)}{E(\X_2)}
\end{aligned}
\]
(in the latter we we have omitted the parameter $t$).
In particular, this gives the following generalization of \eqref{eq:HE}:
\begin{equation}\label{eq:HE2}
h_r(\X_1 - \X_2) = \sum_{j=0}^r (-1)^{j} h_{r-j}(\X_1) e_{j}(\X_2)
\end{equation}

We will need an alternative formulation of this identity, 
in which the lower index of summation starts at $j=1$.


\begin{lem}\label{lemma:somerelations1}
	Let $\X,\X'$ be alphabets, then we have the following identities for all $r\geq 1$:
	\begin{align*}
	e_r(\X) - e_r(\X') &= \sum_{j=1}^r(-1)^{j-1} e_{r-j}(\X) h_j(\X-\X')\, , \\
	h_r(\X-\X') &= \sum_{j=1}^r(-1)^{j-1} h_{r-j}(\X) \Big(e_j(\X)-e_j(\X')\Big)\, .
	\end{align*}
	\end{lem}
\begin{proof}
	This follows immediately from the generating function identities
	\[
	E(\X,t) - E(\X',t) = - E(\X,t) \left(\frac{H(\X,-t)}{H(\X',-t)} - 1\right)
	\]
	and
	\[
	\left(\frac{H(\X,t)}{H(\X',t)} - 1\right) = -H(\X,t)\Big(E(\X,-t) - E(\X',-t)\Big). \qedhere
	\]
	\end{proof}

\begin{rem}
The ring of symmetric functions is a Hopf algebra.  
The antipode 
corresponds to the substitution of alphabets 
$\X\mapsto - \X$, which is to say that
\[
(Sf)(\X) = f(-\X)\in \Sym(\X) \, .
\]
The comultiplication corresponds to the substitution 
$\X\mapsto \X_1+\X_2$,
i.e.
\[
\sum f^{(1)}(\X_1) f^{(2)}(\X_2) = f(\X_1+\X_2) \in \Sym(\X_1 | \X_2) \cong \Sym(\X_1)\otimes \Sym(\X_2)
\]
where we have used the Sweedler notation 
$\Delta(f) = \sum f^{(1)}\otimes f^{(2)}\in \Lambda\otimes \Lambda$.
\end{rem}

\subsection{Singular Soergel bimodules and webs}
\label{ss:ssbim}

Recall from the introduction that a categorification of the Hecke algebroid 
(and the natural setting for colored, triply-graded link homology) is the monoidal
2-category of type $A$ singular Soergel bimodules. Fix $N>0$, and let 
$R :=\K[x_1,\ldots,x_N]$ be the polynomial ring in variables $x_i$, 
graded by declaring $\deg(x_i)=2$. Given a parabolic subgroup 
$J_{\aa} = \mathfrak{S}_{a_1}\times \cdots \times \mathfrak{S}_{a_m}$ 
of the symmetric group $\mathfrak{S}_N$, 
we let $R^{\aa} \subseteq R$ denote the ring of
polynomials invariant under the action of $J_{\aa}$. 
Note that $R^{\bb}\subset R^{\aa}$ if and only if 
$J_{\bb}\supset J_{\aa}$.

Consider the 2-category $\mathrm{Bim}_N$ given as follows:
\begin{itemize}
	\item Objects are tuples $\aa = (a_1,\ldots,a_m)$ 
		with $a_i \geq 1$ and $\sum_{i=1}^m a_i = N$.
	\item $1$-morphisms $\aa \to \bb$ are graded
		$(R^{\bb},R^{\aa})$-bimodules.
	\item $2$-morphisms are homomorphisms of graded bimodules.
\end{itemize}
Horizontal composition is given by tensor product over the rings
$R^{\aa}$, and will be denoted by $\hComp$. 
Vertical composition is the usual composition of bimodule homomorphisms. 
We will write $\oone_{\aa} :=R^{\aa}$ for the identity bimodule, 
saving the notation $R^{\aa}$ for the rings themselves.

A \emph{singular Bott-Samelson bimodule} is, 
by definition, 
any $(R^{\aa_0},R^{\aa_r})$-bimodule of the form
\[
B = R^{\aa_0}\otimes_{R^{\bb_1}} R^{\aa_1} \otimes_{R^{\bb_2}}
\cdots  \otimes_{R^{\bb_r}} R^{\aa_r}
\]
for some sequence of rings and subrings $R^{\aa_0}\supset R^{\bb_1}\subset
\cdots \supset R^{\bb_r} \subset R^{\aa_r}$, 
or a grading shift thereof.
In particular, 
whenever $R^{\bb}\subset R^{\aa}$ (equivalently $J_{\bb}\supset J_{\aa}$),
we have the \emph{merge} and \emph{split} bimodules 
(terminology explained below) given by
\begin{equation}\label{eq:mergeandsplit}
{}_{\bb}M_{\aa} := 
\qdeg^{\ell(\aa) - \ell(\bb)}
{_{R^{\bb}}\hskip-1pt|R^{\aa}} \cong 
\qdeg^{\ell(\aa) - \ell(\bb)}
R^{\bb}  \otimes_{R^{\bb}} R^{\aa}
\, , \quad
{}_{\aa}S_{\bb}  := R^{\aa}|_{R^{\bb}} \cong R^{\aa}  \otimes_{R^{\bb}} R^{\bb} \, .
\end{equation}
Here, $\qdeg^k$ denotes a shift up in degree by $k$, 
and $\ell(\aa)$ denotes the length 
of the longest element in $J_{\aa}$.

\begin{defi}
The 2-category $\SSBim_N$ of singular Soergel bimodules 
is the smallest full $2$-subcategory of $\mathrm{Bim}_N$ 
containing the singular Bott-Samelson bimodules that is closed under 
taking shifts, direct sums, and direct summands. 
We denote the $\Hom$-category from $\aa \to \bb$ by ${}_{\bb}\SSBim_{\aa}$.
\end{defi}

There is an external tensor product
$\boxtimes\colon  \SSBim_{N_1} \times \SSBim_{N_2} \to \SSBim_{N_1+N_2}$ 
given on objects by concatenation of tuples:
\[
(a_1,\ldots,a_{m_1})\boxtimes (b_1,\ldots,b_{m_2}) \ := \ (a_1,\ldots,a_{m_1}, b_1,\ldots,b_{m_2})
\]
and on $1$- and $2$-morphisms by tensor product over $\K$. 
This implies that the collection $\{\SSBim_N \}_{N \geq 0}$ 
assemble to form a monoidal 2-category, that we denote $\SSBim$.

There are a number of combinatorial/diagrammatic 
models for the 2-category generated by the singular Bott-Samelson modules, 
\eg (the singular analogue of) Elias-Williamson's graphical calculus for Soergel bimodules
\cite{EW,ESW}, or the $k \to \infty$
(inverse) limit of the $\slnn{k}$ foam 2-category \cite{QR};
see e.g. \cite[Section 5.2]{QRS} and \cite[Proposition 3.4]{Wed3}. 
(This $k$ is independent/unrelated to $N$.) 
We will use aspects of the latter, 
as the graphical description of the $1$-morphisms therein is 
directly related to braid and link diagrams.

To wit, in this description, 
singular Bott-Samelson bimodules are denoted using MOY webs, 
certain labeled, trivalent graphs,  
\eg for $\aa = (a,b)$ and $\aa' = (a+b)$, we have 
\begin{equation}\label{eq:GenWeb}
	_{\aa'}M_{\aa} = 
	\begin{tikzpicture}[scale =.75, smallnodes, rotate=90,anchorbase]
		\draw[very thick] (0,0) node[right,xshift=-2pt]{$a$} to [out=90,in=210] (.5,.75);
		\draw[very thick] (1,0) node[right,xshift=-2pt]{$b$} to [out=90,in=330] (.5,.75);
		\draw[very thick] (.5,.75) to (.5,1.5) node[left,xshift=2pt]{$a{+}b$};
	\end{tikzpicture}
	\;\; \text{ and } \;\;
	_{\aa}S_{\aa'} =
	\begin{tikzpicture}[scale =.75, smallnodes,rotate=270,anchorbase]
		\draw[very thick] (0,0) node[left,xshift=2pt]{$b$} to [out=90,in=210] (.5,.75);
		\draw[very thick] (1,0) node[left,xshift=2pt]{$a$} to [out=90,in=330] (.5,.75);
		\draw[very thick] (.5,.75) to (.5,1.5) node[right,xshift=-2pt]{$a{+}b$};
	\end{tikzpicture} \, .
	\end{equation}
All other singular Bott-Samelson bimodules can be obtained from these using
direct sum and grading shift, together with the horizontal composition $\hComp$
and tensor product $\boxtimes$.  Graphically, $\hComp$ corresponds to to glueing
of diagrams along a common boundary and $\boxtimes$ corresponds to disjoint
union of diagrams, as depicted in the following.
\begin{exa}
For
$_{\aa'}M_{\aa}$ and $_{\aa}S_{\aa'}$ as in
\eqref{eq:GenWeb}, we have:
\[
{}_{\aa'}M_{\aa} \hComp {}_{\aa}S_{\aa'} = 
\begin{tikzpicture}[scale =.75, smallnodes, rotate=90,anchorbase]
	\draw[very thick] (.5,-.75) to [out=150,in=270] (0,0) node[above,yshift=-2pt]
		{$a$} to [out=90,in=210] (.5,.75);
	\draw[very thick] (.5,-.75) to [out=30,in=270] (1,0) node[below,yshift=1pt]{$b$} 
		to [out=90,in=330] (.5,.75);
	\draw[very thick] (.5,.75) to (.5,1.5) node[left,xshift=2pt]{$a{+}b$};
	\draw[very thick] (.5,-.75) to (.5,-1.5) node[right,xshift=-2pt]{$a{+}b$};
\end{tikzpicture}
\, , \quad
{}_{\aa'}M_{\aa} \boxtimes {}_{\aa}S_{\aa'} = 
\begin{tikzpicture}[scale =.75, smallnodes,rotate=270,anchorbase]
	\draw[very thick] (2.5,1.5) node[right,xshift=-2pt]{$a$} to [out=270,in=30] (2,.5);
	\draw[very thick] (1.5,1.5) node[right,xshift=-2pt]{$b$} to [out=270,in=150] (2,.5);
	\draw[very thick] (2,.5) to (2,0) node[left,xshift=2pt]{$a{+}b$};
	\draw[very thick] (0,0) node[left,xshift=2pt]{$b$} to [out=90,in=210] (.5,1);
	\draw[very thick] (1,0) node[left,xshift=2pt]{$a$} to [out=90,in=330] (.5,1);
	\draw[very thick] (.5,1) to (.5,1.5) node[right,xshift=-2pt]{$a{+}b$};
\end{tikzpicture} \, .
\]
\end{exa}
For the duration, we will refer to the graphs built from the diagrams in 
\eqref{eq:GenWeb} via $\hComp$ and $\boxtimes$ as \emph{webs}, 
which we always understand\footnote{Strictly speaking,
web edges should be equipped with an orientation. 
In this paper, we only consider webs with edges that are oriented
towards the left, so we omit orientation arrows from all figures.}
as mapping from the labels at their right 
endpoints to those at their left.

Let $\cal{W}$ be a web and let $B(\cal{W})$ be the associated 
singular Bott-Samelson bimodule. 
We now given an alternate description of $B(\cal{W})$, following \cite{Ras2}.
For each edge $e$ of $\cal{W}$, choose an alphabet $\X_e$ 
of cardinality equal to the label on the edge
and define the
\emph{edge ring} associated to $\cal{W}$:
\[
R(\cal{W}) := \bigotimes_{e\in \mathrm{Edges}(\cal{W})} \Sym(\X_e) \, .
\]
For each symmetric function $f$, expressions such as 
$f(\X_e)$ and $f(\X_{e_1}+\X_{e_2}-\X_{e_3})$ 
represent well-defined elements of $R(\cal{W})$.
An edge $e$ of $\cal{W}$ is called an \emph{exterior edge} 
if $e$ meets the boundary $\partial \cal{W}$.  
More specifically, if $e$ meets the left boundary we call it \emph{outgoing}, 
and if it meets the right we call it \emph{incoming}. 
We define the \emph{outgoing (respectively incoming) edge rings} by
\[
R^{\text{out}}(\cal{W}) := \bigotimes_{e \text{ is outgoing}} \Sym(\X_e)
\, , \quad
R^{\text{in}}(\cal{W}) := \bigotimes_{e \text{ is incoming}} \Sym(\X_e) \, .
\]
The following is immediate.
\begin{lem}\label{lem:edge}
Up to the shifts coming from \eqref{eq:mergeandsplit}, 
there is an isomorphism
\[
B(\cal{W}) \cong R(\cal{W})/I(\cal{W})
\]
of $\big(R^{\text{out}}(\cal{W}),R^{\text{in}}(\cal{W}) \big)$-bimodules,
where $I(\cal{W})\subset R(\cal{W})$ is the ideal generated by 
all elements of the form $f(\X_{e_1}+\X_{e_2}-\X_{e_3})$, 
where $f \in \Lambda$ and
$e_1, e_2, e_3$ are edges of $\cal{W}$ 
that meet at a trivalent vertex as in:
\[
\begin{tikzpicture}[scale =.75,rotate=90,smallnodes,anchorbase]
	\draw[very thick] (0,0) node[right,xshift=-2pt]{$a$} to [out=90,in=210] 
		 node[below]{$\X_{e_1}$} (.5,.75);
	\draw[very thick] (1,0) node[right,xshift=-2pt]{$b$} to [out=90,in=330]
		 node[above,yshift=1pt]{$\X_{e_2}$} (.5,.75);
	\draw[very thick] (.5,.75) to node[above]{$\X_{e_3}$} 
		(.5,1.5) node[left,xshift=2pt]{$a{+}b$};
\end{tikzpicture} 
\quad 
\text{or}
\quad
\begin{tikzpicture}[scale =.75,rotate=270,smallnodes,anchorbase]
	\draw[very thick] (0,0) node[left,xshift=2pt]{$b$} to [out=90,in=210] 
		 node[above]{$\X_{e_2}$} (.5,.75);
	\draw[very thick] (1,0) node[left,xshift=2pt]{$a$} to [out=90,in=330]
		 node[below,xshift=2pt,yshift=-1pt]{$\X_{e_1}$} (.5,.75);
	\draw[very thick] (.5,.75) to node[below,xshift=1pt]{$\X_{e_3}$} 
		(.5,1.5) node[right,xshift=-2pt]{$a{+}b$};
\end{tikzpicture}
\] \qed
\end{lem}

Despite this result, 
it is at times helpful to distinguish the bimodule $B(\cal{W})$
from the ring $R(\cal{W})/I(\cal{W})$.
Our primary use for the latter will be in specifying 
bimodule endomorphisms of $B(\cal{W})$.
Indeed, in the web-and-foam formalism for $\SSBim$, 
morphisms between singular Bott-Samelson bimodules $B(\cal{W})$ 
are described by (linear combinations of) foams, 
certain 2-dimensional CW complexes with facets labeled
by non-negative integers that are embedded in $[0,1]^3$
and carry decorations by symmetric polynomials on their facets. 
Such foams should be viewed as embedded singular cobordisms with
corners between the domain and codomain webs. 
In particular, elements of $R(\cal{W})/I(\cal{W})$ correspond to the 
singular cobordism $\cal{W} \times [0,1]$, with facets appropriately decorated.

However, almost all of the morphisms between singular Bott-Samelson bimodules needed for the
present work fall into two classes:
\begin{enumerate}
	\item endomorphisms of $B(\cal{W})$ given by multiplication by 
	elements in $R(\cal{W})/I(\cal{W})$, or
	\item those in the image of a $2$-functor from categorified quantum $\glm$
	(see \S \ref{sec:CQG}).
\end{enumerate}
As such, we will rarely use the language of foams, 
but see Appendix \ref{s:FoamSSBim} for a short dictionary.

\begin{conv}\label{conv:Alph}
In many places in the present work, we will consider endomorphisms of
Bott-Samelson bimodules corresponding to webs appearing in equation
\eqref{eq:AlphabetConv} below, for various edge labels. 
As shorthand, we assign alphabets of variables to each web edge 
with cardinality equal to the label on the edge as follows:
\begin{equation}\label{eq:AlphabetConv}
\begin{tikzpicture}[rotate=90,anchorbase]
	\draw[very thick] (0,.25) to [out=150,in=270] (-.25,1) node[left,xshift=2pt]{$\leftX_1$};
	\draw[very thick] (.5,.5) to (.5,1) node[left,xshift=2pt]{$\leftX_2$};
	\draw[very thick] (0,.25) to node[left,yshift=-1pt,xshift=1pt]{$\leftM$} (.5,.5);
	\draw[very thick] (0,-.25) to node[below,yshift=2pt]{$\Fr$} (0,.25);
	\draw[very thick] (.5,-.5) to [out=30,in=330] node[above,yshift=-2pt]{$\B$} (.5,.5);
	\draw[very thick] (0,-.25) to node[right,yshift=-1pt,xshift=-1pt]{$\rightM$} (.5,-.5);
	\draw[very thick] (.5,-1) node[right,xshift=-2pt]{$\rightX_2$} to (.5,-.5);
	\draw[very thick] (-.25,-1)node[right,xshift=-2pt]{$\rightX_1$} to [out=90,in=210] (0,-.25);
\end{tikzpicture}
\end{equation}
\end{conv} 

\begin{exa} For the web $\cal{W}$ from Convention~\ref{conv:Alph}, we have
\[
R(\cal{W}):= \Sym(\X_1|\X_2|\M |\Fr |\B| \M'|\X_1',\X_2')
\]
and $I(\cal{W})$ is the ideal generated by elements of the form
\[
f(\X_2 - \B - \M) \, , \quad f(\X_1+\M-\Fr) \, ,\quad f(\B+\M'-\X_2') \, ,\quad f(\Fr-\M'-\X_1'),
\]
or equivalently
\[
f(\X_2)-f(\B+\M) \, ,\quad f(\X_1+\M)-f(\Fr) \, ,\quad f(\B+\M')-f(\X_2') \, , \quad f(\Fr)-f(\M'+\X_1'),
\]
as $f$ ranges over all symmetric functions.
\end{exa}

\begin{rem}
\label{rem:symm} 
For every 1-morphism ${}_{\bb} M_{\aa}$ in $\SSBim_N$, we
have embeddings $R^{\mathfrak{S}_N} \hookrightarrow R^{\aa}$ and 
$R^{\mathfrak{S}_N} \hookrightarrow R^{\bb}$ and the 
endomorphisms of ${}_{\bb} M_{\aa}$ induced by
$f\in R^{\mathfrak{S}_N}$ on the left and on the right agree.
\end{rem}

\subsection{The dg category of complexes}
\label{ss:cxs}
In order to consider the braid group representation on singular Soergel bimodules, 
we must first discuss the monoidal dg 2-category of complexes of 
singular Soergel bimodules.
We being by recalling the basic framework of dg categories of complexes.

\begin{defi}\label{def:dgCatCom}
Let $\AS$ be a $\K$-linear category, 
then $\CS(\AS)$ denotes the dg category of bounded complexes over $\AS$.
Objects of this category are complexes
\[
\left(X, \delta_X \right) = 
\cdots \xrightarrow{\delta_X} X^k \xrightarrow{\delta_X} 
X^{k+1} \xrightarrow{\delta_X} \cdots
\]
in $\AS$ with $X^k=0$ for all but finitely many $k$. 
Morphism spaces in this category are complexes 
$(\Hom_{\CS(\AS)}(X,Y),d)$ where
\[
\Hom_{\CS(\AS)}^k(X,Y) := \prod_{i\in \Z}\Hom_{\AS}(X^i,Y^{i+k}) 
\]
and the component of the differential 
$d: \Hom_{\CS(\AS)}^k(X,Y) \rightarrow \Hom_{\CS(\AS)}^{k+1}(X,Y)$
is given by
\[
d(f) := [\d,f] = \d_Y\circ f - (-1)^{|f|} f\circ \d_X \, .
\]
\end{defi}

The notation $|f|=k$ means that $f$ is homogeneous of 
(homological) degree $k$, 
\ie that $f \in \Hom_{\CS(\AS)}^k(X,Y)$.
We say that such $f$ is \emph{closed} if $[\d,f]=0$ and 
\emph{exact} (or \emph{null-homotopic}) if $f=[\d,h]$ for some 
$h \in \Hom_{\CS(\AS)}^{k-1}(X,Y)$. 
The category $\CS(\AS)$ is endowed with an autoequivalence 
(homological) shift functor, that we denote by $\tdeg$. 
By convention, $\tdeg^k$ denotes a shift up in homological degree.

We will use the following to build certain complexes 
(in particular, to construct the left-hand side of the colored skein relation).

\begin{defi}\label{def:twist}
If $(X,\d_X)$ is a complex and 
$\a \in \End^1_{\CS(\AS)}(X)$ satisfies $(\d_X+\a)^2=0$, 
then we denote the complex $(X,\d_X+\a)$ by $\tw_\a(X)$. 
We will refer to $\tw_\a(X)$ as a \emph{twist} of the complex $(X,\d_X)$.
Further, we call $\tw_\a(X)$ a \emph{one-sided twisted complex},
if $X$ takes the form 
\[
(X,\delta) = \bigoplus_{i \in \Z} (X_i , \d_i)
\]
where the components $\alpha_{i,j} \colon X_j \to X_i$ of $\alpha$ 
satisfy $\alpha_{i,j} = 0$ for $i \leq j$.
\end{defi}

Note that any complex $(X,\d_X)$ can itself be written as 
a one-sided twisted complex
\[
X = \tw_{\d_X}\Big(\bigoplus_k \tdeg^k X^k\Big)
\]
where we view each $X^k$ as a complex concentrated in homological degree zero 
with differential.

\begin{remark}\label{rem:enrich}
If $\AS$ is enriched in a symmetric monoidal category $\KS$, then $\CS(\AS)$ is
enriched in the category of complexes $\CS(\KS)$. In particular, if
$\Hom$-spaces in $\AS$ are (already) $\Z$-graded $\K$-vector spaces, then
$\Hom$-spaces in $\CS(\AS)$ are $\Z\times \Z$-graded complexes of $\K$-vector spaces.
In this context, we will decorate the grading group by subscripts, e.g.
$\Z_{\qdeg}\times \Z_{\tdeg}$ to distinguish the internal $\Z_{\qdeg}=\Z$-grading from the
homological $\Z_{\tdeg}=\Z$-grading.\end{remark}

We are interested in complexes of singular Soergel bimodules. 
\begin{defi}
Let $\CS(\SSBim)$ be the monoidal 2-category with the same objects as $\SSBim$,
and wherein the $\Hom$-category $\aa\rightarrow \bb$ equals
$\CS({}_{\bb}\SSBim_{\aa})$ and the composition operations and monoidal
structure are inherited from $\SSBim$ and described below.
\end{defi}

In other words, 1-morphisms in $\CS(\SSBim)$ are complexes of Soergel bimodules, 
and 2-morphism spaces in $\CS(\SSBim)$ are complexes of bimodule maps.

\begin{conv}
In the notation of Remark \ref{rem:enrich}, 
the 1-morphism categories of $\SSBim$ are enriched in $\Z_{\qdeg}$-graded
$\K$-vector spaces,
so the 1-morphism category $\oone_{\bb}\CS(\SSBim)\oone_{\aa}$ is
enriched in $\Z_{\qdeg}\times\Z_{\tdeg}$-graded complexes of $\K$-vector spaces.  
We will use the convention that $\deg(f)=(i,j)$ means $f$ has $\qdeg$-degree (or
``Soergel degree'') $i$ and homological degree $j$. Further, the singly-indexed
$\Hom$-space $\Hom_{\CS(\SSBim)}^k(X,Y)$ always refers to homological degree,
while the doubly-indexed $\Hom_{\CS(\SSBim)}^{i,j}(X,Y)$ consists of $f$ with
$\deg(f)=(i,j)$. We will typically indicate these degrees multiplicatively by
writing $\wt(f) = \qdeg^i \tdeg^j$, and will also use the notation 
$\qdeg,\tdeg$ to denote the corresponding shift functors. 
\end{conv}

The (horizontal) composition of 1-morphisms is defined as usual:
\[
(X\hComp Y)^k = \bigoplus_{i+j=k} X^i\hComp Y^j
\, , \quad
\d_{X\hComp Y} = \d_{X}\hComp \Id_Y + \Id_X\hComp \d_Y \, .
\]
Here, the components of a horizontal composition of 2-morphisms 
are defined using the Koszul sign rule. 
Explicitly, if $f\in \Hom_{\CS(\SSBim)}(X,X')$ and $g\in
\Hom_{\CS(\SSBim)}(Y,Y')$ are given, then $f\hComp g$ is defined
component-wise by:
\[
(f\hComp g)|_{X^i\hComp Y^j} = (-1)^{i |g|} f|_{X^i}\hComp g|_{Y^j} \, .
\]
A direct computation shows that the (graded) middle interchange law is satisfied:
\[
(f_1\hComp g_1) \circ (f_2\hComp g_2) 
= (-1)^{|g_1| |f_2|}(f_1 \circ f_2) \hComp (g_1 \circ g_2) \, .
\]

The monoidal structure on $\CS(\SSBim)$ is given by extending the 
external tensor product $\boxtimes \colon \SSBim\rightarrow \SSBim$
to complexes, again following standard conventions. 
Explicitly, 
the external tensor product of 1-morphisms $X,Y\in \CS(\SSBim)$ is defined by
\[
(X\boxtimes Y)^k := \bigoplus_{i+j=k} X^i\boxtimes Y^j 
\, , \quad 
\d_{X\boxtimes Y} = \d_X\boxtimes\Id_Y +  \Id_X\boxtimes \d_Y
\]
where, as before, the external tensor product of 2-morphisms in 
$\CS(\SSBim)$ is defined component-wise using the Koszul sign rule:
\[
(f\boxtimes g)|_{X^i\boxtimes Y^j} = (-1)^{i |g|} f|_{X^i}\boxtimes g|_{Y^j}.
\]

It is straightforward to see that $\CS(\SSBim)$ is a monoidal 2-category in
which the 2-morphism spaces are $\Z_{\qdeg}\times \Z_{\tdeg}$-graded complexes, 
and all three of vertical composition $\circ$, horizontal composition $\hComp$, 
and external tensor product $\boxtimes$ of $2$-morphisms satisfy 
appropriate versions of the Leibniz rule; 
i.e. $\CS(\SSBim)$ is a differential $\Z_{\qdeg}\times \Z_{\tdeg}$-graded monoidal 2-category.  
Henceforth, we will slightly abuse terminology and simply refer to $\CS(\SSBim)$ 
as a dg monoidal 2-category (the additional grading on 2-morphism
complexes will be understood throughout).

We let $\KS(\SSBim) = H^0(\CS(\SSBim))$ be the \emph{cohomology category} of
$\CS(\SSBim)$. Its objects and 1-morphisms are the same as in $\CS(\SSBim)$, but
its 2-morphisms are now given by degree-zero cohomology classes in
$\Hom_{\CS(\SSBim)}(-,-)$, 
i.e. by degree-zero chain maps modulo homotopy.  
In other words, $\KS(\SSBim)$ is the usual homotopy category of (bounded)
complexes over $\SSBim$.  
The horizontal composition and external tensor product descend to $\KS(\SSBim)$,
making the latter into a triangulated monoidal 2-category.

\subsection{Categorified quantum \texorpdfstring{$\glm$}{gl(m)}}
\label{sec:CQG}

Let $\USd(\glm)$ denote the $\glm$ analogue of the
Khovanov-Lauda-Rouquier categorified quantum group \cite{KL,KL2,KL3,Rou2}
associated to the Lie algebra $\slm$. This
2-category is the Karoubi completion of the graded, additive 2-category
$\US(\glm)$ in which 
objects are $\glm$ weights $\aa = (a_1,\ldots,a_m)$, 
$1$-morphisms are generated by
\[
\E_i \one_{\aa} : \aa \to \aa + \varepsilon_i 
\, , \quad
\F_i \one_{\aa} : \aa \to \aa - \varepsilon_i
\]
for $i=1,\ldots,m-1$ (here $\varepsilon_i = (0,\ldots,0,1,-1,0,\ldots,0)$), and
$2$-morphisms are given using $\slm$ Khovanov-Lauda string diagrams. We will
assume some familiarity with the diagrammatic presentation of $\US(\glm)$;
in fact, only categorified $\glnn{2}$ computations will be used in this paper, 
so knowledge of the latter will suffice.

Of particular importance are the ``divided power'' $1$-morphisms $\E_i^{(k)}
\one_{\aa}$ and $\F_i^{(k)} \one_{\aa}$ in $\USd(\glm)$. These are
indecomposable $1$-morphisms that satisfy
\[
\E_i^{k} \one_{\aa} \cong \bigoplus_{[k]!} \E_i^{(k)} \one_{\aa}
\, , \quad
\F_i^{k} \one_{\aa} \cong \bigoplus_{[k]!} \F_i^{(k)} \one_{\aa}
\]
We will use $\USd(\glm)$ as a technical tool for studying $\mathrm{SSBim_N}$
via the following result.
This essentially appears in \cite{KL3}, 
but can also be deduced from the main result of \cite{QR} and the correspondence 
between foams and $\SSBim$.
\begin{prop}\label{prop:SH}
For $m \leq N$, there is a $2$-functor $\Phi\colon \US(\glm) \to \mathrm{SSBim_N}$ 
that extends to the full $2$-subcategory generated by the divided powers, 
that sends objects $\aa \mapsto R^{\aa}$ and 1-morphisms:
\[
\begin{aligned}
\one_{\aa} &\mapsto \oone_{\aa} \\
\E_i^{(k)} \one_{\aa} &\mapsto
\oone_{(a_1,\ldots,a_{i-1})} \boxtimes
\begin{tikzpicture}[smallnodes,rotate=90,anchorbase,scale=.75,yscale=-1]
	\draw[very thick] (0,.25) to [out=150,in=270] (-.25,1) 
		node[right,xshift=-2pt]{$a_i$};
	\draw[very thick] (.5,.5) to (.5,1) node[right,xshift=-2pt]{$a_{i+1}$};
	\draw[very thick] (0,.25) to node[left,xshift=2pt]{$k$} (.5,.5);
	\draw[very thick] (.5,-1) node[left,xshift=2pt]{$a_{i+1}{-}k$} 
		to [out=90,in=330] (.5,.5);
	\draw[very thick] (-.25,-1)node[left,xshift=2pt]{$a_i{+}k$} 
		to [out=90,in=270] (0,.25);
\end{tikzpicture}
\boxtimes \oone_{(a_{i+2},\ldots,a_m)} \\
\F_i^{(k)} \one_{\aa} &\mapsto
\oone_{(a_1,\ldots,a_{i-1})} \boxtimes
\begin{tikzpicture}[smallnodes,rotate=90,anchorbase,scale=.75]
	\draw[very thick] (0,.25) to [out=150,in=270] (-.25,1) 
		node[left,xshift=2pt]{$a_i{-}k$};
	\draw[very thick] (.5,.5) to (.5,1) node[left,xshift=2pt]{$a_{i+1}{+}k$};
	\draw[very thick] (0,.25) to node[right,xshift=-2pt]{$k$} (.5,.5);
	\draw[very thick] (.5,-1) node[right,xshift=-2pt]{$a_{i+1}$} 
		to [out=90,in=330] (.5,.5);
	\draw[very thick] (-.25,-1)node[right,xshift=-2pt]{$a_i$} 
		to [out=90,in=270] (0,.25);
\end{tikzpicture}
\boxtimes \oone_{(a_{i+2},\ldots,a_m)}
\end{aligned}
\]
\end{prop}

The value of $\Phi$ on $2$-morphisms can be deduced from 
\cite[Lemma 3.7, Theorem 3.9, and Corollary 3.10]{QR} 
and the correspondence between foams and singular Soergel bimodules. 
However, we caution the reader that the $2$-functor $\Phi$ appearing in 
Proposition \ref{prop:SH} does not agree on the nose with the one appearing 
in \cite{QR}, since our current conventions for where $\Phi$ sends the $1$-morphisms 
$\E_i^{(k)} \one_{\aa}$ and $\F_i^{(k)} \one_{\aa}$ are opposite. 
Indeed, it is obtained from the $2$-functor in \cite{QR} by further composing with 
an autoequivalence that reflects foams in the direction perpendicular to the page 
(and rescales certain generators by $\pm1$).

The $m=2$ case will be particularly important.  
In this case,
\[
\F^{(l)} \E^{(k)}\one_{a,b}
\xmapsto{\Phi}
\begin{tikzpicture}[rotate=90,anchorbase,smallnodes]
	\draw[very thick] (0,.25) to [out=150,in=270] (-.25,1) 
		node[left,xshift=2pt]{$a{+}k{-}l$};
	\draw[very thick] (.5,.5) to (.5,1) node[left,xshift=2pt]{$b{-}k{+}l$};
	\draw[very thick] (0,.25) to node[left,yshift=-1pt]{$l$} (.5,.5);
	\draw[very thick] (0,-.25) to node[below,yshift=-1pt]{$a{+}k$} (0,.25);
	\draw[very thick] (.5,-.5) to [out=30,in=330] 
		node[above,yshift=-2pt]{$b{-}k$} (.5,.5);
	\draw[very thick] (0,-.25) to node[right,yshift=-1pt,xshift=-1pt]{$k$} (.5,-.5);
	\draw[very thick] (.5,-1) node[right,xshift=-2pt]{$b$} to (.5,-.5);
	\draw[very thick] (-.25,-1)node[right,xshift=-2pt]{$a$} 
		to [out=90,in=210] (0,-.25);
\end{tikzpicture}
\]
and all $2$-morphisms in $\USd(\glnn{2})$ can be described using the 
extended graphical calculus from \cite{KLMS}. 
For example, the following give $2$-morphisms in $\SSBim$
that will appear throughout this paper:
\begin{equation}\label{eq:chi plus}
\chi_r^+ :=
\Phi \left(
\CQGsgn{(-1)^{b-k}}
\begin{tikzpicture}[baseline=0,smallnodes]
\draw[CQG,ultra thick,<-] (0,-.5) node[below]{\scriptsize$l$}  to (0,.7);
\draw[CQG,ultra thick,->] (.75,-.5) node[below]{\scriptsize$k$} to (.75,.7);
\draw[CQG,thick, directed=.75] (.75,0) to [out=90,in=90] 
	node[black,yshift=-.5pt]{$\bullet$} node[below,black]{\scriptsize$r$} (0,0);
\node at (1.25,.5) {$(a,b)$};
\end{tikzpicture}
\right)
:
\begin{tikzpicture}[rotate=90,anchorbase,smallnodes]
	\draw[very thick] (0,.25) to [out=150,in=270] (-.25,1) 
		node[left,xshift=2pt]{$a{+}k{-}l$};
	\draw[very thick] (.5,.5) to (.5,1) node[left,xshift=2pt]{$b{-}k{+}l$};
	\draw[very thick] (0,.25) to node[left,yshift=-1pt]{$l$} (.5,.5);
	\draw[very thick] (0,-.25) to node[below,yshift=-1pt]{$a{+}k$} (0,.25);
	\draw[very thick] (.5,-.5) to [out=30,in=330] 
		node[above,yshift=-2pt]{$b{-}k$} (.5,.5);
	\draw[very thick] (0,-.25) to node[right,yshift=-1pt,xshift=-1pt]{$k$} (.5,-.5);
	\draw[very thick] (.5,-1) node[right,xshift=-2pt]{$b$} to (.5,-.5);
	\draw[very thick] (-.25,-1)node[right,xshift=-2pt]{$a$} 
		to [out=90,in=210] (0,-.25);
\end{tikzpicture}
\longrightarrow
\begin{tikzpicture}[rotate=90,anchorbase,smallnodes]
	\draw[very thick] (0,.25) to [out=150,in=270] (-.25,1) 
		node[left,xshift=2pt]{$a{+}k{-}l$};
	\draw[very thick] (.5,.5) to (.5,1) node[left,xshift=2pt]{$b{-}k{+}l$};
	\draw[very thick] (0,.25) to node[left,yshift=-1pt]{$l{-}1$} (.5,.5);
	\draw[very thick] (0,-.25) to node[below,yshift=-3pt]{$a{+}k{-}1$} (0,.25);
	\draw[very thick] (.5,-.5) to [out=30,in=330] 
		node[above,yshift=-2pt]{$b{-}k{+}1$} (.5,.5);
	\draw[very thick] (0,-.25) to node[right,yshift=-1pt,xshift=-1pt]{$k{-}1$} (.5,-.5);
	\draw[very thick] (.5,-1) node[right,xshift=-2pt]{$b$} to (.5,-.5);
	\draw[very thick] (-.25,-1)node[right,xshift=-2pt]{$a$} 
		to [out=90,in=210] (0,-.25);
\end{tikzpicture}
\end{equation}
\begin{equation}\label{eq:chi minus}
\chi_r^- :=
\Phi \left(
\CQGsgn{(-1)^{a+b+k+l-1}}
\begin{tikzpicture}[baseline=0,smallnodes,yscale=-1]
\draw[CQG,ultra thick,->] (0,-.5) to (0,.7) node[below]{\scriptsize$l$};
\draw[CQG,ultra thick,<-] (.75,-.5) to (.75,.7) node[below]{\scriptsize$k$};
\draw[CQG,thick, rdirected=.3] (.75,0) to [out=90,in=90] 
	node[black,yshift=-.5pt]{$\bullet$} node[below,black]{\scriptsize$r$} (0,0);
\node at (1.25,.5) {$(a,b)$};
\end{tikzpicture}
\right)
:
\begin{tikzpicture}[rotate=90,anchorbase,smallnodes]
	\draw[very thick] (0,.25) to [out=150,in=270] (-.25,1) 
		node[left,xshift=2pt]{$a{+}k{-}l$};
	\draw[very thick] (.5,.5) to (.5,1) node[left,xshift=2pt]{$b{-}k{+}l$};
	\draw[very thick] (0,.25) to node[left,yshift=-1pt]{$l$} (.5,.5);
	\draw[very thick] (0,-.25) to node[below,yshift=-1pt]{$a{+}k$} (0,.25);
	\draw[very thick] (.5,-.5) to [out=30,in=330] 
		node[above,yshift=-2pt]{$b{-}k$} (.5,.5);
	\draw[very thick] (0,-.25) to node[right,yshift=-1pt,xshift=-1pt]{$k$} (.5,-.5);
	\draw[very thick] (.5,-1) node[right,xshift=-2pt]{$b$} to (.5,-.5);
	\draw[very thick] (-.25,-1)node[right,xshift=-2pt]{$a$} 
		to [out=90,in=210] (0,-.25);
\end{tikzpicture}
\longrightarrow
\begin{tikzpicture}[rotate=90,anchorbase,smallnodes]
	\draw[very thick] (0,.25) to [out=150,in=270] (-.25,1) 
		node[left,xshift=2pt]{$a{+}k{-}l$};
	\draw[very thick] (.5,.5) to (.5,1) node[left,xshift=2pt]{$b{-}k{+}l$};
	\draw[very thick] (0,.25) to node[left,yshift=-1pt]{$l{+}1$} (.5,.5);
	\draw[very thick] (0,-.25) to node[below,yshift=-3pt]{$a{+}k{+}1$} (0,.25);
	\draw[very thick] (.5,-.5) to [out=30,in=330] 
		node[above,yshift=-2pt]{$b{-}k{-}1$} (.5,.5);
	\draw[very thick] (0,-.25) to node[right,yshift=-1pt,xshift=-1pt]{$k{+}1$} (.5,-.5);
	\draw[very thick] (.5,-1) node[right,xshift=-2pt]{$b$} to (.5,-.5);
	\draw[very thick] (-.25,-1)node[right,xshift=-2pt]{$a$} 
		to [out=90,in=210] (0,-.25);
\end{tikzpicture}
\end{equation}
Both of these $2$-morphisms have $q$-degree equal to $1+2r+(a-b)+(k-l)$.
The {\color{CQG} green} signs appearing here
(and in some places below) are conventional, 
and guarantee that the image is the bimodule morphism 
corresponding to an \emph{unsigned} foam.
See Appendix \ref{s:FoamSSBim} for the translation between foams 
and bimodule morphisms.

\begin{rem}\label{rem:Decoration}
If $f,g$ are symmetric functions, 
then, by Convention \ref{conv:Alph}, $f(\M)\otimes g(\M')$ is a well-defined 
endomorphisms of $\Phi(\F^{(l)} E^{(k)}\one_{a,b})$.
In fact, this endomorphism is in the image of $\Phi$, 
and is described in extended graphical calculus as:
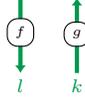
\[
\begin{tikzpicture}[smallnodes]
\draw[CQG,ultra thick,<-] (0,0) node[below]{\scriptsize$l$} to 
	node[black,yshift=-3pt]{$\CQGbox{f}$} (0,1);
\draw[CQG,ultra thick,->] (.75,0) node[below]{\scriptsize$k$} to 
	node[black,yshift=-3pt]{$\CQGbox{g}$} (.75,1);
\end{tikzpicture}
\]
\end{rem}

\begin{rem} 
The graphical calculus for $\USd(\glnn{2})$ contains cap and cup morphisms
between the identity morphisms $\one_{\aa}$ and the (horizontal) compositions
$\F^{(k)} \E^{(k)}\one_{\aa}$ and $\E^{(k)} \F^{(k)}\one_{\aa}$. Vertical
composition of these cap and cup morphisms with 
endomorphisms (as in Remark \ref{rem:Decoration}) give so-called 
\emph{bubble endomorphisms} of $\one_{\aa}$. 
To record the images of these endomorphisms under $\Phi$, let us
denote the alphabets associated to $\oone_{\aa} = \oone_{a,b}$ by $\Fr$ with
$|\Fr|=a$ and $\B$ with $|\B|=b$.
This is compatible with Convention~\ref{conv:Alph}, 
since in the case of no rungs we have $\leftX_1 = \Fr=\rightX_1$ 
and $\leftX_2 = \B = \rightX_2$.

In the case of a \emph{thin bubble} (the $k=1$ case), 
\cite[(3.10) and (3.14)]{QR} imply that
\begin{equation}\label{eq:bubble}
\Phi\left(
	\begin{tikzpicture}[anchorbase,smallnodes]
\draw[CQG,thick,directed=.25] (1,0) arc (0:361:.25) node[black,pos=.75]{$\bullet$} 
	node[below,black,pos=.75]{\scriptsize$\spadesuit{+}r$};
	node
\end{tikzpicture}
\right)
= h_r(\B-\Fr) 
\, , \quad
\Phi\left(
	\begin{tikzpicture}[anchorbase,smallnodes]
\draw[CQG,thick,rdirected=.25] (1,0) arc (0:361:.25) node[black,pos=.75]{$\bullet$} 
	node[below,black,pos=.75]{\scriptsize$\spadesuit{+}r$};
\end{tikzpicture}
\right)
= h_r(\Fr-\B) 
\end{equation}
Here the $\spadesuit$ is a placeholder for a minimal decoration required to
obtain a non-trivial evaluation 
(the precise value, which depends on the weight $\aa$, will not be relevant here).
The values of \emph{thick bubbles} ($k>1$) are then
\begin{equation}\label{eq:thickbubble}
	\Phi\left(
	\begin{tikzpicture}[anchorbase,smallnodes]
		\draw[CQG,ultra thick,directed=.25] (1.5,0) arc (0:361:0.5) node[right]{$k$}
			node[black,pos=.75]{$\CQGbbox{\mathfrak{s}_\alpha^\spadesuit}$};
	\end{tikzpicture}
	\!\!
	\right)
	= (-1)^{k(k-1)/2} \mathfrak{s}_{\alpha}(\B-\Fr)
\, , \quad
	\Phi\left(
	\begin{tikzpicture}[anchorbase,smallnodes]
		\draw[CQG,ultra thick,rdirected=.25] (1.5,0) arc (0:361:0.5) node[right]{$k$}
			node[black,pos=.75]{$\CQGbbox{\mathfrak{s}_\alpha^\spadesuit}$};
	\end{tikzpicture}
	\!\!
	\right)
	= (-1)^{k(k-1)/2} \mathfrak{s}_{\alpha}(\Fr - \B)
\end{equation}
which can be deduced from \eqref{eq:bubble}, 
e.g. using \cite[(4.33) and (4.34)]{KLMS} and the Jacobi-Trudi formula.
\end{rem}

\begin{conv}\label{conv:CQG}
In the following, we will almost exclusively be
interested in the images of $1$- and $2$-morphisms of $\USd(\glm)$ under $\Phi$, 
rather than the elements in the categorified quantum group itself.
As such, we will omit $\Phi$ from our notation and use the notation in $\USd(\glm)$
(but with the identity $1$-morphisms $\one_{\aa}$ in $\USd(\glm)$ replaced 
by the identity $1$-morphisms $\oone_{\aa}$ in $\SSBim$)
to denote the corresponding $1$- and $2$-morphisms in $\SSBim$.
\end{conv}

\subsection{Rickard complexes}
\label{sec:colbraid}

In this section, we recall the complexes of singular Soergel bimodules assigned
to colored braids. To begin, fix a set of \emph{colors} $S$, which will 
be $\Z_{\geq 1}$ in this paper. Let $\Br_m$ denote the $m$-strand
braid group, which acts on $S^m$ by permuting coordinates (this action factors
through the symmetric group
$\mathfrak{S}_m$).

\begin{defi}\label{def:CBG}
The \emph{$S$-colored braid groupoid} $\mathfrak{Br}(S)$ is the category wherein
objects are sequences $(a_1,\ldots,a_m)$ with $a_i\in S$, $m\geq 0$, and
morphisms given by
\[
\Hom_{\mathfrak{Br}(S)}(\aa,\bb)=\left\{\b \in \Br_m\:|\: a_i=b_{\b(i)} 
\text{ for } 1 \leq i \leq m \right\}
\]
with $\aa=(a_1,\ldots,a_m)$ and $\bb=(b_1,\ldots,b_m)$. 
\end{defi}

Morphisms in $\mathfrak{Br}(S)$ are called \emph{colored braids}, and
elements in $\Hom_{\mathfrak{Br}}(\aa,\bb)$ will be denoted by ${}_{\bb}
\b_{\aa}$, or occasionally by ${}_{\bb} \b$ or $ \b_{\aa}$ since the
domain/codomain are determined by one another.

Given a braid $\b\in \Br_m$, a \emph{strand} of $\b$ is a pair of indices
$(i,j)\in \{1,\ldots,m\}^2$ with $i=\b(j)$. In the topological interpretation of
$\Br_m$, a strand of $\b$ corresponds to a connected component. Denote the set
of strands of $\beta$ by $\mathrm{strands}(\b)$. A colored braid ${}_{\bb}
\b_{\aa}$ gives rise to a well-defined function
\begin{equation}\label{eq:strand labelling}
\phi:\mathrm{strands}(\b) \rightarrow \Z_{\geq 1}
\end{equation}
defined by declaring $\phi(s) = b_i = a_j$, 
where $s$ is the strand $s=(i,j)$ (with $i=\b(j)$).
Conversely, given $\b \in \Br_m$, we can associate 
to it a colored braid ${}_{\bb} \b_{\aa}$ by specifying a function 
as in \eqref{eq:strand labelling}.

The colored braid groupoid is generated by the colored Artin generators
\[
\sigma_i : (a_1,\ldots,a_i,a_{i+1},\ldots,a_m) \rightarrow 
(a_1,\ldots,a_{i+1},a_i,\ldots,a_m)
\]
which, when composable, satisfy relations analogous to the usual (type $A$) 
braid relations.
A \emph{colored braid word} is a sequence of colored Artin generators 
and their inverses. 
We say that a colored braid word $(\underline{\b})_{\aa}$ \emph{represents} the
corresponding product of colored Artin generators in $\mathfrak{Br}(S)$.   

We now use the colored Artin generators to associate complexes
$C({}_{\bb}\b_{\aa})$ in $\SSBim$ to colored braids ${}_{\bb} {\b}_{\aa}$.
Strictly speaking $C({}_{\bb}\b_{\aa})$ depends on a choice of colored braid
word $\underline{\b}$ representing $\b$, but two different choices are
(canonically) homotopy equivalent; see Proposition~\ref{prop:rickard invariance}
below. We often abuse notation by writing:
\[
C({}_{\bb} \b_{\aa}) = \oone_{\bb}  C(\b)  \oone_{\aa} 
= \oone_{\bb}  C(\b) = C(\b)  \oone_{\aa} \, .
\]
(Note that $C(\b)$ alone does not denote a well-defined complex.) We will define
$C(\b)\oone_{\aa}$ by first defining it for the colored Artin generators
$\sigma_i^\pm$, and then extending to arbitrary braid words using horizontal
composition $\hComp$.  
In turn, to define $C(\sigma_i^\pm)\oone_{\aa}$ it suffices to consider the $m=2$
case and extend to arbitrary $m$ using the external tensor product.

\begin{defi}\label{def:Rickardcx}
Let $a,b \geq 0$. 
The \emph{$2$-strand Rickard complex} $C_{a,b}$ is the (bounded) complex
\[
\begin{aligned}
C_{a,b} :=& 
\left\llbracket
\begin{tikzpicture}[rotate=90,scale=.5,smallnodes,anchorbase]
	\draw[very thick] (1,-1) node[right,xshift=-2pt]{$b$} to [out=90,in=270] (0,1);
	\draw[line width=5pt,color=white] (0,-1) to [out=90,in=270] (1,1);
	\draw[very thick] (0,-1) node[right,xshift=-2pt]{$a$} to [out=90,in=270] (1,1);
\end{tikzpicture}
\right\rrbracket :=
\left(
\cdots 
\xrightarrow{\;\; \chi_0^+ \;}
\qdeg^{-k} \tdeg^k
\begin{tikzpicture}[smallnodes,rotate=90,anchorbase,scale=.66]
	\draw[very thick] (0,.25) to [out=150,in=270] (-.25,1) node[left,xshift=2pt]{$b$};
	\draw[very thick] (.5,.5) to (.5,1) node[left,xshift=2pt]{$a$};
	\draw[very thick] (0,.25) to (.5,.5);
	\draw[very thick] (0,-.25) to (0,.25);
	\draw[very thick] (.5,-.5) to [out=30,in=330] node[above,yshift=-2pt]{$k$} (.5,.5);
	\draw[very thick] (0,-.25) to (.5,-.5);
	\draw[very thick] (.5,-1) node[right,xshift=-2pt]{$b$} to (.5,-.5);
	\draw[very thick] (-.25,-1)node[right,xshift=-2pt]{$a$} to [out=90,in=210] (0,-.25);
\end{tikzpicture}
\xrightarrow{\;\; \chi_0^+ \;}
\qdeg^{-k-1}\tdeg^{k+1}
\begin{tikzpicture}[smallnodes,rotate=90,anchorbase,scale=.66]
	\draw[very thick] (0,.25) to [out=150,in=270] (-.25,1) node[left,xshift=2pt]{$b$};
	\draw[very thick] (.5,.5) to (.5,1) node[left,xshift=2pt]{$a$};
	\draw[very thick] (0,.25) to (.5,.5);
	\draw[very thick] (0,-.25) to (0,.25);
	\draw[very thick] (.5,-.5) to [out=30,in=330] 
		node[above,yshift=-2pt]{$k{+}1$} (.5,.5);
	\draw[very thick] (0,-.25) to (.5,-.5);
	\draw[very thick] (.5,-1) node[right,xshift=-2pt]{$b$} to (.5,-.5);
	\draw[very thick] (-.25,-1)node[right,xshift=-2pt]{$a$} to [out=90,in=210] (0,-.25);
\end{tikzpicture}
\xrightarrow{\;\; \chi_0^+ \;}
\cdots
\right)
\end{aligned}
\]
of singular Soergel bimodules. The rightmost non-zero term is either
$\qdeg^{-b}\tdeg^b \F^{(a-b)}\oone_{a,b}$ or $\qdeg^{-a}\tdeg^a
\E^{(b-a)}\oone_{a,b}$ (via Convention \ref{conv:CQG}) depending on whether
$a\geq b$ or $a\leq b$, respectively. 
As a graded object, we identify
$C_{a,b}=\bigoplus_{k=0}^{\min(a,b)} \qdeg^{-k}\tdeg^k C_{a,b}^k$, 
where $C_{a,b}^k = \F^{(a-k)} \E^{(b-k)} \oone_{a,b}$.
\end{defi}

\begin{rem}
In some works, the complex in Definition \ref{def:Rickardcx} is used only in the 
case that $a \geq b$, and is instead replaced by an analogously defined complex
\[
\left(
\cdots 
\xrightarrow{\;\; \chi_0^+ \;}
\qdeg^{-k} \tdeg^k
\begin{tikzpicture}[smallnodes,rotate=90,anchorbase,scale=.66, xscale=-1]
	\draw[very thick] (0,.25) to [out=150,in=270] (-.25,1) node[left,xshift=2pt]{$a$};
	\draw[very thick] (.5,.5) to (.5,1) node[left,xshift=2pt]{$b$};
	\draw[very thick] (0,.25) to (.5,.5);
	\draw[very thick] (0,-.25) to (0,.25);
	\draw[very thick] (.5,-.5) to [out=30,in=330] node[below,yshift=1pt]{$k$} (.5,.5);
	\draw[very thick] (0,-.25) to (.5,-.5);
	\draw[very thick] (.5,-1) node[right,xshift=-2pt]{$a$} to (.5,-.5);
	\draw[very thick] (-.25,-1)node[right,xshift=-2pt]{$b$} to [out=90,in=210] (0,-.25);
\end{tikzpicture}
\xrightarrow{\;\; \chi_0^+ \;}
\qdeg^{-k-1}\tdeg^{k+1}
\begin{tikzpicture}[smallnodes,rotate=90,anchorbase,scale=.66, xscale=-1]
	\draw[very thick] (0,.25) to [out=150,in=270] (-.25,1) node[left,xshift=2pt]{$a$};
	\draw[very thick] (.5,.5) to (.5,1) node[left,xshift=2pt]{$b$};
	\draw[very thick] (0,.25) to (.5,.5);
	\draw[very thick] (0,-.25) to (0,.25);
	\draw[very thick] (.5,-.5) to [out=30,in=330] node[below,yshift=1pt]{$k{+}1$} (.5,.5);
	\draw[very thick] (0,-.25) to (.5,-.5);
	\draw[very thick] (.5,-1) node[right,xshift=-2pt]{$a$} to (.5,-.5);
	\draw[very thick] (-.25,-1)node[right,xshift=-2pt]{$b$} to [out=90,in=210] (0,-.25);
\end{tikzpicture}
\xrightarrow{\;\; \chi_0^+ \;}
\cdots
\right)
\]
when $a < b$.
However, it follows \eg from \cite[Corollary 5.5]{KLMS} 
that these complexes are isomorphic for all $a,b \geq 0$.
\end{rem}

For $\b=\sigma_i$ and $\aa=(a_1,\ldots,a_m)$, we then set
\begin{equation}\label{eq:ArtinGen}
\begin{aligned}
C(\sigma_i) \oone_{\aa} &:= \oone_{(a_1,\ldots,a_{i-1})}\boxtimes 
C_{a_i,a_{i+1}}\boxtimes \oone_{(a_{i+2} \ldots,a_m)} \\
C(\sigma_i\inv) \oone_{\aa} &:= \oone_{(a_1,\ldots,a_{i-1})}\boxtimes 
C_{a_{i},a_{i+1}}^\vee\boxtimes \oone_{(a_{i+2},\ldots,a_m)}
\end{aligned}
\end{equation}
where $C_{a_i,a_{i+1}}$ is the 2-strand Rickard complex from
Definition~\ref{def:Rickardcx} and $C_{a_{i+1},a_i}^\vee$ is its inverse. 
The latter is obtained from $C_{a_i,a_{i+1}}$ by applying 
the contravariant duality functor 
$(-)^\vee := \Hom_{R^{(a_i,a_{i+1})}}(-,R^{(a_i,a_{i+1})})$ 
and is explicitly given by
\[
C^\vee_{a,b} 
:= 
	\left\llbracket
	\begin{tikzpicture}[rotate=90,yscale=.5,xscale=-.5,smallnodes,anchorbase]
		\draw[very thick] (1,-1) node[right,xshift=-2pt]{$a$} to [out=90,in=270] (0,1);
		\draw[line width=5pt,color=white] (0,-1) to [out=90,in=270] (1,1);
		\draw[very thick] (0,-1) node[right,xshift=-2pt]{$b$} to [out=90,in=270] (1,1);
	\end{tikzpicture}
	\right\rrbracket 
:=
\left(
	\cdots 
	\xrightarrow{\;\; \chi_0^- \;}
	\qdeg^{k+1} \tdeg^{-k-1}
	\begin{tikzpicture}[smallnodes,rotate=90,anchorbase,scale=.75]
		\draw[very thick] (0,.25) to [out=150,in=270] (-.25,1) node[left,xshift=2pt]{$b$};
		\draw[very thick] (.5,.5) to (.5,1) node[left,xshift=2pt]{$a$};
		\draw[very thick] (0,.25) to (.5,.5);
		\draw[very thick] (0,-.25) to (0,.25);
		\draw[very thick] (.5,-.5) to [out=30,in=330] node[above,yshift=-2pt]{$k{+}1$} (.5,.5);
		\draw[very thick] (0,-.25) to (.5,-.5);
		\draw[very thick] (.5,-1) node[right,xshift=-2pt]{$b$} to (.5,-.5);
		\draw[very thick] (-.25,-1)node[right,xshift=-2pt]{$a$} to [out=90,in=210] (0,-.25);
	\end{tikzpicture}
	\xrightarrow{\;\; \chi_0^- \;}
	\qdeg^{k}\tdeg^{-k}
	\begin{tikzpicture}[smallnodes,rotate=90,anchorbase,scale=.75]
		\draw[very thick] (0,.25) to [out=150,in=270] (-.25,1) node[left,xshift=2pt]{$b$};
		\draw[very thick] (.5,.5) to (.5,1) node[left,xshift=2pt]{$a$};
		\draw[very thick] (0,.25) to (.5,.5);
		\draw[very thick] (0,-.25) to (0,.25);
		\draw[very thick] (.5,-.5) to [out=30,in=330] 
			node[above,yshift=-2pt]{$k$} (.5,.5);
		\draw[very thick] (0,-.25) to (.5,-.5);
		\draw[very thick] (.5,-1) node[right,xshift=-2pt]{$b$} to (.5,-.5);
		\draw[very thick] (-.25,-1)node[right,xshift=-2pt]{$a$} to [out=90,in=210] (0,-.25);
	\end{tikzpicture}
	\xrightarrow{\;\; \chi_0^- \;}
	\cdots
\right) \, .
\]

The assignment \eqref{eq:ArtinGen} extends to
arbitrary colored braid words using horizontal composition:
\begin{equation}\label{eq:Rick-hComp}
C(\sigma_{i_1}^{\e_1}\cdots \sigma_{i_r}^{\e_r})\oone_{\aa}
:= C(\sigma_{i_1}^{\e_1}) \hComp \cdots \hComp C(\sigma_{i_r}^{\e_r})\oone_{\aa}
\end{equation}
for $\e_1,\ldots,\e_r\in \{+1,-1\}$ and $1\leq i_1,\ldots,i_r\leq m-1$.

\begin{prop}\label{prop:rickard invariance}
The complexes $C(\sigma_{i_1}^{\e_1}\cdots \sigma_{i_r}^{\e_r})\oone_{\aa}$
satisfy the (colored) braid relations, up to canonical homotopy equivalence.
\end{prop}
This is well-known in the uncolored case, i.e. when $\aa$ has $a_i=1$ 
for all $1\leq i \leq m$; see e.g. \cite{MR2721032}. 

\begin{proof}
The existence of such homotopy equivalences was conjectured in~\cite{MR2746676}
and proven in the geometric setting in~\cite{MR3687104}. In the singular Soergel
bimodule setting, the braid relations follow from \cite{CKbraiding,CKL} and
Proposition \ref{prop:SH}. As in the uncolored case, these homotopy equivalences
live in 1-dimensional $\Hom$-spaces in $\KS(\SSBim)$, and canonicity amounts to
a coherent choice of scaling. The latter can be obtained from the corresponding
coherent scaling in the framework of $\glnn{N}$ foams for $N \gg 0$
that was constructed in~\cite{ETW}.
\end{proof}

\begin{conv}
If $\b = \sigma_{i_1}^{\e_1}\cdots \sigma_{i_r}^{\e_r}$, 
then we call 
$C(\beta) \oone_{\aa} = C(\sigma_{i_1}^{\e_1}\cdots \sigma_{i_r}^{\e_r}) \oone_{\aa}$ 
the \emph{Rickard complex} assigned to the colored braid $\b_{\aa}$.
\end{conv}

Rickard complexes of colored braids extend to invariants
of braided webs (using horizontal composition and external tensor product),
since they satisfy the following \emph{fork-slide} and
\emph{twist-zipper} relations

\begin{prop} \label{prop:forkslide} We have homotopy equivalences
\begin{equation}\label{eq:forkslide}
\left\llbracket
\begin{tikzpicture}[rotate=90,scale=.5,smallnodes,anchorbase]
	\draw[very thick] (.5,-1) node[right,xshift=-2pt]{$c$} to [out=90,in=270] 
		(-1,2) node[left,xshift=2pt]{$c$};
	\draw[line width=5pt,color=white] (-.5,-1) to [out=90,in=270] (.5,1);
	\draw[very thick] (-.5,-1) node[right,xshift=-2pt]{$a{+}b$} to [out=90,in=270] (.5,1);
	\draw[very thick] (.5,1) to [out=30,in=270] (1,2) node[left,xshift=2pt]{$b$};
	\draw[very thick] (.5,1) to [out=150,in=270] (0,2) node[left,xshift=2pt]{$a$};
\end{tikzpicture}
\right\rrbracket	
\simeq
\left\llbracket
\begin{tikzpicture}[rotate=90,scale=.5,smallnodes,anchorbase]
	\draw[very thick] (.5,-1) node[right,xshift=-2pt]{$c$} to [out=90,in=270] 
		(-1,2) node[left,xshift=2pt]{$c$};
	\draw[line width=5pt,color=white] (-.5,-.25) to [out=30,in=270] (1,2);
	\draw[line width=5pt,color=white] (-.5,-.25) to [out=150,in=270] (0,2);
	\draw[very thick] (-.5,-1) node[right,xshift=-2pt]{$a{+}b$} to (-.5,-.25);
	\draw[very thick] (-.5,-.25) to [out=30,in=270] (1,2) node[left,xshift=2pt]{$b$};
	\draw[very thick] (-.5,-.25) to [out=150,in=270] (0,2) node[left,xshift=2pt]{$a$};
\end{tikzpicture}
\right\rrbracket
\, , \quad
\left\llbracket
\begin{tikzpicture}[rotate=90,scale=.5,smallnodes,anchorbase,yscale=-1]
	\draw[very thick] (-.5,-1) node[left,xshift=2pt]{$b{+}c$} to [out=90,in=270] (.5,1);
	\draw[very thick] (.5,1) to [out=30,in=270] (1,2) node[right,xshift=-2pt]{$c$};
	\draw[very thick] (.5,1) to [out=150,in=270] (0,2) node[right,xshift=-2pt]{$b$};
	\draw[line width=5pt,color=white] (.5,-1) to [out=90,in=270] (-1,2);
	\draw[very thick] (.5,-1) node[left,xshift=2pt]{$a$} to [out=90,in=270] 
		(-1,2) node[right,xshift=-2pt]{$a$};
\end{tikzpicture}
\right\rrbracket	
\simeq
\left\llbracket
\begin{tikzpicture}[rotate=90,scale=.5,smallnodes,anchorbase,yscale=-1]
	\draw[very thick] (-.5,-1) node[left,xshift=2pt]{$b{+}c$} to (-.5,-.25);
	\draw[very thick] (-.5,-.25) to [out=30,in=270] (1,2) node[right,xshift=-2pt]{$c$};
	\draw[very thick] (-.5,-.25) to [out=150,in=270] (0,2) node[right,xshift=-2pt]{$b$};
	\draw[line width=5pt,color=white] (.5,-1) to [out=90,in=270] (-1,2);
	\draw[very thick] (.5,-1) node[left,xshift=2pt]{$a$} to [out=90,in=270] 
		(-1,2) node[right,xshift=-2pt]{$a$};
\end{tikzpicture}
\right\rrbracket,
\end{equation}
\begin{equation}\label{eq:twistzipper}
	\left\llbracket
\begin{tikzpicture}[rotate=90,scale=.5,smallnodes,anchorbase]
	\draw[very thick] (1,-1) node[right,xshift=-2pt]{$b$} to [out=90,in=270] (0,1)
		to [out=90,in=210] (.5,2);
	\draw[line width=5pt,color=white] (0,-1) to [out=90,in=270] (1,1);
	\draw[very thick] (0,-1) node[right,xshift=-2pt]{$a$} to [out=90,in=270] (1,1)
		to [out=90,in=330] (.5,2);
	\draw[very thick] (.5,2) to (.5,2.75);
\end{tikzpicture}
\right\rrbracket		
\simeq
\qdeg^{a b}
\left\llbracket
\begin{tikzpicture}[rotate=90,scale=.5,smallnodes,anchorbase]
	\draw[very thick] (0,1) node[right,xshift=-2pt]{$b$} to [out=90,in=210] (.5,2);
	\draw[very thick] (1,1) node[right,xshift=-2pt]{$a$} to [out=90,in=330] (.5,2);
	\draw[very thick] (.5,2) to (.5,2.75);
\end{tikzpicture}
\right\rrbracket	
\end{equation}
as well as reflections thereof.
\end{prop}
\begin{proof} See \cite[(4.3) and (4.16)]{QR} and \cite[Lemma 5.2]{Cautis}. 
\end{proof}

\subsection{Shifted Rickard complexes}
\label{ss:shifted rickards}
We now define the \emph{shifted Rickard complexes},
which previously appeared in \cite[equations (12) and (13)]{Cautis} 
in the setting of the categorified quantum group $\USd(\slnn{2})$.
In passing to $\SSBim$, we show that these 
complexes possess a topological interpretation.

\begin{definition}\label{def:shifted rickards}
Fix integers $a,b,c,d$ with $a+b=c+d$, and consider the complex
\begin{align*}
{}_{(c,d)}C_{(a,b)} 
:=& 
\left( \begin{tikzpicture}[smallnodes,rotate=90,baseline=.1em,scale=.66]
\draw[very thick] (0,.25) to [out=150,in=270] (-.25,1) node[left,xshift=2pt]{$c$};
	\draw[very thick] (.5,.5) to (.5,1) node[left,xshift=2pt]{$d$};
\draw[very thick] (0,.25) to  (.5,.5);
\draw[very thick] (0,-.25) to (0,.25);
\draw[dotted] (.5,-.5) to [out=30,in=330] node[above=-2pt]{$0$} (.5,.5);
\draw[very thick] (0,-.25) to  (.5,-.5);
\draw[very thick] (.5,-1) node[right,xshift=-2pt]{$b$} to (.5,-.5);
\draw[very thick] (-.25,-1)node[right,xshift=-2pt]{$a$} to [out=90,in=210] (0,-.25);
\end{tikzpicture}
\xrightarrow{\; \chi_0^+}
\qdeg^{-(a-d+1)} \tdeg
\begin{tikzpicture}[smallnodes,rotate=90,baseline=.1em,scale=.66]
\draw[very thick] (0,.25) to [out=150,in=270] (-.25,1) node[left,xshift=2pt]{$c$};
\draw[very thick] (.5,.5) to (.5,1) node[left,xshift=2pt]{$d$};
\draw[very thick] (0,.25) to  (.5,.5);
\draw[very thick] (0,-.25) to (0,.25);
\draw[very thick] (.5,-.5) to [out=30,in=330] node[above=-2pt]{$1$} (.5,.5);
\draw[very thick] (0,-.25) to  (.5,-.5);
\draw[very thick] (.5,-1) node[right,xshift=-2pt]{$b$} to (.5,-.5);
\draw[very thick] (-.25,-1)node[right,xshift=-2pt]{$a$} to [out=90,in=210] (0,-.25);
\end{tikzpicture}
\xrightarrow{\; \chi_0^+}
\qdeg^{-2(a-d+1)} \tdeg^2
\begin{tikzpicture}[smallnodes,rotate=90,baseline=.1em,scale=.66]
\draw[very thick] (0,.25) to [out=150,in=270] (-.25,1) node[left,xshift=2pt]{$c$};
\draw[very thick] (.5,.5) to (.5,1) node[left,xshift=2pt]{$d$};
\draw[very thick] (0,.25) to  (.5,.5);
\draw[very thick] (0,-.25) to (0,.25);
\draw[very thick] (.5,-.5) to [out=30,in=330] node[above=-2pt]{$2$} (.5,.5);
\draw[very thick] (0,-.25) to  (.5,-.5);
\draw[very thick] (.5,-1) node[right,xshift=-2pt]{$b$} to (.5,-.5);
\draw[very thick] (-.25,-1)node[right,xshift=-2pt]{$a$} to [out=90,in=210] (0,-.25);
\end{tikzpicture}
\xrightarrow{\; \chi_0^+}
\cdots \right) \\
=& \left( \bigoplus_{k \geq 0} \qdeg^{-k(a-d+1)} \tdeg^k \, \F^{(d-k)} \E^{(b-k)} , \d_C \right)
\end{align*}
for 
\[
\d_C := \bigoplus_k \left(\chi_0^+ \colon \qdeg^{-k(a-d+1)} \tdeg^k \, \F^{(d-k)} \E^{(b-k)} 
\to \qdeg^{-(k+1)(a-d+1)} \tdeg^{k+1} \, \F^{(d-k-1)} \E^{(b-k-1)} \right) \, .
\]
We refer to ${}_{(c,d)}C_{(a,b)}$ as an \emph{$\ell$-shifted Rickard complex},
where $\ell = a-d = c-b$.
\end{definition}
The right-most term in the complex ${}_{(c,d)}C_{(a,b)}$ is either:
\[
\qdeg^{-b(a-d+1)} \tdeg^b
\begin{tikzpicture}[smallnodes,rotate=90,baseline=.1em,scale=.66]
\draw[very thick] (0,.25) to [out=150,in=270] (-.25,1) node[left,xshift=2pt]{$c$};
\draw[very thick] (.5,.5) to (.5,1) node[left,xshift=2pt]{$d$};
\draw[very thick] (0,.25) to  (.5,.5);
\draw[very thick] (0,-.25) to (0,.25);
\draw[very thick] (.5,-.5) to [out=30,in=330] node[above=-2pt]{$b$} (.5,.5);
\draw[dotted] (0,-.25) to  (.5,-.5);
\draw[very thick] (.5,-1) node[right,xshift=-2pt]{$b$} to (.5,-.5);
\draw[very thick] (-.25,-1)node[right,xshift=-2pt]{$a$} to [out=90,in=210] (0,-.25);
\end{tikzpicture} 
\quad(\text{if $b\leq d$}) \, ,\quad \text{ or } \quad 
\qdeg^{-d(a-d+1)} \tdeg^d
\begin{tikzpicture}[smallnodes,rotate=90,baseline=.1em,scale=.66]
\draw[very thick] (0,.25) to [out=150,in=270] (-.25,1) node[left,xshift=2pt]{$c$};
\draw[very thick] (.5,.5) to (.5,1) node[left,xshift=2pt]{$d$};
\draw[dotted] (0,.25) to  (.5,.5);
\draw[very thick] (0,-.25) to (0,.25);
\draw[very thick] (.5,-.5) to [out=30,in=330] node[above=-2pt]{$d$} (.5,.5);
\draw[very thick] (0,-.25) to  (.5,-.5);
\draw[very thick] (.5,-1) node[right,xshift=-2pt]{$b$} to (.5,-.5);
\draw[very thick] (-.25,-1)node[right,xshift=-2pt]{$a$} to [out=90,in=210] (0,-.25);
\end{tikzpicture}  \quad(\text{if $d\leq b$}).
\]

\begin{remark}
The usual Rickard complex is the unshifted case ${}_{(b,a)}C_{(a,b)}$.
In subsequent sections, 
we will be especially interested in the case ${}_{(a,b)}C_{(a,b)}$.
\end{remark}

Via Convention \ref{conv:Alph}, there is an algebra homomorphism
\[
\Sym(\leftX_1|\leftX_2| \rightX_1| \rightX_2) \to 
	Z(\End_{\CS(\SSBim)}({}_{(c,d)}C_{(a,b)}))
\]
for all $a,b,c,d \geq 0$. 
In the special case of the (unshifted) Rickard complex $C_{a,b} = {}_{(b,a)}C_{(a,b)}$, 
\cite[Proposition 5.7]{RW} shows that, for any symmetric function $f \in \Lambda$, 
$f(\leftX_2) \sim f(\rightX_1)$. 
Equivalently, by Lemma \ref{lemma:somerelations1}, 
the action of $h_{r+1}(\leftX_2 - \rightX_1)$ is null-homotopic for all $r \geq 0$.
We now generalize this fact to the shifted Rickard complexes.

\begin{lemma}\label{lemma:dot sliding} 
The action of $h_{a-d+r+1}(\X_2-\X_1')$ on the complex
${}_{(c,d)}C_{(a,b)}$ is null-homotopic for all $r\geq 0$.  
In particular, if $a<d$ then ${}_{(c,d)}C_{(a,b)}\simeq 0$.
\end{lemma}
\begin{proof}
Consider the homotopies
$\Theta_{r+1} \in \End_{\CS(\SSBim)} \big( {}_{(c,d)}C_{(a,b)} \big)$
that are given as the direct sum of the maps
\[
(-1)^{a-d+k} \chi_r^{-} \colon 
\qdeg^{-k(a-d+1)} \tdeg^k \, \F^{(d-k)} \E^{(b-k)} \to 
\qdeg^{(1-k)(a-d+1)} \tdeg^{k-1} \, \F^{(d-k+1)} \E^{(b-k+1)} \, .
\]
Note that $\wt(\Theta_{r+1}) = \qdeg^{2(a-d+r+1)}\tdeg^{-1}$.
The component of $[\d_C,\Theta_{r+1}]$ in $\tdeg$-degree $k$ is
\begin{align*}
(-1)^{a-d+k} \chi^+_0\circ \chi^-_{r} + (-1)^{a-d+k+1} \chi^-_{r}\circ \chi^+_0
&=
\begin{tikzpicture}[anchorbase,smallnodes]
	\draw[CQG,ultra thick,<-] (0,-.75) node[below=-2pt]{$d{-}k$} to (0,.75);
	\draw[CQG,ultra thick,->] (.75,-.75) node[below=-2pt]{$b{-}k$} to (.75,.75);
	\draw[CQG,thick, directed=.55] (.75,.125) to [out=90,in=90] (0,.125);
	\draw[CQG,thick, directed=.75] (0,-.125) to [out=270,in=270] 
		node[black,yshift=-1pt]{\normalsize$\bullet$} node[black,below]{$r$} (.75,-.125);
	\node at (1.25,.5) {$(a,b)$};
\end{tikzpicture} \!\!\!
+
\begin{tikzpicture}[anchorbase,smallnodes]
	\draw[CQG,ultra thick,<-] (0,-.75) node[below=-2pt]{$d{-}k$} to (0,.75);
	\draw[CQG,ultra thick,->] (.75,-.75) node[below=-2pt]{$b{-}k$} to (.75,.75);
	\draw[CQG,thick, directed=.55] (.75,-.5) to [out=90,in=90] (0,-.5);
	\draw[CQG,thick, directed=.75] (0,.5) to [out=270,in=270] 
		node[black,yshift=-1pt]{\normalsize$\bullet$} node[below,black]{$r$} (.75,.5);
	\node at (1.25,.5) {$(a,b)$};
\end{tikzpicture} \\
&=
\sum_{\substack{p+q+s= \\ a-d+r+1}}
\begin{tikzpicture}[anchorbase,smallnodes]
	\draw[CQG,ultra thick,<-] (0,-.75) node[below=-2pt]{$d{-}k$} to node[black]{$\CQGbox{h_p}$} (0,.75);
	\draw[CQG,ultra thick,->] (1.5,-.75) node[below=-2pt]{$b{-}k$} to 
		node[black]{$\CQGbox{h_q}$} (1.5,.75);
	\node at (2.25,.5) {$(a,b)$};
	\draw[CQG,thick,directed=.25] (1,0) arc (0:361:.25) 
		node[black,pos=.75]{$\bullet$} node[below,black,pos=.75]{\scriptsize$\spadesuit{+}s$};
\end{tikzpicture} \, .
\end{align*}
Here we have used (a reflection of) the ``square flop'' relation in \cite[Lemma 4.6.4]{KLMS}. 
By \eqref{eq:bubble}, the bubble on the right-hand side above is equal to the
endomorphism $h_s(\B-\Fr)$; here we use Convention \ref{conv:Alph}.
The result now follows since this gives
\begin{align*}
(\chi^+_0\circ (-1)^{a-d+k} \chi^-_{r-1} + (-1)^{a-d+k+1}\chi^-_{r-1}\circ 
\chi^+_0)|_{\F^{(d-k)} \E^{(b-k)}}
&= \sum_{\substack{p+q+s= \\ a-d+r+1}} h_p(\leftM)h_s(\B-\Fr)h_q(\rightM) \\
&= h_{a-d+r+1}((\leftM + \B) - (\Fr - \rightM)) \\
&= h_{a-d+r+1}(\leftX_2 - \rightX_1) \, . \qedhere
\end{align*}
\end{proof}

We now arrive at the topological interpretation of ${}_{(b,a)}C_{(a,b)}$.

\begin{proposition}\label{prop:topological shifted rickard}
For all integers $a, b,c,d\geq 0$ with $a+b=c+d$ we have a homotopy equivalence 
\[
{}_{(c,d)}C_{(a,b)} 
\simeq 
\left\llbracket
\begin{tikzpicture}[scale=.4,smallnodes,anchorbase,rotate=270]
\draw[very thick] (1,-1) to [out=150,in=270] (0,1) to (0,2) node[right=-2pt]{$b$}; 
\draw[line width=5pt,color=white] (0,-2) to (0,-1) to [out=90,in=210] (1,1);
\draw[very thick] (0,-2) node[left=-2pt]{$d$} to (0,-1) to [out=90,in=210] (1,1);
\draw[very thick] (1,1) to (1,2) node[right=-2pt]{$a$};
\draw[very thick] (1,-2) node[left=-2pt]{$c$} to (1,-1); 
\draw[very thick] (1,-1) to [out=30,in=330] node[below=-1pt]{$a{-}d$} (1,1); 
\end{tikzpicture}
\right\rrbracket \, .
\]
This remains valid even when $a<d$, 
provided we interpret the right-hand side as zero.
\end{proposition}
\begin{proof}
If $a<d$, then contractibility of ${}_{(c,d)}C_{(a,b)}$ was established in Lemma
\ref{lemma:dot sliding}.  If $a\geq d$, then using Reidemeister II, fork-sliding
\eqref{eq:forkslide}, and twist-zipper \eqref{eq:twistzipper} moves, we have
\[ 
\left\llbracket
\begin{tikzpicture}[scale=.4,smallnodes,anchorbase,rotate=270]
\draw[very thick] (1,-1) to [out=150,in=270] (0,1) to (0,2) node[right=-2pt]{$b$}; 
\draw[line width=5pt,color=white] (0,-2) to (0,-1) to [out=90,in=210] (1,1);
\draw[very thick] (0,-2) node[left=-2pt]{$d$} to (0,-1) to [out=90,in=210] (1,1);
\draw[very thick] (1,1) to (1,2) node[right=-2pt]{$a$};
\draw[very thick] (1,-2) node[left=-2pt]{$c$} to (1,-1); 
\draw[very thick] (1,-1) to [out=30,in=330] node[below=-1pt]{$a{-}d$} (1,1); 
\end{tikzpicture}
\right\rrbracket
\simeq
\qdeg^{-b(a-d)}
\left\llbracket
\begin{tikzpicture}[scale=.5,smallnodes,anchorbase,rotate=270]
\draw[very thick] (1,-2) node[left]{$c$} to (1,0) \pu (0,2) node[right]{$b$};
\draw[line width=5pt,color=white] (0,0) \pu (1,2);
\draw[very thick] (0,-2) node[left]{$d$} to (0,0) \pu (1,2) node[right]{$a$};
\draw[very thick] (1,-1.5) to (.5,-1)node[right,xshift=-2pt,yshift=-1pt]{$a{-}d$} to  (0,-.5);
\end{tikzpicture}
\right\rrbracket =  \qdeg^{-b(a-d)} \E^{(a-d)}\hComp C_{a,b}.
\] 
The homotopy equivalence 
$
\E^{(a-d)}\hComp C_{a,b} \simeq  q^{b(a-d)} {}_{(c,d)}C_{(a,b)}
$
is proved in Lemma \ref{lem:Cautis+} below.
\end{proof}

\begin{lem}
\label{lem:Cautis+}
We have
\[
\E^{(\ell)}\hComp C_{a,b} \simeq  q^{b\ell} {}_{(b+\ell,a-\ell)}C_{(a,b)}
\]
for all integers $a,b,\ell\geq 0$. 
\end{lem}

If $a\leq b$ or $a \geq b+\ell$, this follows from \cite[Proposition 4.5]{Cautis}. 
In our setting of $\SSBim$ 
(as opposed to the setting of an arbitrary integrable $\USd(\slnn{2})$ 
representation from \cite{Cautis}), 
the proof strategy of \cite[Proposition 4.5]{Cautis} carries over 
to give a uniform proof with no assumptions other than $a,b,\ell\geq 0$.
Note that exactly one (additional) step here 
(the observation that $X_{-1} = 0$ below) 
uses that we are working in $\SSBim$.

\begin{proof}
We proceed by induction on $\ell$. 
The case $\ell=0$ case holds trivially.
Thus, suppose we have established the result for some fixed $\ell \geq 0$. 
Set $c:=b+\ell$ and $d:=a-\ell$, so $\ell = a-d=c-b$.
We begin by computing $\E\hComp  {}_{(c,d)}C_{(a,b)}$ on the level of chain groups.
Note that
\[
{}_{(c,d)}C_{(a,b)} = \bigoplus_{k \geq 0}
	\qdeg^{-k(\ell+1)} \tdeg^{k} \, \F^{(d-k)}\E^{(b-k)}\oone_{a,b}
\]
where we interpret $\F^{(m)} = 0 = \E^{(m)}$ when $m <0$.
For $k \geq 0$, we thus have
\begin{align*}
\qdeg^{-k(\ell+1)} \, \E \F^{(d-k)}\E^{(b-k)}\oone_{a,b}
&\cong \qdeg^{-k(\ell+1)}\left( \F^{(d-k)} \E\E^{(b-k)}\oone_{a,b} 
	\oplus  [b-k+\ell+1] \, \F^{(d-k-1)}\E^{(b-k)}\right) \\
&\cong \qdeg^{-k(\ell+1)}\left([b-k+1] \, \F^{(d-k)} \E^{(b-k+1)}\oone_{a,b}
	\oplus[b-k+1+\ell] \, \F^{(d-k-1)}\E^{(b-k)}\right)\\
 &\cong X_{k-1}\oplus X_{k} \oplus Y_k
\end{align*}
where\footnote{Here we use the quantum integer identity
$[b-k+1+\ell] = q^{b-k}[1+\ell] + q^{-(1+\ell)}[b-k]$.} 
we set
\[
X_k := \qdeg^{-(k+1)(1+\ell)}[b-k] \, \F^{(d-k-1)} \E^{(b-k)} \, , \quad 
Y_k:=\qdeg^{b-k(2+\ell)}[1+\ell] \, \F^{(d-k-1)} \E^{(b-k)}\oone_{a,b} \, .
\]
Note that $X_{-1}=0$ since $\E^{(b+1)}\oone_{a,b}=0$.
We thus have an isomorphism
\[
\E \hComp  {}_{(c,d)}C_{(a,b)} \cong \bigoplus_{k\geq 0} \tdeg^k  (X_{k-1}\oplus X_{k} \oplus Y_k) =: M
\]
for some differential $\d_M$ on $M$.

Applying the $p=0,1,2$ cases of Corollary \ref{cor:HomForSR}, 
we find that the components 
$\d_M \colon X_{k-1}\oplus X_{k} \oplus Y_k \to X_{k}\oplus X_{k+1} \oplus Y_{k+1}$
take the form
\[
\begin{pmatrix}
\ast & \phi & \ast \\
\ast & \ast & \ast \\
0 & 0 & \ast
\end{pmatrix}
\]
where $\phi$ is upper triangular 
with multiples of the identity on the diagonal. 
The zeros in the bottom left tell us that
\[
\bigoplus_{k\geq 0}\tdeg^k(X_{k-1}\oplus X_{k})
\]
is a subcomplex of $M$, with differential 
$\left(\begin{smallmatrix}\ast &\phi \\ \ast&\ast\end{smallmatrix}\right)$.   
Moreover, an explicit computation (e.g. using the extended graphical calculus
from \cite{KLMS}) shows that the diagonal entries of $\phi$ are non-zero, 
hence $\phi$ is an isomorphism. Successive Gaussian
elimination homotopies show that this subcomplex is contractible, hence
\[
\E\hComp  {}_{(c,d)}C_{(a,b)}  \simeq \bigoplus_{k\geq 0} 
\qdeg^{b-k(2+\ell)}[\ell+1]\tdeg^k \F^{(d-k-1)} \E^{(b-k)}\oone_{a,b}
\]
for some differential. 
The ``trick'' used in the proof of \cite[Proposition 4.5]{Cautis} now applies
\emph{mutatis mutandis}, showing that $\E \hComp  {}_{(c,d)}C_{(a,b)}$ is homotopy
equivalent to $[\ell+1]$ copies of a complex of the form
\[
N := \bigoplus_{k\geq 0} \qdeg^{b-k(\ell+2)}\tdeg^k \, \F^{(d-1-k)} \E^{(b-k)}\oone_{a,b} 
\]
for some differential.

Now, by induction, we have that 
${}_{(c,d)}C_{(a,b)} \simeq \qdeg^{-b \ell} \, \E^{(\ell)}C_{a,b}$, hence
$\E \hComp {}_{(c,d)}C_{(a,b)}  \simeq \qdeg^{-b\ell}[\ell+1] \, \E^{(\ell+1)} C_{a,b}$.
Since $\E^{(\ell+1)}$ is indecomposable and $C_{a,b}$ is
invertible, the complex $\E^{(\ell+1)}C_{a,b}$ is indecomposable. 
The equivalence $[\ell+1]\E^{(\ell+1)}C_{a,b} \simeq \qdeg^{b\ell}[\ell+1]N$ now implies
that $\E^{(\ell+1)}C_{a,b} \simeq \qdeg^{b\ell} N$, so the latter is indecomposable. 
In particular, all differentials in $N$ are non-zero. 
Corollary \ref{cor:Choms} implies that the space of ($\qdeg$-degree zero) 
maps between consecutive terms in $N$ is one-dimensional, 
thus $N\simeq \qdeg^{b} {}_{(c+1,d-1)}C_{(a,b)}$, which completes the proof.
\end{proof}

We conclude this section by establishing a technical result that is needed below. 
It shows that Lemma \ref{lemma:dot sliding} uniquely characterizes the 
Rickard complex $C_{a,b}$ (and its inverse $C_{a,b}^\vee$) amongst 
complexes having the same underlying bigraded bimodule.

\begin{prop}\label{prop:UniqueYCrossing} 
Let $X := \bigoplus_k \qdeg^{-k} \tdeg^k C_{a,b}^k$. 
Suppose $X$ is equipped with a differential $\d_X$ with
respect to which $h_{r+1}(\leftX_2 - \rightX_1)$ is null-homotopic for some $r\geq 0$, 
then $(X,\d_X)$ is isomorphic to $C_{a,b}$. 
The analogous statement for $C_{a,b}^\vee$ holds as well.
\end{prop}
\begin{proof}
We only consider $C_{a,b}$ and assume that $a \geq b$, 
since the other cases are similar. 
Further, suppose that $b>0$ since otherwise the result holds trivially.
Proposition \ref{cor:Choms} implies that
\[
\d_X|_{C_{a,b}^k} = c_k \cdot \chi^+_0.
\]
for some scalars $c_k \in \K$, 
and that $(X,\d_X)\cong C_{a,b}$ if and only if 
$c_k\neq 0$ for all $0\leq k\leq b-1$.
Let $f=h_{r+1}$ and observe that 
$0 \neq f(\leftX_2 - \rightX_1) \in Z(\End_{\CS(\SSBim)}(X))$.
By hypothesis, there exists 
$\eta\in \End_{\CS(\SSBim)}(X)$ so that $[\d_X,\eta] = f(\leftX_2 - \rightX_1)$.

Now suppose that $(X,\d_X)\ncong C_{a,b}$, 
thus $\d_X|_{C_{a,b}^k}=0$ for some $0 \leq k \leq b-1$.
Since $[\d_X, \eta] = f(\leftX_2 - \rightX_1)$, 
this implies that
$ f(\leftX_2 - \rightX_1) |_{C_{a,b}^k} =  c_{k-1} \cdot \chi^+_0\circ \eta |_{C_{a,b}^k}$.
The equality $(\chi^+_0)^2=0$ then implies that 
$\chi^+_0 \circ f(\X_2-\X_1')=0$ on ${C_{a,b}^k}$.
Since $f(\leftX_2 - \rightX_1)$ is central, the composition
\[
C_{a,b}^b 
\xrightarrow{\begin{tikzpicture}[anchorbase,smallnodes]
	\draw[CQG,ultra thick,<-] (0,-1.25) node[below]{$a{-}b$} to (0,0);
	\draw[CQG,ultra thick,->] (0,-.5) to [out=270,in=270] (.75,-.5) 
		to (.75,0) node[above,yshift=-2pt]{$b{-}k$};
\end{tikzpicture}}
C_{a,b}^k
\xrightarrow{\quad \chi^+_0 \quad}
C_{a,b}^{k+1}
\xrightarrow{\begin{tikzpicture}[anchorbase,smallnodes,yscale=-1]
	\draw[CQG,ultra thick,->] (0,-1.25) node[above,yshift=-2pt]{$a{-}b$} to (0,0);
	\draw[CQG,ultra thick] (0,-.5) to [out=270,in=270] (.75,-.5)
		to node[pos=.3,black]{$\CQGbox{e_{b{-}k{-}1}}$} (.75,0) node[below,yshift=2pt]{$b{-}k$};
\end{tikzpicture}}
C_{a,b}^b 
\]
is thus annihilated by $f(\leftX_2 - \rightX_1)$ as well. 
On the other hand this composition equals
\begin{equation}\label{eq:zerodiv}
(-1)^k
\begin{tikzpicture}[anchorbase,smallnodes]
	\draw[CQG,ultra thick,<-] (0,-1.25) node[below]{$a{-}b$} to (0,1.25);
	\draw[CQG,ultra thick,directed=.85] (0,-.5) to [out=270,in=270] (.75,-.5) node[right]{$b{-}k$} 
		to node[pos=.8,black]{$\CQGbox{e_{b{-}k{-}1}}$} (.75,.5) to [out=90,in=90] (0,.5);
	\draw[CQG,thick, directed=.55] (.75,-.25) to [out=90,in=90] (0,-.25);
	\node at (1.25,.75) {$(a,b)$};
\end{tikzpicture}
=
(-1)^{b-1}\begin{tikzpicture}[anchorbase,smallnodes]
	\draw[CQG,ultra thick,<-] (0,-1.25) node[below]{$a{-}b$} to (0,1.25);
	\draw[CQG,ultra thick,directed=.5] (0,-.5) to [out=270,in=270] (.75,-.5) node[right]{$b{-}k$} 
		to (.75,.5) to [out=90,in=90] (0,.5);
	\node at (1.25,.75) {$(a,b)$};
\end{tikzpicture}
=
(-1)^{b-1}\sum_{\alpha,\gamma} c_{\alpha,\gamma}^{(b-k)^{b-k}}
\begin{tikzpicture}[anchorbase,smallnodes]
	\draw[CQG,ultra thick,<-] (0,-1) node[below]{$a{-}b$} to 
		node[black]{$\CQGbox{\mathfrak{s}_\alpha}$} (0,1);
	\draw[CQG,ultra thick,directed=.25] (1.5,0) arc (0:361:0.5) node[right]{$b-k$}
		node[black,pos=.75]{$\CQGbbox{\mathfrak{s}_\gamma^\spadesuit}$};
\end{tikzpicture}.
\end{equation}
By \eqref{eq:thickbubble}, the thick bubble evaluates to
$(-1)^{(b-k)(b-k-1)/2} \mathfrak{s}_{\gamma}(\rightX_2-\rightX_1)$ 
thus
\[
\eqref{eq:zerodiv} = \pm 
\sum_{\alpha,\gamma}  c_{\alpha,\gamma}^{(b-k)^{b-k}}
\mathfrak{s}_{\alpha}(\X_2-\X'_2)\mathfrak{s}_{\gamma}(\X'_2-\X'_1)  
= 
\pm 
\mathfrak{s}_{(b-k)^{b-k}}(\X_2-\X'_1) \neq 0.
\]
This endomorphism of $C_{a,b}^b $ is annihilated by $f(\X_2-\X_1')$, 
contradicting the fact that $\End_{\SSBim}(C_{a,b}^b)$
contains no zero divisors. 
To see the latter, note that
\[
C_{a,b}^b =
\begin{tikzpicture}[smallnodes,rotate=90,baseline=.1em,scale=.66]
\draw[very thick] (0,.25) to [out=150,in=270] (-.25,1) node[left,xshift=2pt]{$b$};
\draw[very thick] (.5,.5) to (.5,1) node[left,xshift=2pt]{$a$};
\draw[very thick] (0,.25) to  (.5,.5);
\draw[very thick] (0,-.25) to (0,.25);
\draw[very thick] (.5,-.5) to [out=30,in=330] node[above=-2pt]{$b$} (.5,.5);
\draw[dotted] (0,-.25) to  (.5,-.5);
\draw[very thick] (.5,-1) node[right,xshift=-2pt]{$b$} to (.5,-.5);
\draw[very thick] (-.25,-1)node[right,xshift=-2pt]{$a$} to [out=90,in=210] (0,-.25);
\end{tikzpicture} 
\]
 is a quotient of its incoming and outgoing edge rings. 
Thus, $C_{a,b}^b $ is a cyclic bimodule and its algebra of endomorphism is
isomorphic to $C_{a,b}^b  \cong \Sym(\X_2|\M|\X_1')$, 
which has no zero divisors.
\end{proof}

\section{The colored skein relation}
\label{sec:coloredskein}
The colored skein relation (Theorem \ref{thm:coloredskein} below) 
asserts that there exists a one-sided twisted complex
constructed from the complexes of ``threaded digons''
\[
\left\llbracket
\begin{tikzpicture}[scale=.4,smallnodes,rotate=90,anchorbase]
\draw[very thick] (1,-1) to [out=150,in=270] (0,0); 
\draw[line width=5pt,color=white] (0,-2) to [out=90,in=270] (.5,0) to [out=90,in=270] (0,2);
\draw[very thick] (0,-2) node[right=-2pt]{$a$} to [out=90,in=270] (.5,0) 
	to [out=90,in=270] (0,2) node[left=-2pt]{$a$};
\draw[very thick] (1,1) to (1,2) node[left=-2pt]{$b$};
\draw[line width=5pt,color=white] (0,0) to [out=90,in=210] (1,1); 
\draw[very thick] (0,0) to [out=90,in=210] (1,1); 
\draw[very thick] (1,-2) node[right=-2pt]{$b$} to (1,-1); 
\draw[very thick] (1,-1) to [out=30,in=330] node[above=-2pt]{$s$} (1,1); 
\end{tikzpicture}
\right\rrbracket
\]
for $0 \leq s \leq b$ that is homotopy equivalent to a certain 
Koszul complex constructed from the complexes
\[
\left\llbracket
\begin{tikzpicture}[scale=.4,smallnodes,anchorbase,rotate=270]
\draw[very thick] (1,-1) to [out=150,in=270] (0,1) to (0,2) node[right=-2pt]{$b$}; 
\draw[line width=5pt,color=white] (0,-2) to (0,-1) to [out=90,in=210] (1,1);
\draw[very thick] (0,-2) node[left=-2pt]{$b$} to (0,-1) to [out=90,in=210] (1,1);
\draw[very thick] (1,1) to (1,2) node[right=-2pt]{$a$};
\draw[very thick] (1,-2) node[left=-2pt]{$a$} to (1,-1); 
\draw[very thick] (1,-1) to [out=30,in=330] node[below=-1pt]{$a{-}b$} (1,1); 
\end{tikzpicture}
\right\rrbracket \, .
\]

This section is organized as follows.  
In \S \ref{ss:rhs}, we develop just enough background to
precisely state the colored skein relation.  
In \S \ref{ss:filtration}, we give an explicit algebraic model
for the right-hand side of the skein relation 
and construct a filtration thereof. 
The subquotients with respect to this filtration will be denoted by
$\MCCSmin_{a,b}^s$ for the duration.
In \S \ref{ss:ft computation}, we show that
\[
\MCCSmin_{a,b}^0 \simeq 
\left\llbracket
 \begin{tikzpicture}[scale=.5,smallnodes,anchorbase,rotate=90]
 \draw[very thick] (1,0) to [out=90,in=270] (0,1.5);
 \draw[line width=5pt,color=white] (1,-1.5) to [out=90,in=270] (0,0) 
to [out=90,in=270] (1,1.5);
 \draw[very thick] (1,-1.5) node[right=-2pt]{$b$} to [out=90,in=270] (0,0) 
to [out=90,in=270] (1,1.5);
 \draw[line width=5pt,color=white] (0,-1.5) to [out=90,in=270] (1,0);
 \draw[very thick] (0,-1.5) node[right=-2pt]{$a$} to [out=90,in=270] (1,0);
 \end{tikzpicture}
\right\rrbracket \, .
\]
This equivalence proves (a version of) \cite[Conjecture 1.3]{BH}. 
The proof of our colored skein relation is completed in \S \ref{ss:colored skein}; 
the main ingredient is an isomorphism
\[
\MCCSmin_{a,b}^s \cong
\begin{tikzpicture}[scale=.45,smallnodes,rotate=90,anchorbase]
	\draw[very thick] (1,2.25) to (1,3) node[left=-2pt]{$b$};
	\draw[very thick] (1,-3) node[right=-2pt]{$b$} to (1,-2.25); 
	\draw[very thick] (1,-2.25) to [out=30, in=270] (1.75,-1.6) 
		to (1.75,0) node[below=-2pt]{$s$}to (1.75,1.6)  to  [out=90,in=330] (1,2.25); 
	\node[yshift=-2pt] at (0,0) {$\MCCSmin_{a,b{-}s}^0$};
	\draw[very thick] (.75,1.6) rectangle (-.75,-1.6);
	\draw[very thick] (1,-2.25) to [out=150,in=270] (.25,-1.6);
	\draw[very thick] (1,2.25) to [out=210,in=90] (.25,1.6);
	\draw[very thick] (-.25,3) node[left=-2pt]{$a$} to (-.25,1.6);
	\draw[very thick] (-.25,-3) node[right=-2pt]{$a$} to (-.25,-1.6);
\end{tikzpicture} \, .
\]

\subsection{Statement of the colored skein relation}\label{ss:rhs}

For the duration, fix integers $a,b\geq 0$ and let $\CS_{a,b} := \CS({}_{a,b}\SSBim_{a,b})$. 
For $X \in \CS_{a,b}$, we will use the following conventions for the boundary alphabets
\[
\begin{tikzpicture}[scale=.5,smallnodes,anchorbase]
	\draw[very thick] (-1.5, .375) node[left=-2pt]{$\leftX_2$} to (-.75,.375);
	\draw[very thick] (-1.5, -.375) node[left=-2pt]{$\leftX_1$} to (-.75, -.375);
	\node[yshift=-2pt] at (0,0) {\normalsize$X$};
	\draw[very thick] (.75,.75) rectangle (-.75,-.75);
	\draw[very thick] (1.5, .375) node[right=-2pt]{$\rightX_2$} to (.75,.375);
	\draw[very thick] (1.5, -.375) node[right=-2pt]{$\rightX_1$} to (.75, -.375);
\end{tikzpicture} \, .
\]
Note that we have an algebra homomorphism
$\Sym(\leftX_1 | \leftX_2 | \rightX_1 | \rightX_2) \to Z(\End_{\CS_{a,b}}(X))$.

The ``right-hand side'' of our skein relation involves the construction of
Koszul complexes, which we now recall.

\begin{definition}\label{def:Koszul cx} 
For each $X \in \CS_{a,b}$, 
let $K(X)$ denote the Koszul complex associated to the action of 
$h_1(\leftX_2 - \rightX_2),\ldots,h_b(\leftX_2-\rightX_2)$ on $X$. 
Explicitly, we consider the bigraded $\K$-vector space $\largewedge[\xi_1,\dots,\xi_b]$ 
in which the $\xi_i$ are exterior variables with $\wt(\xi_i)=\qdeg^{2i} \tdeg\inv$
and define bimodules
\[
K(X) := \tw_{\sum_{i=1}^b h_i(\leftX_2-\rightX_2)\otimes \xi_i^\ast}\left(X\otimes \largewedge[\xi_1,\dots,\xi_b]\right) \, .
\]
Here, $\xi_i^\ast$ is the endomorphism (in fact, derivation) of
$\largewedge[\xi_1,\dots,\xi_b]$ with $\wt(\xi_i^{\ast})=\qdeg^{-2i} \tdeg^{1}$
defined by
\[
\xi_i^\ast(\xi_i)=1
\, , \quad
\xi_i^\ast(\xi_j)=0\quad(i\neq j)
\, , \quad
\xi_i^\ast(\eta \nu)=\xi_i^\ast(\eta) \nu + (-1)^{|\eta|} \eta \xi_i^\ast(\nu) \, .
\]
\end{definition}

\begin{remark}\label{rem:sign}
Before turning on the Koszul differential we have
\[
X\otimes \largewedge[\xi_1,\dots,\xi_b] = \bigoplus_{l=0}^b \bigoplus_{i_1<\cdots<i_l} X\otimes \xi_{i_1}\cdots\xi_{i_l},
\]
where each $X\otimes \xi_{i_1}\cdots\xi_{i_l}$ 
denotes a copy of $X$ (appropriately shifted).
The usual Koszul sign conventions tell us that the differential on 
$X\otimes \xi_{i_1}\cdots\xi_{i_l}$ coincides with $\d_X$ \emph{with no sign}, 
since the monomial in $\xi$'s appears on the right. 
\end{remark}

\begin{lemma}\label{lemma:K is dg fun}
The assignment $X\mapsto K(X)$ is a dg functor.
\end{lemma}
\begin{proof}
This follows since we can describe
\[
K(X) \cong X \otimes_{\Sym(\leftX_2 | \rightX_2)} 
\tw_{\sum_{i=1}^b h_i(\leftX_2-\rightX_2)\otimes \xi_i^\ast} 
\big( \Sym(\leftX_2 | \rightX_2) \otimes \largewedge[\xi_1,\ldots,\xi_b] \big) \, . \qedhere
\]
\end{proof}

Using this construction, we can now state our main theorem.

\begin{thm}[Colored skein relation]\label{thm:coloredskein} For each pair of
integers $a,b\geq 0$ there is homotopy equivalence in $\CS_{a,b}$ of the form
\begin{equation}\label{eq:skeinrel}
\tw_{D^c}\left( \bigoplus_{s=0}^b \qdeg^{s(b-1)} \tdeg^s 
\left\llbracket
\begin{tikzpicture}[scale=.4,smallnodes,rotate=90,anchorbase]
\draw[very thick] (1,-1) to [out=150,in=270] (0,0); 
\draw[line width=5pt,color=white] (0,-2) to [out=90,in=270] (.5,0) to [out=90,in=270] (0,2);
\draw[very thick] (0,-2) node[right=-2pt]{$a$} to [out=90,in=270] (.5,0) 
	to [out=90,in=270] (0,2) node[left=-2pt]{$a$};
\draw[very thick] (1,1) to (1,2) node[left=-2pt]{$b$};
\draw[line width=5pt,color=white] (0,0) to [out=90,in=210] (1,1); 
\draw[very thick] (0,0) to [out=90,in=210] (1,1); 
\draw[very thick] (1,-2) node[right=-2pt]{$b$} to (1,-1); 
\draw[very thick] (1,-1) to [out=30,in=330] node[above=-2pt]{$s$} (1,1); 
\end{tikzpicture}
\right\rrbracket 
\right) 
\simeq
\qdeg^{b(a-b-1)}\tdeg^b K\left( 
\left\llbracket
\begin{tikzpicture}[scale=.4,smallnodes,anchorbase,rotate=270]
\draw[very thick] (1,-1) to [out=150,in=270] (0,1) to (0,2) node[right=-2pt]{$b$}; 
\draw[line width=5pt,color=white] (0,-2) to (0,-1) to [out=90,in=210] (1,1);
\draw[very thick] (0,-2) node[left=-2pt]{$b$} to (0,-1) to [out=90,in=210] (1,1);
\draw[very thick] (1,1) to (1,2) node[right=-2pt]{$a$};
\draw[very thick] (1,-2) node[left=-2pt]{$a$} to (1,-1); 
\draw[very thick] (1,-1) to [out=30,in=330] node[below=-1pt]{$a{-}b$} (1,1); 
\end{tikzpicture}
\right\rrbracket
\right)
\end{equation}
in which the twist $D^c$ strictly increases the index $s$.
\end{thm}

The following shorthand will often be useful.

\begin{definition}\label{def:MCS and KMCS}
We will use the following notation for the complexes appearing in \eqref{eq:skeinrel}
\[
\MCCS_{a,b}^s := \left\llbracket
\begin{tikzpicture}[scale=.4,smallnodes,rotate=90,anchorbase]
\draw[very thick] (1,-1) to [out=150,in=270] (0,0); 
\draw[line width=5pt,color=white] (0,-2) to [out=90,in=270] (.5,0) to [out=90,in=270] (0,2);
\draw[very thick] (0,-2) node[right=-2pt]{$a$} to [out=90,in=270] (.5,0) 
	to [out=90,in=270] (0,2) node[left=-2pt]{$a$};
\draw[very thick] (1,1) to (1,2) node[left=-2pt]{$b$};
\draw[line width=5pt,color=white] (0,0) to [out=90,in=210] (1,1); 
\draw[very thick] (0,0) to [out=90,in=210] (1,1); 
\draw[very thick] (1,-2) node[right=-2pt]{$b$} to (1,-1); 
\draw[very thick] (1,-1) to [out=30,in=330] node[above=-2pt]{$s$} (1,1); 
\end{tikzpicture}
\right\rrbracket
\, , \quad
\MCS_{a,b} := \left\llbracket
\begin{tikzpicture}[scale=.4,smallnodes,anchorbase,rotate=270]
\draw[very thick] (1,-1) to [out=150,in=270] (0,1) to (0,2) node[right=-2pt]{$b$}; 
\draw[line width=5pt,color=white] (0,-2) to (0,-1) to [out=90,in=210] (1,1);
\draw[very thick] (0,-2) node[left=-2pt]{$b$} to (0,-1) to [out=90,in=210] (1,1);
\draw[very thick] (1,1) to (1,2) node[right=-2pt]{$a$};
\draw[very thick] (1,-2) node[left=-2pt]{$a$} to (1,-1); 
\draw[very thick] (1,-1) to [out=30,in=330] node[below=-1pt]{$a{-}b$} (1,1); 
\end{tikzpicture}
\right\rrbracket
\] 
(read\footnote{It is best to read this, and the notation, from right-to-left.} 
as ``Merge-Crossing-Crossing-Split" and ``Merge-Crossing-Split").
Additionally, set $\KMCS_{a,b}:=K(\MCS_{a,b})$.
\end{definition}

Using Proposition \ref{prop:topological shifted rickard}, we can give a precise
algebraic model for $\KMCS_{a,b}$.

\begin{definition}
\label{def:MCSmin}
Set $\MCSmin_{a,b}:={}_{(a,b)}C_{(a,b)}$, i.e. diagrammatically:
\begin{equation}
\label{eq:defMCSmin}
\MCSmin_{a,b}:=
\left(
\begin{tikzpicture}[smallnodes,rotate=90,baseline=.1em,scale=.525]
\draw[very thick] (0,.25) to [out=150,in=270] (-.25,1) node[left,xshift=2pt]{$a$};
\draw[very thick] (.5,.5) to (.5,1) node[left,xshift=2pt]{$b$};
\draw[very thick] (0,.25) to  (.5,.5);
\draw[very thick] (0,-.25) to (0,.25);
\draw[dotted] (.5,-.5) to [out=30,in=330] node[above=-2pt]{$0$} (.5,.5);
\draw[very thick] (0,-.25) to  (.5,-.5);
\draw[very thick] (.5,-1) node[right,xshift=-2pt]{$b$} to (.5,-.5);
\draw[very thick] (-.25,-1)node[right,xshift=-2pt]{$a$} to [out=90,in=210] (0,-.25);
\end{tikzpicture}
\to  
\qdeg^{-(a-b+1)} \tdeg\begin{tikzpicture}[smallnodes,rotate=90,baseline=.1em,scale=.525]
\draw[very thick] (0,.25) to [out=150,in=270] (-.25,1) node[left,xshift=2pt]{$a$};
\draw[very thick] (.5,.5) to (.5,1) node[left,xshift=2pt]{$b$};
\draw[very thick] (0,.25) to  (.5,.5);
\draw[very thick] (0,-.25) to (0,.25);
\draw[very thick] (.5,-.5) to [out=30,in=330] node[above=-2pt]{$1$} (.5,.5);
\draw[very thick] (0,-.25) to  (.5,-.5);
\draw[very thick] (.5,-1) node[right,xshift=-2pt]{$b$} to (.5,-.5);
\draw[very thick] (-.25,-1)node[right,xshift=-2pt]{$a$} to [out=90,in=210] (0,-.25);
\end{tikzpicture}
\to \cdots \to
\qdeg^{-b(a-b+1)} \tdeg^b
\begin{tikzpicture}[smallnodes,rotate=90,baseline=.1em,scale=.525]
\draw[very thick] (0,.25) to [out=150,in=270] (-.25,1) node[left,xshift=2pt]{$a$};
\draw[very thick] (.5,.5) to (.5,1) node[left,xshift=2pt]{$b$};
\draw[dotted] (0,.25) to  (.5,.5);
\draw[very thick] (0,-.25) to (0,.25);
\draw[very thick] (.5,-.5) to [out=30,in=330] node[above=-2pt]{$b$} (.5,.5);
\draw[dotted] (0,-.25) to  (.5,-.5);
\draw[very thick] (.5,-1) node[right,xshift=-2pt]{$b$} to (.5,-.5);
\draw[very thick] (-.25,-1)node[right,xshift=-2pt]{$a$} to [out=90,in=210] (0,-.25);
\end{tikzpicture}
\right) \, .
\end{equation}
Let $\KMCSmin_{a,b}:=K(\MCSmin_{a,b})$.
\end{definition}

The $d=b$ case of Proposition \ref{prop:topological shifted rickard}
gives that
\begin{equation}\label{eq:topological MCS}
\MCSmin_{a,b} \simeq \MCS_{a,b}
\, , \quad
\KMCS_{a,b} \simeq \KMCSmin_{a,b},
\end{equation}
where the second homotopy equivalence follows from the first 
by Lemma~\ref{lemma:K is dg fun}.

We now establish language for discussing $\KMCSmin_{a,b}$ 
and its chain groups. 
Fix $a, b\geq 0$ and consider the bimodules
\[
W_k := 
\begin{tikzpicture}[smallnodes,rotate=90,anchorbase,scale=.625]
	\draw[very thick] (0,.25) to [out=150,in=270] (-.25,1) node[left,xshift=2pt]{$a$};
	\draw[very thick] (.5,.5) to (.5,1) node[left,xshift=2pt]{$b$};
	\draw[very thick] (0,.25) to node[left,xshift=2pt,yshift=-2pt]{$k$} (.5,.5);
	\draw[very thick] (0,-.25) to (0,.25);
	\draw[very thick] (.5,-.5) to [out=30,in=330] (.5,.5);
	\draw[very thick] (0,-.25) to node[right,xshift=-2pt,yshift=-2pt]{$k$} (.5,-.5);
	\draw[very thick] (.5,-1) node[right,xshift=-2pt]{$b$} to (.5,-.5);
	\draw[very thick] (-.25,-1)node[right,xshift=-2pt]{$a$} to [out=90,in=210] (0,-.25);
\end{tikzpicture}
= \F^{(k)}\E^{(k)}\oone_{a,b}
\]
for $0 \leq k \leq b$.  We follow Convention \ref{conv:Alph} in assigning
alphabets to each of the edges in the web depicting these bimodules, namely:
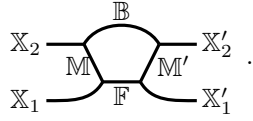
\[
\begin{tikzpicture}[rotate=90,anchorbase]
	\draw[very thick] (0,.25) to [out=150,in=270] (-.25,1) node[left,xshift=2pt]{$\leftX_1$};
	\draw[very thick] (.5,.5) to (.5,1) node[left,xshift=2pt]{$\leftX_2$};
	\draw[very thick] (0,.25) to node[left,yshift=-1pt,xshift=3pt]{$\leftM$} (.5,.5);
	\draw[very thick] (0,-.25) to node[below,yshift=2pt]{$\Fr$} (0,.25);
	\draw[very thick] (.5,-.5) to [out=30,in=330] node[above,yshift=-2pt]{$\B$} (.5,.5);
	\draw[very thick] (0,-.25) to node[right,yshift=-1pt,xshift=-1pt]{$\rightM$} (.5,-.5);
	\draw[very thick] (.5,-1) node[right,xshift=-2pt]{$\rightX_2$} to (.5,-.5);
	\draw[very thick] (-.25,-1)node[right,xshift=-2pt]{$\rightX_1$} to [out=90,in=210] (0,-.25);
\end{tikzpicture} \, .
\]
If we wish to emphasize the index $k$, we will write $\leftM^{(k)}, {\rightM}^{(k)}$, etc. 
In particular, we note that 
\[
|\leftM^{(k)}| = k = |{\rightM}^{(k)}|
\, , \quad 
|\B^{(k)}| = b-k
\, , \quad 
|\Fr^{(k)}| = a+k
\]
while
\[
|\leftX_1| = a = |\rightX_1|
\, , \quad
|\leftX_2| = b = |\rightX_2|
\]
for all $k$.  
The Koszul complex $\KMCSmin_{a,b}$ can be efficiently described as follows.

\begin{proposition}\label{prop:KMCS 1}
We have
\begin{equation}
\KMCSmin_{a,b} = \Big(K(W_b)\xrightarrow{\d^H} \qdeg^{a-b+1}\tdeg K(W_{b-1})\xrightarrow{\d^H} \cdots \xrightarrow{\d^H} \qdeg^{b(a-b+1)}\tdeg^b K(W_0)\Big),
\end{equation}
where $\d^H = K(\chi^+_0) \colon K(W_k)\rightarrow K(W_{k-1})$.
\qed
\end{proposition}

The differential internal to each $K(W_k)$ will be denoted $\d^v$, 
and referred to as the \emph{vertical differential}.
The differential $\d^H$ will be referred to as the 
\emph{total horizontal differential}. 
In \S \ref{ss:filtration} below, we introduce an additional ``$s$-grading'' on 
$K(\MCSmin_{a,b})$ and decompose $\d^H$ further as $\d^H=\d^h + \d^c$ where 
$\d^h$ respects the $s$-grading and $\d^c$ strictly increases it. 
These differentials $\d^h$ and $\d^c$ will be called the 
\emph{horizontal differential} and the \emph{connecting differential},
respectively.

\subsection{The \texorpdfstring{$\zeta$}{zeta}-filtration}
\label{ss:filtration}

We now aim to filter the complex $\KMCSmin_{a,b}$ 
and explicitly identify the associated graded complex. 
To do so, we perform a change of basis within 
the exterior algebra tensor factor of each $K(W_k)$, 
i.e. we replace each \emph{column complex} $W_k\otimes
\largewedge[\xi_1,\ldots,\xi_b]$ by an isomorphic Koszul complex.    

\begin{definition}\label{def:zeta}
Let $\zeta_1^{(k)},\ldots,\zeta_b^{(k)}$ be odd variables given by the formula 
\[
\zeta_j^{(k)} := \sum_{i=1}^j (-1)^{i-1} e_{j-i}(\leftM^{(k)}) \otimes \xi_i.
\]
\end{definition}
Given this, equation \eqref{eq:HE2} implies that the formula
\[
\xi_i = \sum_{j=1}^i (-1)^{j-1}h_{i-j}(\leftM^{(k)}) \otimes \zeta_j^{(k)}
\]
recovers the variables $\xi_i$ from the $\zeta_j^{(k)}$.
We now wish to describe $\KMCSmin_{a,b}$ in terms of the $\zeta$-basis.

\begin{lemma}\label{lemma:zeta d2}
Consider the dg algebra $\Sym(\M|\M')\otimes \largewedge[\xi_1,\ldots,\xi_b]$ 
with $\Sym(\M|\M')$-linear derivation defined by $d(\xi_i)  = h_i(\M-\M')$ for all $1\leq i\leq b$. 
The elements $\zeta_j:=\sum_{i=1}^j (-1)^{i-1} e_{j-i}(\M) \otimes \xi_i$ satisfy
\[
d(\zeta_j) = e_j(\M) - e_j(\M').
\]
\end{lemma}
\begin{proof}
This is an immediate consequence of Lemma~\ref{lemma:somerelations1}.
\end{proof}


\begin{prop}\label{prop:diffKMCSexplicit}
We have that $K(W_k) \cong \tw_{\d^v}(W_k\otimes \largewedge[\zeta^{(k)}_1,\ldots,\zeta^{(k)}_b])$ where
\begin{equation}\label{eq:dv}
\d^v = \sum_{i=1}^k (e_{i_j}(\leftM^{(k)})-e_{i_j}({\rightM}^{(k)}))\otimes (\zeta_i^{(k)})^\ast \, .
\end{equation}
With respect to this isomorphism, 
the differential $\d^H\colon K(W_k)\rightarrow K(W_{k-1})$ has a nonzero component
\[
W_k \otimes \zeta^{(k)}_{i_1}\cdots \zeta^{(k)}_{i_r}
\xrightarrow{\d^H}
W_{k-1}\otimes \zeta^{(k-1)}_{j_1}\cdots \zeta^{(k-1)}_{j_r}
\]
if and only if $i_p - j_p\in \{0,1\}$ for all $1\leq p\leq r$.
In that case, it equals $\chi_m^+$ 
where $m=\sum_{p=1}^r (i_p-j_p)$.
\end{prop}
\begin{proof}
The first statement is immediate from Lemma \ref{lemma:zeta d2}. 

For the second, recall that the components of $\d^H$ are described in the $\xi$-basis by
\begin{equation}\label{eq:std map}
\chi^{+}_{0} |_{W_k} \otimes \Id 
=
\CQGsgn{(-1)^{b-k}}
\begin{tikzpicture}[baseline=0em]
	\draw[CQG,ultra thick,<-] (0,-.5) node[below]{\scriptsize$k$} node[above=3.2em]{\scriptsize $k{-}1$} to (0,.7);
	\draw[CQG,ultra thick,->] (.75,-.5) node[below]{\scriptsize$k$} node[above=3.2em]{\scriptsize $k{-}1$}to (.75,.7);
	\draw[CQG,thick, directed=.75] (.75,0) to [out=90,in=90]  (0,0);
\end{tikzpicture}
\otimes \Id \colon 
W_k\otimes \largewedge[\xi_1,\ldots,\xi_b]
\longrightarrow
W_{k-1}\otimes \largewedge[\xi_1,\ldots,\xi_b] \, .
\end{equation}
We now compute these components under 
a basis change to monomials in the variables $\zeta^{(k)}_i$ and
$\zeta^{(k-1)}_i$ in the domain and co-domain, respectively. In the domain, the
requisite basis change is given by maps
\[
W_k\otimes \zeta^{(k)}_{i_1}\cdots \zeta^{(k)}_{i_r} \to \bigoplus_{l_1,\ldots,l_r} W_k \otimes \xi_{l_1}\cdots \xi_{l_r}
\]
with components
\[
(-1)^{l_1+\cdots+l_r-r} \prod_{p=1}^r e_{i_p - l_p}(\M^{(k)}) \, .
\]
Note that these are non-zero only if $1\leq l_p\leq i_p$ for all $1\leq p\leq r$. 
Next, each $W_k \otimes \xi_{l_1}\cdots \xi_{l_r}$ maps to 
$W_{k-1}\otimes \xi_{l_1}\cdots \xi_{l_r}$ 
via $\chi^+_0\otimes \Id$. 
Finally, the basis change in the codomain is given by maps
\[
W_{k-1}\otimes \xi_{l_1}\cdots \xi_{l_r} \to \bigoplus_{j_1,\ldots,j_r} W_{k-1}\otimes \zeta^{(k-1)}_{j_1}\cdots \zeta^{(k-1)}_{j_r}
\]
with components
\[
(-1)^{j_1+\cdots+j_r-r} \prod_{p=1}^r h_{l_p - j_p}(\M^{(k-1)}) \, .
\]
As before, this is non-zero only if $1\leq j_p\leq l_p$ for all $1\leq p \leq r$.
Thus, the component of $\d^H$ from $W_k\otimes \zeta^{(k)}_{i_1}\cdots \zeta^{(k)}_{i_r}$ to 
$W_{k-1}\otimes \zeta^{(k-1)}_{j_1}\cdots \zeta^{(k-1)}_{j_r}$
is:
\[
\begin{aligned}
\CQGsgn{(-1)^{b-k}}
\sum_{l_1,\ldots,l_r} (-1)^{\sum_{p=1}^r l_p - j_p} \; 
\begin{tikzpicture}[baseline=0em]
	\draw[CQG,ultra thick,<-] (0,-1) node[below]{\scriptsize$k$} to node[black,yshift=15]{$\CQGbox{\Pi_{p=1}^r h_{l_p-j_p}}$} 
		node[black,yshift=-20pt]{$\CQGbox{\Pi_{p=1}^r e_{i_p-l_p}}$} (0,1.5) node[above]{\scriptsize$k{-}1$};
	\draw[CQG,ultra thick,->] (1.5,-1) node[below]{\scriptsize$k$} to (1.5,1.5) node[above]{\scriptsize$k{-}1$};
	\draw[CQG,thick, directed=.55] (1.5,0) to [out=90,in=90] (0,0);
\end{tikzpicture}
&=
\CQGsgn{(-1)^{b-k}}
\begin{tikzpicture}[baseline=0em]
	\draw[CQG,ultra thick,<-] (0,-1) node[below]{\scriptsize$k$} to (0,1.2) node[above]{\scriptsize $k{-}1$} ;
	\draw[CQG,ultra thick,->] (3,-1) node[below]{\scriptsize$k$} to (3,1.2) node[above]{\scriptsize $k{-}1$};
	\draw[CQG,thick, directed=.85] (3,0) to [out=90,in=90] node[black,yshift=-10pt]{$\CQGbox{\Pi_{p=1}^r e_{i_p-j_p}}$} (0,0);
\end{tikzpicture} \\
&=
\begin{cases}
\chi_m^+ & \text{if } i_p-j_p\in \{0,1\} \text{ for all } 1\leq p\leq r \\
0 & \text{else}
\end{cases}
\end{aligned}
\]
where here $m=\sum_{p=1}^r (i_p-j_p)$.  This gives the description of $\d^H$ from the statement.
\end{proof}

The fact that $e_i(\M^{(k)})-e_i(\M'^{(k)})$ is zero when $i>k$ 
suggests that we should treat the variables $\zeta^{(k)}_i$
differently according to whether $i\leq k$ or $i>k$.  
The following definition emphasizes this distinction.  

\begin{definition}\label{def:Pkls}
Set
$
P_{k,l,s} :=
\qdeg^{k(a-b+1)-2b} \tdeg^{2b-k} 
W_k
\otimes \largewedge^{l}[\zeta^{(k)}_1,\ldots, \zeta^{(k)}_k]\otimes \largewedge^s[\zeta^{(k)}_{k+1},\ldots, \zeta^{(k)}_b]
$.
\end{definition}

From this point forward, 
we will work with the shifted Koszul complex $\qdeg^{b(a-b-1)}\tdeg^b\KMCSmin_{a,b}$.
The shift is conventional, but will guarantee that a quotient of this complex is 
homotopy equivalent to the Rickard complex of the $(a,b)$-colored full twist braid.
Definition \ref{def:Pkls} now allows us to filter $\qdeg^{b(a-b-1)}\tdeg^b\KMCSmin_{a,b}$ 
as follows.

\begin{proposition}\label{prop:KMCSdiffs}
We have
\[
\qdeg^{b(a-b-1)}\tdeg^b\KMCSmin_{a,b} \cong  
\tw_{\d^v+\d^h+\d^c}\left(\bigoplus_{0\leq l\leq k\leq b-s} P_{k,l,s}\right) \, ,
\]
where $\d^v$, $\d^h$, $\d^c$ are pairwise anti-commuting differentials given as
follows:
\begin{itemize}
\item the \emph{vertical differential} $\d^v \colon P_{k,l,s}\rightarrow P_{k,l-1,s}$
is the direct sum of the Koszul differentials, 
up to sign $(-1)^k$; its component
\[
W_k \otimes \zeta^{(k)}_{i_1}\cdots \zeta^{(k)}_{i_r} \xrightarrow{\d^v}
W_{k}\otimes \zeta^{(k)}_{i_1}\cdots \widehat{\zeta^{(k)}_{i_j}}\cdots
\zeta^{(k)}_{i_r}
\]
is $(-1)^{-k+j-1} (e_{i_j}(\leftM^{(k)})-e_{i_j}({\rightM}^{(k)}))$ if $1\leq i_j\leq k$
(and all other components are zero). 
\item  the \emph{horizontal differential} $\d^h$ and the connecting differential
$\d^c$ are uniquely characterized by $\d^h+\d^c=\d^H$ from Proposition
\ref{prop:diffKMCSexplicit}, together with
\[
\d^h(P_{k,l,s})\subset P_{k-1,l,s} \, , \quad \d^c(P_{k,l,s})\subset P_{k-1,l-1,s+1}.
\]
\end{itemize}
\end{proposition}
That is, $\d^h$ is the part of $\d^H$ which preserves the 
$s$-degree and $\d^c$ is the part of $\d^H$ which increases $s$-degree by $1$.


\begin{remark}
Since each $\zeta^{(k)}_i$ carries cohomological degree $-1$, the object
$P_{k,l,s}$ contributes to the cohomological degree $2b-k-l-s$ part of
$\qdeg^{b(a-b-1)}\tdeg^b\KMCSmin_{a,b}$.
\end{remark}

\begin{proof}[Proof of Proposition \ref{prop:KMCSdiffs}]
By construction, the complex $\qdeg^{b(a-b-1)}\tdeg^b\KMCSmin_{a,b}$ from
Definition \ref{def:MCS and KMCS} is isomorphic to $\bigoplus_{k,l,s}P_{k,l,s}$
with differential $\d^v+\d^H$ as in Proposition \ref{prop:diffKMCSexplicit}.
It is immediate from \eqref{eq:dv} that $\d^v$ maps $P_{k,l,s}$ to $P_{k,l-1,s}$.
It follows from Definition \ref{def:Pkls} and the characterization of the non-zero components 
of $\d^H$ in Proposition \ref{prop:diffKMCSexplicit} that
$\d^H$ maps $P_{k,l,s}$ to $P_{k-1,l,s}\oplus P_{k-1,l-1,s+1}$.
Hence $\d^h$ and $\d^c$ are well-defined.

The desired relations concerning $\d^v,\d^h,\d^c$ follow from taking components
of $(\d^v+\d^h+\d^c)^2=0$ under the trigrading $(l+s,k+s,-s)$.
(This uses the fact that $\d^v,\d^h$, and $\d^c$ have
tridegrees $(-1,0,0),(0,-1,0)$, and $(0,0,-1)$ with respect to this trigrading.) 
\end{proof}

An instructive example of the complex $\qdeg^{b(a-b-1)}\tdeg^b\KMCSmin_{a,b}$
showing the three types of differentials is given in the following.
\begin{exa}
	\label{exa:KMCS} 
We illustrate the complex $\qdeg^{-2}\tdeg^2\KMCSmin_{2,2}$, as well as the subquotients
$P_{\bullet,\bullet,s}=\qdeg^{s} \tdeg^s \MCCSmin^s_{2,2}$ for $0\leq s\leq 2$.
We use the symbol $\cdot$ instead of $\otimes$ to declutter the diagram. 
We also suppress the homological shifts $\tdeg^k$, which are determined by 
placing the underlined term in the top left in homological degree zero 
(and noting that all arrows increase homological degree by one).
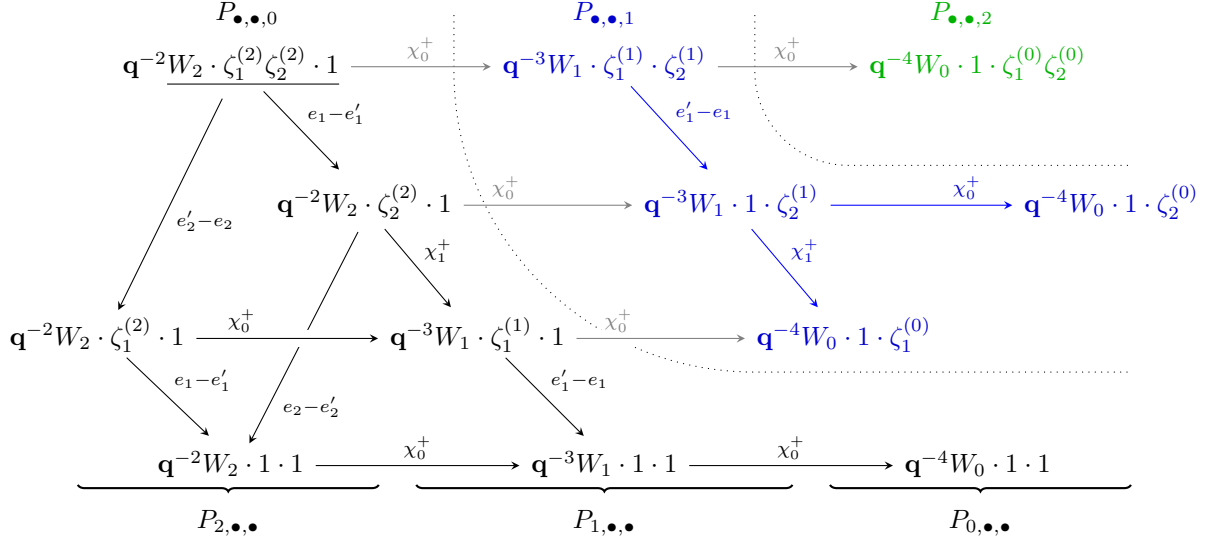
\[\!\!\!\!\!
	\begin{tikzpicture}[anchorbase]
		\draw[dotted]  (2,3.75) to (2,3) to [out=270,in=180] (3.25,1.75) to (7,1.75);
		\draw[dotted] (-2,3.75) to (-2,3) to [out=270,in=180] (2,-1) to (7,-1);
		\node[scale=1] at (-4.75,3.75){$P_{\bullet,\bullet,0}$};
		\node[scale=1] at (0,3.75){$\BLUE{P_{\bullet,\bullet,1}}$};
		\node[scale=1] at (4.75,3.75){$\GREEN{P_{\bullet,\bullet,2}}$};
		\node[scale=1] at (0,0.5){
\begin{tikzcd}[row sep=3em,column sep=-3.2em]
& 
\qdeg^{-2} \underline{W_2\cdot \zeta^{(2)}_1\zeta^{(2)}_2 \cdot 1}
	\arrow[ddl, "e_2'-e_2", shift left] 
	\arrow[dr,"e_1-e_1'", shift left]
	\arrow[rrr,gray, "\GRAY{\chi^+_0}"] & & & 
\BLUE{\qdeg^{-3}W_1\cdot\zeta^{(1)}_1 \cdot \zeta^{(1)}_2}
	\arrow[dr,blue,"\BLUE{e_1'-e_1}", shift left] 
	\arrow[rrr,gray, "\GRAY{\chi^+_0}"] & & & 
\GREEN{\qdeg^{-4}W_0\cdot 1 \cdot \zeta^{(0)}_1\zeta^{(0)}_2}	
 & 
\\
& &
\qdeg^{-2}W_2\cdot \zeta^{(2)}_2 \cdot  1	
	\arrow[ldd, "e_2-e_2'" near end, shift left]
	\arrow[dr,, "\chi^+_1"] 
	\arrow[rrr,gray, "\GRAY{\chi^+_0}" near start] & & & 
\BLUE{\qdeg^{-3}W_1\cdot 1 \cdot \zeta^{(1)}_2}  
	\arrow[dr,blue,"\BLUE{\chi^+_1}"] 
	\arrow[rrr,blue, "\BLUE{\chi^+_0}" near end]  & & & 
\BLUE{\qdeg^{-4}W_0\cdot 1 \cdot \zeta^{(0)}_2}  
\\
\qdeg^{-2}W_2\cdot \zeta^{(2)}_1 \cdot 1	
	\arrow[dr,"e_1-e_1'", shift left]
	\arrow[rrr,crossing over, "\chi^+_0" near start]  & & & 
\qdeg^{-3}W_1\cdot \zeta^{(1)}_1 \cdot 1	
	\arrow[dr,"e_1'-e_1", shift left]
	\arrow[rrr,crossing over,gray, "\GRAY{\chi^+_0}" near start] 
	 & & & 
\BLUE{\qdeg^{-4}W_0\cdot 1 \cdot \zeta^{(0)}_1}	
	 & & 
\\
&
 \qdeg^{-2}W_2 \cdot 1 \cdot 1 
 	\arrow[rrr, "\chi^+_0"]& & & 
\qdeg^{-3}W_1 \cdot 1\cdot 1  
	\arrow[rrr, "\chi^+_0"]& & & 
\qdeg^{-4} W_0 \cdot 1 \cdot 1	 &
\end{tikzcd}
};
\draw[thick,decoration={brace,mirror,raise=0.5cm},decorate] (-7,-2) to (-3,-2);
\node at (-5,-3) {$P_{2,\bullet,\bullet}$};
\draw[thick,decoration={brace,mirror,raise=0.5cm},decorate] (-2.5,-2) to (2.5,-2);
\node at (0,-3) {$P_{1,\bullet,\bullet}$};
\draw[thick,decoration={brace,mirror,raise=0.5cm},decorate] (3,-2) to (7,-2);
\node at (5,-3) {$P_{0,\bullet,\bullet}$};
\end{tikzpicture}
\]

Black and blue horizontal arrows correspond to components of $\d^h$. 
All other black and blue arrows indicate non-zero components of $\d^v$. 
The connecting differential $\d^c$ is depicted by the grey horizontal arrows. 
\end{exa}

We may regard $\qdeg^{b(a-b-1)}\tdeg^b\KMCSmin_{a,b}$ as filtered by $s$-degree, since
the differentials $\d^v$ and $\d^h$ preserve $s$-degree, 
while $\d^c$ increases $s$-degree by one.  The following gives names to the
subquotients with respect to this filtration.

\begin{definition}\label{def:Qs as subquotient}
For each $0\leq s\leq b$, let 
\[
\MCCSmin^s_{a,b}:= 
\qdeg^{-s(b-1)}\tdeg^{-s} \tw_{\d^v+\d^h}
\left(\bigoplus_{0\leq l\leq k\leq b-s}P_{k,l,s}\right) \, .
\]
\end{definition}

Given this, the complex $\qdeg^{b(a-b-1)}\tdeg^b\KMCSmin_{a,b}$ from Definition
\ref{def:MCS and KMCS} can be described as the one-sided twisted complex
\begin{equation}\label{eq:MainThmMin}
\begin{aligned}
	\qdeg^{b(a-b-1)}\tdeg^b\KMCSmin_{a,b}
&=
\left(\MCCSmin^0_{a,b} \xrightarrow{\d^c}
\qdeg^{b-1}\tdeg\MCCSmin^1_{a,b}  
\xrightarrow{\d^c} \cdots  \xrightarrow{\d^c} 
\qdeg^{b(b-1)} \tdeg^{b} \MCCSmin^b_{a,b}\right) \\
&= \tw_{\d^c}
\left(\bigoplus_{s=0}^b \qdeg^{s(b-1)} \tdeg^s \MCCSmin_{a,b}^s \right) \, .
\end{aligned}
\end{equation}

Our ultimate goal is to show that $\MCCSmin_{a,b}^s$ is homotopy equivalent 
to the complex $\MCCS_{a,b}^s$ from Definition \ref{def:MCS and KMCS}.
For this, we need one more technical result, namely that any partially symmetric
function of the form $f(\leftX_2)-f(\rightX_2)$ acts null-homotopically on
$\qdeg^{b(a-b-1)}\tdeg^b\KMCSmin_{a,b}$ and its subquotients $\MCCSmin^s_{a,b}$.

\begin{definition}\label{def:Theta} For each $r \in \{1,\ldots,b\}$, let
$\Theta_r \in \End^{2r,-1}(\qdeg^{b(a-b-1)}\tdeg^b\KMCSmin_{a,b})$ be given by 
\[
\Theta_r := \bigoplus_{k=0}^b (-1)^{b-k} \Id_{W_k} \otimes \xi_r \, .
\]
\end{definition}

Since $\Theta_r \colon P_{k,l,s} \to P_{k,l+1,s} \oplus P_{k,l,s+1}$, 
we have the decomposition $\Theta_r = \Theta^v_r + \Theta^c_r$, 
where $\Theta^v_r$ and $\Theta^c_r$ are uniquely characterized by
\begin{enumerate}
\item $\Theta^v_r$ restricts to morphisms $P_{k,l,s}\rightarrow P_{k,l+1,s}$, and
\item $\Theta^c_r$ restricts to morphisms $P_{k,l,s}\rightarrow P_{k,l,s+1}$.
\end{enumerate}

\begin{prop}
	\label{prop:trivMCCS}
The element $\Theta_r^v \in \End^{2r,-1}(\MCCSmin_{a,b}^s)$ satisfies 
$[\d^v+\d^h, \Theta_r^v] = h_r(\leftX_2 - \rightX_2)$ for all $0 \leq s \leq b$.
\end{prop}
\begin{proof}
Definition \ref{def:Koszul cx} and Definition \ref{def:Theta} directly
imply that
\[
[\d^v,\Theta^v_r+\Theta^c_r] = h_r(\leftX_2-\rightX_2)
\quad \text{and} \quad
[\d^h+\d^c, \Theta^v_r+\Theta^c_r] = 0 \, .
\]
Taking the components that preserve $s$-degree gives
$[\d^v,\Theta^v_r] = h_r(\leftX_2-\rightX_2)$
and $[\d^h,\Theta^v_r]=0$.
\end{proof}

\subsection{The colored 2-strand full twist}
\label{ss:ft computation}
In this section, we prove that $\MCCS_{a,b}^0 \simeq \MCCSmin_{a,b}^0$.  
This gives an explicit model for the Rickard complex of the $(a,b)$-colored 
$2$-strand full twist.
This result is of independent interest, but will also serve as an ingredient in proving
that $\MCCS_{a,b}^s\simeq \MCCSmin_{a,b}^s$ for all $0 \leq s \leq b$ below in 
Corollary \ref{cor:MCCSmin topological}.

We visualize the main object of study 
$\MCCSmin_{a,b}^0 = \bigoplus_{0\leq l\leq k\leq b} P_{k,l,0}$, 
with its two anti-commuting differentials $\d^v,\d^h$, 
as the following double complex
\begin{equation}\label{eq:Q0 as bicomplex}
\begin{tikzpicture}[anchorbase]
\node (aa) at (0,0) {$P_{0,0,0}$};
\node (ba) at (-2,0) {$P_{1,0,0}$};
\node (ca) at (-4,0) {$P_{2,0,0}$};
\node (da) at (-6,0) {$P_{3,0,0}$};
\node (ea) at (-8,0) {$\cdots$};
\node (bb) at (-2,2) {$P_{1,1,0}$};
\node (cb) at (-4,2) {$P_{2,1,0}$};
\node (db) at (-6,2) {$P_{3,1,0}$};
\node (eb) at (-8,2) {$\cdots$};
\node (cc) at (-4,4) {$P_{2,2,0}$};
\node (dc) at (-6,4) {$P_{3,2,0}$};
\node (ec) at (-8,4) {$\cdots$};
\node (dd) at (-6,6) {$P_{3,3,0}$};
\node (ed) at (-8,6) {$\cdots$};
\node at (-7,7) {$\ddots$};
\path[->,>=stealth,shorten >=1pt,auto,node distance=1.8cm]
%
%
(ea) edge node {$\d^h$} (da)
(da) edge node {$\d^h$} (ca)
(ca) edge node {$\d^h$} (ba)
(ba) edge node {$\d^h$} (aa)
(eb) edge node {$\d^h$} (db)
(db) edge node {$\d^h$} (cb)
(cb) edge node {$\d^h$} (bb)
(ec) edge node {$\d^h$} (dc)
(dc) edge node {$\d^h$} (cc)
(ed) edge node {$\d^h$} (dd)
%
%
([xshift=-2pt] dd.south) edge node[left]  {$\d^v$} ([xshift=-2pt] dc.north)
([xshift=-2pt] dc.south) edge node[left]  {$\d^v$} ([xshift=-2pt] db.north)
([xshift=-2pt] db.south) edge node[left]  {$\d^v$} ([xshift=-2pt] da.north)
([xshift=-2pt] cc.south) edge node[left]  {$\d^v$} ([xshift=-2pt] cb.north)
([xshift=-2pt] cb.south) edge node[left]  {$\d^v$} ([xshift=-2pt] ca.north)
([xshift=-2pt] bb.south) edge node[left] {$\d^v$} ([xshift=-2pt] ba.north);
\end{tikzpicture}
\end{equation}

\begin{rem}\label{rem:BH} 
Up to grading shift, this double complex is isomorphic to 
the image of the categorical inverse ribbon element
$\mathbf{r}\inv\oone_{a-b}$ of quantum $\slnn{2}$, as defined by
Beliakova--Habiro \cite{BH}, under the $2$-functor $\Phi$ to singular Soergel bimodules. 
More precisely, the version of the double complex considered here 
has vertical differentials modeled on differences of elementary
symmetric polynomials, corresponding to the version
$\tilde{\mathbf{r}}\inv\oone_{a-b}$ from \cite[Section 11]{BH}. 
The original version $\mathbf{r}\inv\oone_{a-b}$ of the inverse ribbon complex 
defined in \cite[Section 4]{BH} uses differentials modeled on complete symmetric
polynomials in a difference of alphabets and is closer to $\MCCSmin^0$ expressed
in terms of the the exterior algebra generators $\xi_i$. Also note that the
notions of horizontal and vertical differentials are interchanged between this
paper and \cite{BH}.
\end{rem} 

It will be convenient to give special notation to the rows of the double
complex $\MCCSmin^0$.

\begin{definition}\label{def:Q0 rows}
Let $R_l$ denote the complex $(\bigoplus_{k=l}^b P_{k,l,0}, \d^h)$.
\end{definition}

By construction, we have 
$\MCCSmin^0 = \tw_{\d^v}(\bigoplus_{l=0}^b R_l)$. 
The key to proving $\MCCS^0\simeq \MCCSmin^0$ is the following
topological interpretation of the rows $R_l$.  

\begin{prop}\label{prop:interpretation of Rl0}
For $0\leq l \leq b$ we have
\[
R_l \; \simeq \; \qdeg^{-(b-l)} \tdeg^{b-l} \F^{(l)}\E^{(a-b+l)} \hComp C_{a,b} = 
\qdeg^{-(b-l)} \tdeg^{b-l} 
\left\llbracket
\begin{tikzpicture}[scale=.5,smallnodes,anchorbase]
	\draw[very thick] (-2,0) node[left]{$a$} to (0,0) \pr (2,1) node[right]{$b$};
	\draw[line width=5pt,color=white] (0,1) \pr (2,0);
	\draw[very thick] (-2,1) node[left]{$b$} to (0,1) \pr (2,0) node[right]{$a$};
	\draw[very thick] (-1.75,1) to (-1.5,.5)node[right,yshift=-1pt,xshift=-1pt]{$l$} to  (-1.25,0);
	\draw[very thick] (-.75,0) to (-.5,.5) to  (-.25,1);
\end{tikzpicture}
\right\rrbracket.
\]
\end{prop}

The proof requires the following preparatory results.

\begin{definition}\label{def:partitions in rectangle} 
For each pair of integers $r,s\geq 0$, 
let $P(r,s)$ denote the set of partitions $\a$ with $\a_1\leq s$
and at most $r$ parts 
(i.e.~the Young diagram for $\a$ fits in an $r\times s$ rectangle). 
For each $\a\in P(r,s)$, let $\hat{\a}\in P(s,r)$ denote the
\emph{dual complementary partition}. 
Let $\Z^{P(r,s)}$ denote the graded abelian group that is free on the
partitions $\a\in P(r,s)$, graded by declaring that $\deg(\a) = 2|\a|-rs$.
\end{definition}

\begin{lem}[{\cite[Theorem 5.1.1]{KLMS}}] \label{lem:thickdigon} 
For $r,s\geq 0$, there is an isomorphism 
\[
\F^{(s)}\F^{(r)} \; \cong \bigoplus_{\alpha\in P(r,s)} \qdeg^{2|\a|-r s} \F^{(r+s)}
\]
with components given by
\[
(-1)^{|\hat\alpha|}
\begin{tikzpicture}[scale =.75,smallnodes,anchorbase]
	\draw[CQG, ultra thick,<-] (0,0)  to [out=90,in=210] 
	node[black,pos=.5]{\CQGbox{\mathfrak{s}_{\hat{\alpha}}}}(.5,.875);
	\draw[CQG, ultra thick,<-] (1,0)  to [out=90,in=330] (.5,.875);
	\draw[CQG, ultra thick] (.5,.875) to (.5,1.625);
\end{tikzpicture}
\colon  \F^{(s)}\F^{(r)} \to \qdeg^{2|\a|-r s} \F^{(r+s)}
\, , \quad
\begin{tikzpicture}[scale =.75,smallnodes,anchorbase]
	\draw[CQG, ultra thick,<-] (.5,.625) to (.5,1.375);
	\draw[CQG, ultra thick] (.5,1.375) to [out=150,in=270] 
		 (0,2.25);
	\draw[CQG, ultra thick] (.5,1.375)  to [out=30,in=270] 
		node[black,pos=.6]{\CQGbox{\mathfrak{s}_{\alpha}}} (1,2.25);
\end{tikzpicture}
\colon \qdeg^{2|\a|-r s} \F^{(r+s)} \to \F^{(s)}\F^{(r)}
\] \qed
\end{lem}

\begin{lem}\label{lem:bij} 
Fix $r,s\geq 0$ and let $\zeta_{j}$ for $1\leq j \leq r+s$ be variables of
degree $\qdeg^{2j} \tdeg\inv$. The bijections between the following:
	\begin{enumerate}
		\item the set $B(r,s)$ of binary sequences $\e\in \{0,1\}^{r+s}$with
		exactly $r$ $0$'s in positions $i_1<\dots<i_r$ and $s$ $1$'s in
		positions $j_1<\cdots<j_s$,
		\item the set of of non-zero monomial basis elements $\zeta_{\e} :=
		\zeta_{j_1}\cdots\zeta_{j_s}$ in $\qdeg^{-s(r+s+1)} \tdeg^s
		\largewedge^s[\zeta_1,\ldots,\zeta_{r+s}]$,
		\item the set of partitions $P(r,s)$
	\end{enumerate}
given by $\e \leftrightarrow \zeta_\e \leftrightarrow \a(\e)$ with 
$\a(\e)_m  :=  \#\{e\in \{1,\ldots,s\} \: | \: j_e> j_m\}$ determine an isomorphism of
(bi)graded abelian groups
\begin{equation}\label{eq:iso of mult spaces}
	\psi \colon   \qdeg^{-s(r+s+1)} \tdeg^s  \largewedge^s[\zeta_1,\ldots,\zeta_{r+s}] \xrightarrow{\cong} \Z^{P(r,s)}
	\, , \quad 
	\psi(\zeta_\e) := \alpha(\e)
\end{equation}
\end{lem}
\begin{proof} The bijections are standard, thus clearly induce an isomorphism
$\psi$ of abelian groups. To verify that $\psi$ preserves the bigrading, note
that, prior to any shifts, the monomial $\zeta_{\e}$ is of degree
$\qdeg^{2|\a(\e)|+s(s+1)}\tdeg^{-s}$ in
$\largewedge^s[\zeta_1,\ldots,\zeta_{r+s}]$. To see this, observe that it holds
for the sequence $1,\dots,1, 0, \dots ,0$, and that if a sequence $\e$ is
obtained from a sequence $\e'$ by replacing $1,0$ by $0,1$, then $2|\alpha(\e)|
- 2|\alpha(\e')| = 2 = \deg_q(\zeta_\e) - \deg_q(\zeta_{\e'})$.
\end{proof}

\begin{proof}[Proof of Proposition \ref{prop:interpretation of Rl0}]
Let $0\leq l\leq b$.
Lemma \ref{lem:Cautis+} implies that
\begin{equation}\label{eq:FEC}
\qdeg^{-(b-l)} \tdeg^{b-l} \F^{(l)}\E^{(a-b+l)} C_{a,b} 
\; \simeq \;
\tw_{(\d^h)'} \! \left(\bigoplus_{k=l}^{b}  \qdeg^{kd-2b+l} \tdeg^{2b-k-l} \F^{(l)}\F^{(k-l)}\E^{(k)}
\right)
\end{equation}
where $d=a-b+l+1$ and
\[
(\d^h)':=
\bigoplus_{i=0}^b
\CQGsgn{(-1)^{b-k}}
\begin{tikzpicture}[anchorbase,smallnodes]
	\draw[CQG,ultra thick,<-] (-.75,-.5) node[below]{$l$} to (-.75,.7);
	\draw[CQG,ultra thick,<-] (0,-.5) node[below]{$k{-}l$} to 
		(0,.7) node[above=-2pt,xshift=-2pt]{$k{-}l{-}1$};
	\draw[CQG,ultra thick,->] (.75,-.5) node[below]{\scriptsize$k$} to 
		(.75,.7) node[above=-2pt,xshift=2pt]{$k{-}1$};
	\draw[CQG,thick, directed=.75] (.75,0) to [out=90,in=90]  (0,0);
\end{tikzpicture} \, .
\]
Using Lemmata~\ref{lem:thickdigon} and \ref{lem:bij}, we deduce that
\begin{equation}\label{eq:Riso}
\begin{aligned}
\text{Right-hand side of \eqref{eq:FEC}}
\;&\cong \;  \tw_{(\d^h)'}\left( \bigoplus_{k=l}^{b} 
\bigoplus_{\a\in P(k-l,l)}  \qdeg^{kd-2b+l-l(k-l)+2|\a|} \tdeg^{2b-k-l} W_k\right) \\
\;& \cong\; \tw_{(\d^h)'}\left(\bigoplus_{k=l}^b  \qdeg^{kd-2b+l} \tdeg^{2b-k-l} W_k\otimes \Z^{P(k-l,l)} \right) \\
\;& \cong\; \tw_{(\d^h)'}\left(\bigoplus_{k=l}^b  \qdeg^{k(a-b-1)-2b} \tdeg^{2b-k} W_k\otimes  \largewedge^l[\zeta^{(k)}_1,\ldots, \zeta^{(k)}_k] \right). 
\end{aligned}
\end{equation}
We conclude that the latter chain complex has the same chain groups as the
complex $(R_l,\d^h)$, so it suffices to equate their differentials.

The component of the differential $(\d^h)'$ in the first line of \eqref{eq:Riso}
from the $(k,\a)$ summand to the $(k-1,\gamma)$ summand is 
\begin{equation}\label{eq:lemdiff3}
(-1)^{|\hat{\gamma}|} \cdot
(-1)^{b-k}
\begin{tikzpicture}[smallnodes,anchorbase]
	\draw[CQG,ultra thick,->] (1,0)  to (1,2.25) node[above, xshift=3pt]{\scriptsize$k{-}1$};  
	\draw[CQG,thick, directed=.55] (1,1.25) to [out=90,in=90] (.3,1.25); 
	\draw[CQG, ultra thick,<-] (0,0)  to (0,.5); 
	\draw[CQG, ultra thick] (0,.5) to [out=150,in=210]
		node[black,pos=.6,xshift=0pt]{\CQGbbox{\mathfrak{s}_{\hat{\gamma}}}} (0,1.75); 
	\draw[CQG, ultra thick] (0,.5) to [out=30,in=330] 
		node[black,pos=.3,xshift=-2pt]{\CQGbox{\mathfrak{s}_\alpha}} (0,1.75); 
	\draw[CQG, ultra thick] (0,1.75) to (0,2.25) node[above, xshift=-3pt]{\scriptsize$k{-}1$};
\end{tikzpicture}
\end{equation} 
so we must show that, with respect to the isomorphism \eqref{eq:iso of mult
spaces}, we have $\psi\circ \d^h =  (\d^h)'\circ \psi$. (Recall that the
differential $\d^h$ on $R_l$ was characterized in
Proposition~\ref{prop:diffKMCSexplicit}.) For this we use the following
symmetric function identity:
\[
\mathfrak{s}_{\a}(\X + z) = \sum_{\lbd}\mathfrak{s}_{\lbd}(\X) z^{m_\lambda},
\]
where $m_{\lbd} = | \a | - | \lbd |$
and the sum on the right is over all Young diagrams $\lbd\subset \a$ for which
the skew diagram $\a / \lbd$ does not contain two boxes in the same column. 
Such a skew diagram is called a \emph{horizontal strip}.
Thus,  
\[
(-1)^{|\hat{\gamma}|} \cdot
(-1)^{b-k}
\begin{tikzpicture}[smallnodes,anchorbase]
	\draw[CQG,ultra thick,->] (1,0)  to (1,2.25) node[above, xshift=3pt]{\scriptsize$k{-}1$};  
	\draw[CQG,thick, directed=.55] (1,1.25) to [out=90,in=90] (.3,1.25); 
	\draw[CQG, ultra thick,<-] (0,0)  to (0,.5); 
	\draw[CQG, ultra thick] (0,.5) to [out=150,in=210]
		node[black,pos=.6,xshift=0pt]{\CQGbox{\mathfrak{s}_{\hat{\gamma}}}} (0,1.75); 
	\draw[CQG, ultra thick] (0,.5) to [out=30,in=330] 
		node[black,pos=.3,xshift=-2pt]{\CQGbox{\mathfrak{s}_\alpha}} (0,1.75); 
	\draw[CQG, ultra thick] (0,1.75) to (0,2.25) node[above, xshift=-3pt]{\scriptsize$k{-}1$};
\end{tikzpicture}
= 
(-1)^{b-k}
\sum_{\lbd} 
(-1)^{|\hat{\gamma}|} \cdot
\begin{tikzpicture}[smallnodes,anchorbase]
	\draw[CQG,ultra thick,->] (1,0)  to (1,2.25) node[above, xshift=3pt]{\scriptsize$k{-}1$};  
	\draw[CQG,thick, directed=.35] (1,.85) to [out=90,in=90] node[black,pos=.5] {$\bullet$} node[black,below]{$m_\lbd$} (.3,.85); 
	\draw[CQG, ultra thick,<-] (0,0)  to (0,.5); 
	\draw[CQG, ultra thick] (0,.5) to [out=150,in=210]
		node[black,pos=.4,xshift=0pt]{\CQGbox{\mathfrak{s}_{\hat{\gamma}}}} (0,1.75); 
	\draw[CQG, ultra thick] (0,.5) to [out=30,in=330] 
		node[black,pos=.7,xshift=-2pt]{\CQGbox{\mathfrak{s}_\lbd}} (0,1.75); 
	\draw[CQG, ultra thick] (0,1.75) to (0,2.25) node[above, xshift=-3pt]{\scriptsize$k{-}1$};
\end{tikzpicture} \, ,
\]
where the sum on the right is over partitions $\lbd\in P(k-l-1,l)$ such that 
$\a / \lbd$ is a horizontal strip. By Lemma~\ref{lem:thickdigon}, all terms in this sum
vanish, unless $\lbd = \gamma$. The latter holds precisely when $\alpha / \gamma$
is a horizontal strip, in which case the only surviving term in the sum
evaluates to $\chi_m^+$ with $m:=|\a|-|\gamma|$.

Now, suppose $\e,\e'\in \{0,1\}^k$ are binary sequences 
with $l$ occurrences of 1.  
Let $j_1 < \cdots < j_l$ be the indices for which $\e_{j_p}=1$, and similarly
for $j_1'<\cdots < j_l'$.  
Let $\a(\e) ,\a(\e') \in P(k-l,l)$ be the associated partitions, then
Lemma~\ref{lem:bij} implies this is a horizontal strip if and only if $j_p -
j_p' \in \{0,1\}$ for all $p=1,\ldots,l$. Indeed, the bijection therein gives
that $\alpha(\e') \subset \alpha(\e)$ if and only if $\e'$ can be obtained from
$\e$ by a sequence of operations on binary sequences that replace the (adjacent)
symbols $0,1$ with $1,0$. We hence can pass from $\e$ to $\e'$ by permuting the
initial $1$ in $\e$ left through some $0$'s to its position in $\e'$, then do
the same for the second $1$ in $\e$, and so on. If $j_p - j_p' > 1$, then at the
$p^{th}$ step of this procedure, we move a $1$ past more than one $0$, which
produces two or more boxes in a column of $\alpha(\e)$ that are not in
$\alpha(\e')$. Given this, the result now follows from
Proposition~\ref{prop:diffKMCSexplicit}.
\end{proof}

\begin{thm}\label{thm:Q0 is FT} 
We have $\MCCS^0_{a,b}\simeq \MCCSmin^0_{a,b}$. 
In alternative notation:
\[
\left\llbracket
 \begin{tikzpicture}[scale=.5,smallnodes,anchorbase,rotate=90]
 	\draw[very thick,->] (1,0) to [out=90,in=270] (0,1.5);
 	\draw[line width=5pt,color=white,->] (1,-1.5) to [out=90,in=270] (0,0) 
		to [out=90,in=270] (1,1.5);
 	\draw[very thick,->] (1,-1.5) node[right,xshift=-2pt]{$b$} to [out=90,in=270] (0,0) 
		to [out=90,in=270] (1,1.5);
 	\draw[line width=5pt,color=white] (0,-1.5) to [out=90,in=270] (1,0);
 	\draw[very thick] (0,-1.5) node[right,xshift=-2pt]{$a$} to [out=90,in=270] (1,0);
 \end{tikzpicture}
\right\rrbracket
\simeq 
\tw_{\d^v+\d^h}\left(\bigoplus_{0\leq l\leq k\leq b}^b \qdeg^{k(a-b+1)-2b} \tdeg^{2b-k}
\begin{tikzpicture}[smallnodes,rotate=90,anchorbase,scale=.75,baseline=0em]
	\draw[very thick] (0,.25) to [out=150,in=270] (-.25,1) node[left,xshift=2pt]{$a$};
	\draw[very thick] (.5,.5) to (.5,1) node[left,xshift=2pt]{$b$};
	\draw[very thick] (0,.25) to node[left,xshift=2pt,yshift=-1pt]{$k$} (.5,.5);
	\draw[very thick] (0,-.25) to (0,.25);
	\draw[very thick] (.5,-.5) to [out=30,in=330] (.5,.5);
	\draw[very thick] (0,-.25) to node[right,xshift=-2pt,yshift=-1pt]{$k$} (.5,-.5);
	\draw[very thick] (.5,-1) node[right,xshift=-2pt]{$b$} to (.5,-.5);
	\draw[very thick] (-.25,-1)node[right,xshift=-2pt]{$a$} to [out=90,in=210] (0,-.25);
\end{tikzpicture}
\otimes \largewedge^l[\zeta^{(k)}_1,\ldots,\zeta^{(k)}_k]\right)
\]
where the anticommuting differentials $\d^v$ and $\d^h$ are as described in
Proposition~\ref{prop:diffKMCSexplicit}.
\end{thm}

This shows that \cite[Conjecture 1.3]{BH} holds in the singular Soergel bimodule 
$2$-representation of categorified quantum $\slnn{2}$, 
and hence in any integrable quotient of $\USd(\slnn{2})$.
See Remark~\ref{rem:BH}.

\begin{proof}
Recall that $C_{a,b}^\vee$ denotes the inverse to the Rickard complex $C_{a,b}$. 
Using Proposition \ref{prop:interpretation of Rl0}, we compute
\[
\MCCSmin^0\hComp C_{a,b}^\vee 
\cong 
\left(\bigoplus_{l=0}^b R_l\hComp C_{a,b}^\vee, \d^v\hComp \Id_{C_{a,b}^\vee}\right)
\simeq
\tw_{\d} \! \left(\bigoplus_{l=0}^b  \qdeg^{-(b-l)} \tdeg^{b-l}  \F^{(l)}\E^{(a-b+l)}\oone_{b,a}\right)
\]
for some differential $\d$. 
Note that this agrees with $C_{b,a}$ as a graded bimodule.

Proposition~\ref{prop:trivMCCS} shows that 
the action of $h_r(\leftX_2 - \rightX_2)$ on $\MCCSmin^0$
is null-homotopic
for all $r>0$. 
Further, by Proposition~\ref{prop:UniqueYCrossing}, 
the action of $h_r(\leftX_2 - \rightX_1)$ on $C_{a,b}^\vee$ 
is null-homotopic
for all $r >0$. Together, these facts imply that 
the action of $h_r(\leftX_2 - \rightX_1)$ on $\MCCSmin^0\hComp C_{a,b}^\vee$
is null-homotopic. 
Proposition \ref{prop:UniqueYCrossing} then implies that
$\MCCSmin^0\hComp C_{a,b}^\vee \simeq C_{b,a}$, 
and thus $\MCCSmin^0 \simeq C_{b,a}\hComp C_{a,b} = \MCCS^0$. 
\end{proof}

\subsection{Proof of the colored skein relation}
\label{ss:colored skein}
In this section, we prove Theorem \ref{thm:coloredskein}. 
The key step is to show that $\MCCSmin_{a,b}^s$ is related to $\MCCSmin_{a,b-s}^0$ 
in precisely the same way that $\MCCS_{a,b}^s$ is related to $\MCCS_{a,b-s}^0$.

\begin{definition}
Let $\I^{(s)}\colon \CS_{a,\ell}\rightarrow \CS_{a,\ell+s}$ 
denote the functor defined by
\[
\I^{(s)}(X):=
\begin{tikzpicture}[scale=.5,smallnodes,rotate=90,anchorbase]
	\draw[very thick] (1,2.25) to (1,3) node[left]{$\ell{+}s$};
	\draw[very thick] (1,-3) node[right]{$\ell{+}s$} to (1,-2.25); 
	\draw[very thick] (1,-2.25) to [out=30, in=270] (1.75,-1.6) 
		to (1.75,0) node[below,yshift=1pt]{$s$}to (1.75,1.6)  to  [out=90,in=330] (1,2.25); 
	\node[yshift=-2pt] at (0,0) {\normalsize$X$};
	\draw[very thick] (.75,1.6) rectangle (-.75,-1.6);
	\draw[very thick] (1,-2.25) to [out=150,in=270] 
		node[right,xshift=-1pt,yshift=-1pt]{$\ell$} (.25,-1.6);
	\draw[very thick] (1,2.25) to [out=210,in=90] 
		node[left,xshift=1pt,yshift=-1pt]{$\ell$} (.25,1.6);
	\draw[very thick] (-.25,3) node[left]{$a$} to (-.25,1.6);
	\draw[very thick] (-.25,-3) node[right]{$a$} to (-.25,-1.6);
\end{tikzpicture} \, .
\]
In other words, 
$\I^{(s)}(X) = (\oone_{a}\boxtimes {}_{(\ell+s)}M_{(\ell,s)})\hComp (X\boxtimes \oone_s) \hComp (\oone_{a}\boxtimes {}_{(\ell,s)}S_{(\ell+s)})$.
We will write $\I:=\I^{(1)}$.
\end{definition}

\begin{remark}
We have\footnote{The isomorphism is given by applying the ``associativity'' relation for 
webs/bimodules, and then ``removing the digon.''} 
$\I^{(s_1)}\circ \I^{(s_2)}(X) \cong \bigoplus_{\qbinom{s_1+s_2}{s_1}}\I^{(s_1+s_2)}$,
so $\I^s(X)\cong \bigoplus_{[s]!} \I^{(s)}(X)$.  
Thus $\I^{(s)}$ may be thought of as the $s^{th}$ divided power of $\I$, 
in the same way that $\E^{(s)}$ and $\F^{(s)}$ are the divided powers of $\E$ and $\F$
in the setting of categorified quantum groups.
\end{remark}

Theorem \ref{thm:coloredskein} will follow almost immediately from the following result.

\begin{proposition}\label{prop:I of Q0}
We have $\MCCSmin_{a,b}^s\cong \I^{(s)}(\MCCSmin_{a,b-s}^0)$.
\end{proposition}

This proposition requires careful bookkeeping, taken care of by the following.

\begin{lemma}\label{lemma:I of Wk}
For each $0\leq s\leq b$ and each $0\leq k\leq b-s$, 
we have an isomorphism of weight $\qdeg^{s(b+k+1)} \tdeg^{-s}$:
\begin{equation}\label{eq:mu}
\mu_k \colon  
\begin{tikzpicture}[smallnodes, scale=.75,rotate=90,anchorbase]
	\draw[very thick,] (0,.25) to [out=150,in=270] (-.25,1.25) node[left,xshift=2pt]{$a$};
	\draw[very thick,] (.5,.5) to [out=90,in=210] (.75,1);
	\draw[very thick] (.75,1) to (.75,1.25) node[left,xshift=2pt]{$b$};
	\draw[very thick] (0,.25) to node[left,xshift=2pt,yshift=-1pt]{$k$} (.5,.5);
	\draw[very thick] (0,-.25) to (0,.25);
	\draw[very thick] (.5,-.5) to [out=30,in=330] (.5,.5);
	\draw[very thick] (0,-.25) to node[right,xshift=-2pt,yshift=-1pt]{$k$} (.5,-.5);
	\draw[very thick] (.75,-1) to [out=150,in=270] (.5,-.5);
	\draw[very thick] (.75,-1.25) node[right,xshift=-2pt]{$b$} to (.75,-1);
	\draw[very thick] (-.25,-1.25) node[right,xshift=-2pt]{$a$} to [out=90,in=210] (0,-.25);
	\draw[very thick] (.75,-1) to [out=30,in=330] node[above,yshift=-2pt]{$s$} (.75,1);
\end{tikzpicture}
\xrightarrow{\cong} 
W_{k} \otimes \largewedge^s[\zeta^{(k)}_{k+1},\ldots,\zeta^{(k)}_{b}] \, .
\end{equation}
For each integer $m\geq 0$, these isomorphisms fit into a commutative diagram
\[
\begin{tikzcd}[column sep=huge]
\begin{tikzpicture}[smallnodes, scale=.75,rotate=90,anchorbase]
	\draw[very thick,] (0,.25) to [out=150,in=270] (-.25,1.25) node[left,xshift=2pt]{$a$};
	\draw[very thick,] (.5,.5) to [out=90,in=210] (.75,1);
	\draw[very thick] (.75,1) to (.75,1.25) node[left,xshift=2pt]{$b$};
	\draw[very thick] (0,.25) to node[left=-3pt,yshift=-1pt]{$k$} (.5,.5);
	\draw[very thick] (0,-.25) to (0,.25);
	\draw[very thick] (.5,-.5) to [out=30,in=330] (.5,.5);
	\draw[very thick] (0,-.25) to node[right=-3pt,yshift=-1pt]{$k$} (.5,-.5);
	\draw[very thick] (.75,-1) to [out=150,in=270] (.5,-.5);
	\draw[very thick] (.75,-1.25) node[right,xshift=-2pt]{$b$} to (.75,-1);
	\draw[very thick] (-.25,-1.25) node[right,xshift=-2pt]{$a$} to [out=90,in=210] (0,-.25);
	\draw[very thick] (.75,-1) to [out=30,in=330] node[above,yshift=-2pt]{$s$} (.75,1);
\end{tikzpicture}
\ar[r,"I^{(s)}(\chi^+_m)"] \ar[d,"\mu_k"]
& 
\begin{tikzpicture}[smallnodes, scale=.75,rotate=90,anchorbase]
	\draw[very thick,] (0,.25) to [out=150,in=270] (-.25,1.25) node[left,xshift=2pt]{$a$};
	\draw[very thick,] (.5,.5) to [out=90,in=210] (.75,1);
	\draw[very thick] (.75,1) to (.75,1.25) node[left,xshift=2pt]{$b$};
	\draw[very thick] (0,.25) to node[left=-3pt,yshift=-1pt]{$k{-}1$} (.5,.5);
	\draw[very thick] (0,-.25) to (0,.25);
	\draw[very thick] (.5,-.5) to [out=30,in=330] (.5,.5);
	\draw[very thick] (0,-.25) to node[right=-3pt,yshift=-1pt]{$k{-}1$} (.5,-.5);
	\draw[very thick] (.75,-1) to [out=150,in=270] (.5,-.5);
	\draw[very thick] (.75,-1.25) node[right,xshift=-2pt]{$b$} to (.75,-1);
	\draw[very thick] (-.25,-1.25) node[right,xshift=-2pt]{$a$} to [out=90,in=210] (0,-.25);
	\draw[very thick] (.75,-1) to [out=30,in=330] node[above,yshift=-2pt]{$s$} (.75,1);
\end{tikzpicture}
\ar[d,"\mu_{k-1}"] \\
W_{k} \otimes \largewedge^s[\zeta^{(k)}_{k+1},\ldots,\zeta^{(k)}_{b}]
\ar[r,"f"]
&
W_{k-1} \otimes \largewedge^s[\zeta^{(k-1)}_{k},\ldots,\zeta^{(k-1)}_{b}]
\end{tikzcd}
\]
where, for $k+1 \leq i_1 < \cdots < i_s \leq b$ and 
$k+1 \leq j_1 < \cdots < j_s \leq b$, the component
\[
W_{k} \otimes \zeta^{(k)}_{i_1}\cdots \zeta^{(k)}_{i_s}
\xrightarrow{f} W_{k-1} \otimes  \zeta^{(k-1)}_{j_1}\cdots \zeta^{(k-1)}_{j_s}
\]
is zero unless $i_p-j_p\in\{0,1\}$ for all $k+1\leq p\leq b$. 
In this case, it equals $\chi^+_{m+n}$ where $n=\sum_p(i_p-j_p)$.
\end{lemma}
\begin{proof}
The isomorphism $\mu_k$ is defined to be the composition of
\[
\begin{tikzpicture}[smallnodes, scale=.75,rotate=90,baseline=0em]
	\draw[very thick,] (0,.25) to [out=150,in=270] (-.25,1.25) node[left,xshift=2pt]{$a$};
	\draw[very thick,] (.5,.5) to [out=90,in=210] (.75,1);
	\draw[very thick] (.75,1) to (.75,1.25) node[left,xshift=2pt]{$b$};
	\draw[very thick] (0,.25) to node[left,xshift=2pt,yshift=-1pt]{$k$} (.5,.5);
	\draw[very thick] (0,-.25) to (0,.25);
	\draw[very thick] (.5,-.5) to [out=30,in=330] (.5,.5);
	\draw[very thick] (0,-.25) to node[right,xshift=-2pt,yshift=-1pt]{$k$} (.5,-.5);
	\draw[very thick] (.75,-1) to [out=150,in=270] (.5,-.5);
	\draw[very thick] (.75,-1.25) node[right,xshift=-2pt]{$b$} to (.75,-1);
	\draw[very thick] (-.25,-1.25) node[right,xshift=-2pt]{$a$} to [out=90,in=210] (0,-.25);
	\draw[very thick] (.75,-1) to [out=30,in=330] node[above,yshift=-2pt]{$s$} (.75,1);
\end{tikzpicture}
\xrightarrow{\cong}
\begin{tikzpicture}[smallnodes,rotate=90,baseline=0em]
	\draw[very thick] (0,.375) to [out=150,in=270] (-.25,1) node[left,xshift=2pt]{$a$};
	\draw[very thick] (.5,.625) to (.5,1) node[left,xshift=2pt]{$b$};
	\draw[very thick] (0,.375) to node[left,xshift=2pt,yshift=-1pt]{$k$} (.5,.625);
	\draw[very thick] (0,-.375) to (0,.375);
	\draw[very thick] (.75,.25) to [out=90,in=330] (.5,.625);
	\draw[very thick] (.5,-.625) to [out=30,in=270] (.75,-.25);
	\draw[very thick] (.75,-.25) to [out=30,in=330] node[above,yshift=-2pt]{$s$} (.75,.25);
	\draw[very thick] (.75,-.25) to [out=150,in=210] (.75,.25);
	\draw[very thick] (0,-.375) to node[right,xshift=-2pt,yshift=-1pt]{$k$} (.5,-.625);
	\draw[very thick] (.5,-1) node[right,xshift=-2pt]{$b$} to (.5,-.625);
	\draw[very thick] (-.25,-1)node[right,xshift=-2pt]{$a$} to [out=90,in=210] (0,-.375);
\end{tikzpicture}
\]
followed by the ``digon removal'' isomorphism described as follows.  
Let $S\subset \{k+1, \ldots,b\}$ with $|S|=s$, and set $S^c:= \{k+1,\ldots,b\}\setminus S$.
We may write $S=\{i_1<\cdots<i_s\}$ and $S^c= \{j_1<\cdots<j_{b-k-s}\}$. 
With this notation in place, define
$\zeta^{(k)}_S:=\zeta^{(k)}_{i_1}\cdots \zeta^{(k)}_{i_s}$ and 
$\a(\bar{S})_{b-k-s-m+1} :=  \#\{e\in \{1,\ldots,s\} \: | \: i_e < j_m\}$.
Using this setup, and the alphabet labeling conventions for the digon:
\[
\begin{tikzpicture}[scale=.5,rotate=90,anchorbase]
	\draw[very thick] (1,1) to (1,2) node[left=-2pt]{$\B$};
	\draw[very thick] (1,-2) node[right=-2pt]{$\B'$} to (1,-1) ; 
	\draw[very thick] (1,-1) to [out=150,in=210] node[below=-1pt]{$\D$}(1,1);
	\draw[very thick] (1,-1) to [out=30,in=330] node[above=-2pt]{$\Eb$} (1,1); 
\end{tikzpicture}
\]
we have the isomorphism
\begin{equation}\label{eq:digon-morphisms}
\begin{tikzpicture}[scale=.5,rotate=90,anchorbase,smallnodes]
	\draw[very thick] (1,1) to (1,2) node[left=-2pt]{$b{-}k$};
	\draw[very thick] (1,-2) to (1,-1) ; 
	\draw[very thick] (1,-1) to [out=150,in=210] node[below=-1pt]{$b{-}k{-}s$}(1,1);
	\draw[very thick] (1,-1) to [out=30,in=330] node[above=-2pt]{$s$} (1,1); 
\end{tikzpicture}
\xrightarrow{\bigoplus_S \col \circ \Schur_{\alpha(\bar{S})}(\D)}
\bigoplus_S \oone_{b{-}k}
\otimes \zeta_{S}^{(k)} 
\xrightarrow{\bigoplus_S (-1)^{|\widehat{\alpha(\bar{S})}|} \cre \circ \Schur_{\widehat{\alpha(\bar{S})}}(\Eb)}
\begin{tikzpicture}[scale=.5,rotate=90,anchorbase,smallnodes]
	\draw[very thick] (1,1) to (1,2);
	\draw[very thick] (1,-2) node[right=-2pt]{$b{-}k$} to (1,-1) ; 
	\draw[very thick] (1,-1) to [out=150,in=210] node[below=-1pt]{$b{-}k{-}s$}(1,1);
	\draw[very thick] (1,-1) to [out=30,in=330] node[above=-2pt]{$s$} (1,1); 
\end{tikzpicture}
\end{equation}
Here, the bimodule morphisms $\col$ and $\cre$ are given in Appendix \ref{s:FoamSSBim}.
Note that the correspondence between degree-$s$ monomials $\zeta_S$ and 
partitions $\alpha(\bar{S}) \in P(b-k-s,s)$ used here differs from 
the standard bijection\footnote{Note also that the roles of $i$ and $j$ are opposite 
to Lemma \ref{lem:bij}; the indices $i_e$ here index the terms in the monomial $\zeta_S$, 
thus correspond to $1$'s in the corresponding binary sequence.} 
from Lemma \ref{lem:bij} by the symmetry
$S \mapsto \bar{S}$ that reverses the order of a binary sequence. 
Nonetheless, \cite[Equations (3.10) and (3.11)]{QR} imply that
\eqref{eq:digon-morphisms} define inverse isomorphisms. 
The degree of the map $\mu_k$ obtained in this way can be deduced by 
comparing minimal degree summands. 

Finally, the statement concerning the components of $f$ holds since the
map $f:=\mu_{k-1} \circ I^{(s)}(\chi_m^+) \circ \mu_k\inv$ 
can be simplified in a manner analogous to the computation that 
simplifies \eqref{eq:lemdiff3} in the proof of Proposition \ref{prop:interpretation of Rl0}.
(Alternatively, this can be computed explicitly using foams.) 
\end{proof}

\begin{proof}[Proof of Proposition \ref{prop:I of Q0}]
By definition, 
\[
\I^{(s)}(\MCCSmin_{a,b-s}^0) = 
\left(\bigoplus_{0\leq l\leq k\leq b-s} P_{k,l,s}', (\d^v)' + (\d^h)'\right) \, ,
\]
where
\begin{equation}\label{eq:Pprime}
P_{k,l,s}' :=
\qdeg^{k(a-b+s+1)-2(b-s)} \tdeg^{2(b-s)-k}
\begin{tikzpicture}[smallnodes, anchorbase, scale=.75,rotate=90,baseline=.4em]
	\draw[very thick,] (0,.25) to [out=150,in=270] (-.25,1.25) node[left,xshift=2pt]{$a$};
	\draw[very thick,] (.5,.5) to [out=90,in=210] (.75,1);
	\draw[very thick] (.75,1) to (.75,1.25) node[left,xshift=2pt]{$b$};
	\draw[very thick] (0,.25) to node[left,xshift=2pt,yshift=-1pt]{$k$} (.5,.5);
	\draw[very thick] (0,-.25) to (0,.25);
	\draw[very thick] (.5,-.5) to [out=30,in=330] (.5,.5);
	\draw[very thick] (0,-.25) to node[right,xshift=-2pt,yshift=-1pt]{$k$} (.5,-.5);
	\draw[very thick] (.75,-1) to [out=150,in=270] (.5,-.5);
	\draw[very thick] (.75,-1.25) node[right,xshift=-2pt]{$b$} to (.75,-1);
	\draw[very thick] (-.25,-1.25) node[right,xshift=-2pt]{$a$} to [out=90,in=210] (0,-.25);
	\draw[very thick] (.75,-1) to [out=30,in=330] node[above,yshift=-2pt]{$s$} (.75,1);
\end{tikzpicture}
\otimes \largewedge^l[\zeta^{(k)}_1,\ldots,\zeta^{(k)}_k] \, .
\end{equation}
Here, $(\d^v)'=\I^{(s)}(\d^v)$ and $(\d^h)'=\I^{(s)}(\d^h)$ where $\d^v,\d^h$ 
in this instance are the differentials on $\MCCSmin_{a,b-s}^{0}$
from Proposition~\ref{prop:KMCSdiffs}.
Moreover, recall from Definition~ \ref{def:Qs as subquotient} that
\[
\MCCSmin_{a,b}^s = \left(\bigoplus_{0\leq l\leq k\leq b-s} 
\qdeg^{-s(b-1)} \tdeg^{-s} P_{k,l,s}, \d^v+\d^h\right),
\]
where
\[
\qdeg^{-s(b-1)} \tdeg^{-s} P_{k,l,s} = \qdeg^{k(a-b+1)-sb-2b+s} \tdeg^{2b-k-s} 
W_k
\otimes \largewedge^l[\zeta^{(k)}_1,\ldots, \zeta^{(k)}_k]\otimes 
\largewedge^s[\zeta^{(k)}_{k+1},\ldots, \zeta^{(k)}_b].
\]
Lemma \ref{lemma:I of Wk} implies that $\qdeg^{-s(b-1)} \tdeg^{-s} P_{k,l,s}$
and $P_{k,l,s}'$ are isomorphic.  This isomorphism involves the natural
isomorphism which swaps the order of tensor factors
$\largewedge[\zeta^{(k)}_1,\ldots,\zeta^{(k)}_k]\otimes
\largewedge^s[\zeta^{(k)}_{k+1},\ldots,\zeta^{(k)}_b] \cong
\largewedge^s[\zeta^{(k)}_{k+1}, \ldots,\zeta^{(k)}_b]\otimes
\largewedge[\zeta^{(k)}_1,\ldots,\zeta^{(k)}_k]$.  
By slight abuse of the notation from Lemma \ref{lemma:I of Wk}, 
we also denote this isomorphism by 
$\mu_k:P_{k,l,s}'\rightarrow \qdeg^{-s(b-1)} \tdeg^{-s} P_{k,l,s}$.

It remains to show that the isomorphisms $\mu_k$ intertwine the
differentials $\d^v$, $\d^h$ with the differentials $(\d^v)'$, $(\d^h)'$, 
i.e.~ that $\mu_{k-1} \circ (\d^v)' \circ \mu_k\inv = \d^v$ and 
$\mu_{k-1} \circ (\d^h)' \circ \mu_k\inv = \d^h$. 
For the vertical differentials, this is immediate since these differentials are of Koszul type in both complexes, 
acting by differences of elementary symmetric polynomials on the $k$-labeled ``rungs'' of the web.
Such endomorphisms commute with the digon removal isomorphism \eqref{eq:mu}. 

To compare the horizontal differentials, we explicitly match the components of
$\mu_{k-1} \circ (\d^h)' \circ \mu_k\inv$ with those of $\d^h$ using Lemma
\ref{lemma:I of Wk}.  Suppose we have subsets
\[
S=\{i_1<\cdots<i_l\}\subset \{1,\ldots,k\} 
\, , \quad
T=\{j_1<\cdots<j_s\}\subset \{k+1,\ldots,b\}
\]
\[
S'=\{i_1'<\cdots<i_l'\}\subset \{1,\ldots,k-1\}
\, , \quad 
T'=\{j_1'<\cdots<j_s'\}\subset \{k,\ldots,b\}
\]
with $i_p-i_p'\in\{0,1\}$ and $j_p-j_p'\in\{0,1\}$ for all $p$.  The
corresponding component
\[
W_k\otimes \zeta^{(k)}_{S\cup T}  \xrightarrow{\d^h} W_{k-1}\otimes \zeta^{(k-1)}_{S'\cup T'}
\]
is $\chi^+_{m+n}$ where $m=\sum_p(i_p-i_p')$ and $n=\sum_p(j_p-j_p')$ 
(and all nonzero components of $\d^h$ are of this form).  
Now, Lemma \ref{lemma:I of Wk} gives us commutative squares
\[
\begin{tikzcd}[column sep=6em]
\begin{tikzpicture}[smallnodes, scale=.75,rotate=90,anchorbase]
	\draw[very thick,] (0,.25) to [out=150,in=270] (-.25,1.25) node[left,xshift=2pt]{$a$};
	\draw[very thick,] (.5,.5) to [out=90,in=210] (.75,1);
	\draw[very thick] (.75,1) to (.75,1.25) node[left,xshift=2pt]{$b$};
	\draw[very thick] (0,.25) to node[left=-3pt,yshift=-1pt]{$k$} (.5,.5);
	\draw[very thick] (0,-.25) to (0,.25);
	\draw[very thick] (.5,-.5) to [out=30,in=330] (.5,.5);
	\draw[very thick] (0,-.25) to node[right=-3pt,yshift=-1pt]{$k$} (.5,-.5);
	\draw[very thick] (.75,-1) to [out=150,in=270] (.5,-.5);
	\draw[very thick] (.75,-1.25) node[right,xshift=-2pt]{$b$} to (.75,-1);
	\draw[very thick] (-.25,-1.25) node[right,xshift=-2pt]{$a$} to [out=90,in=210] (0,-.25);
	\draw[very thick] (.75,-1) to [out=30,in=330] node[above,yshift=-2pt]{$s$} (.75,1);
\end{tikzpicture} \otimes \zeta^{(k)}_S
\ar[r,"(\d^h)' = I^{(s)}(\chi^+_m)"]
& 
\begin{tikzpicture}[smallnodes, scale=.75,rotate=90,anchorbase]
	\draw[very thick,] (0,.25) to [out=150,in=270] (-.25,1.25) node[left,xshift=2pt]{$a$};
	\draw[very thick,] (.5,.5) to [out=90,in=210] (.75,1);
	\draw[very thick] (.75,1) to (.75,1.25) node[left,xshift=2pt]{$b$};
	\draw[very thick] (0,.25) to node[left=-3pt,yshift=-1pt]{$k{-}1$} (.5,.5);
	\draw[very thick] (0,-.25) to (0,.25);
	\draw[very thick] (.5,-.5) to [out=30,in=330] (.5,.5);
	\draw[very thick] (0,-.25) to node[right=-3pt,yshift=-1pt]{$k{-}1$} (.5,-.5);
	\draw[very thick] (.75,-1) to [out=150,in=270] (.5,-.5);
	\draw[very thick] (.75,-1.25) node[right,xshift=-2pt]{$b$} to (.75,-1);
	\draw[very thick] (-.25,-1.25) node[right,xshift=-2pt]{$a$} to [out=90,in=210] (0,-.25);
	\draw[very thick] (.75,-1) to [out=30,in=330] node[above,yshift=-2pt]{$s$} (.75,1);
\end{tikzpicture} \otimes \zeta^{(k-1)}_{S'} \\
W_{k} \otimes  \zeta^{(k)}_{S \cup T}
\ar[r,"\d^h = \chi^+_{m+n}"] \ar[u,"\mu_{k}\inv"]
&
W_{k-1} \otimes  \zeta^{(k-1)}_{S' \cup T'}
\ar[u,"\mu_{k-1}\inv"]
\end{tikzcd}
\] 
in which the vertical arrows are restrictions of $\mu_{k}\inv$ and $\mu_{k-1}\inv$ 
from Lemma \ref{lemma:I of Wk} to the indicated direct summands.
Taking the direct sum over all such $S,T,S',T'$ shows that the isomorphisms $\mu_k$ 
intertwine the horizontal differentials.
\end{proof}

\begin{cor}\label{cor:MCCSmin topological} 
$\MCCSmin_{a,b}^s\simeq \MCCS_{a,b}^s$.
\end{cor}
\begin{proof}
We have
\[
\MCCSmin_{a,b}^s \cong \I^{(s)}(\MCCSmin_{a,b-s}^0)
\simeq 
\I^{(s)}\left(
\left\llbracket
\begin{tikzpicture}[scale=.4,smallnodes,anchorbase,rotate=90]
	\draw[very thick] (1,0) to [out=90,in=270] (0,1.5);
	\draw[line width=5pt,color=white] (1,-1.5) to [out=90,in=270] (0,0) 
		to [out=90,in=270] (1,1.5);
	\draw[very thick] (1,-1.5) node[right=-2pt]{$b{-}s$} to [out=90,in=270] (0,0) 
		to [out=90,in=270] (1,1.5);
	\draw[line width=5pt,color=white] (0,-1.5) to [out=90,in=270] (1,0);
	\draw[very thick] (0,-1.5) node[right=-2pt]{$a$} to [out=90,in=270] (1,0);
\end{tikzpicture}
\right\rrbracket
\right) 
=
\left\llbracket
\begin{tikzpicture}[scale=.35,smallnodes,rotate=90,anchorbase]
\draw[very thick] (1,-1) to [out=150,in=270] (0,0); 
\draw[line width=5pt,color=white] (0,-2) to [out=90,in=270] (.5,0) to [out=90,in=270] (0,2);
\draw[very thick] (0,-2) node[right=-2pt]{$a$} to [out=90,in=270] (.5,0) 
	to [out=90,in=270] (0,2) node[left=-2pt]{$a$};
\draw[very thick] (1,1) to (1,2) node[left=-2pt]{$b$};
\draw[line width=5pt,color=white] (0,0) to [out=90,in=210] (1,1); 
\draw[very thick] (0,0) to [out=90,in=210] (1,1); 
\draw[very thick] (1,-2) node[right=-2pt]{$b$} to (1,-1); 
\draw[very thick] (1,-1) to [out=30,in=330] node[above=-2pt]{$s$} (1,1); 
\end{tikzpicture}
\right\rrbracket 
= \MCCS_{a,b}^s
\]
where the first isomorphism holds by Proposition \ref{prop:I of Q0} and the
homotopy equivalence follows from Theorem \ref{thm:Q0 is FT}. 
\end{proof}

\begin{proof}[Proof of Theorem~\ref{thm:coloredskein}]
We have that
\begin{align*}
\qdeg^{b(a-b-1)}\tdeg^b K
\left( 
\left\llbracket
\begin{tikzpicture}[scale=.35,smallnodes,anchorbase,rotate=270]
	\draw[very thick] (1,-1) to [out=150,in=270] (0,1) to (0,2) node[right=-2pt]{$b$}; 
	\draw[line width=5pt,color=white] (0,-2) to (0,-1) to [out=90,in=210] (1,1);
	\draw[very thick] (0,-2) node[left=-2pt]{$b$} to (0,-1) to [out=90,in=210] (1,1);
	\draw[very thick] (1,1) to (1,2) node[right=-2pt]{$a$};
	\draw[very thick] (1,-2) node[left=-2pt]{$a$} to (1,-1); 
	\draw[very thick] (1,-1) to [out=30,in=330] node[below=-1pt]{$a{-}b$} (1,1); 
\end{tikzpicture}
\right\rrbracket
\right)
= \qdeg^{b(a-b-1)}\tdeg^b\KMCS_{a,b} &\stackrel{\eqref{eq:topological MCS}}{\simeq} \qdeg^{b(a-b-1)}\tdeg^b\KMCSmin_{a,b} \\
&\stackrel{\eqref{eq:MainThmMin}}{\cong} \tw_{\d^c}\left(\bigoplus_{s=0}^b\qdeg^{s(b-1)} \tdeg^s\MCCSmin_{a,b}^s\right) \, .
\end{align*}
The latter is a one-sided twisted complex (see Definition \ref{def:twist}) since $\d^c$ strictly increases the index $s$.
Corollary \ref{cor:MCCSmin topological}, 
together with standard homological perturbation techniques 
(see \cite[Crude Perturbation Lemma]{Markl} or \cite[Corollary 4.10]{Hog3}), 
gives us a homotopy equivalence
\[
\tw_{\d^c}\left(\bigoplus_{s=0}^b\qdeg^{s(b-1)} \tdeg^s  \MCCSmin_{a,b}^s\right) 
\simeq 
\tw_{D^c}\left(\bigoplus_{s=0}^b \qdeg^{s(b-1)} \tdeg^s
\left\llbracket
\begin{tikzpicture}[scale=.35,smallnodes,rotate=90,anchorbase]
\draw[very thick] (1,-1) to [out=150,in=270] (0,0); 
\draw[line width=5pt,color=white] (0,-2) to [out=90,in=270] (.5,0) to [out=90,in=270] (0,2);
\draw[very thick] (0,-2) node[right=-2pt]{$a$} to [out=90,in=270] (.5,0) 
	to [out=90,in=270] (0,2) node[left=-2pt]{$a$};
\draw[very thick] (1,1) to (1,2) node[left=-2pt]{$b$};
\draw[line width=5pt,color=white] (0,0) to [out=90,in=210] (1,1); 
\draw[very thick] (0,0) to [out=90,in=210] (1,1); 
\draw[very thick] (1,-2) node[right=-2pt]{$b$} to (1,-1); 
\draw[very thick] (1,-1) to [out=30,in=330] node[above=-2pt]{$s$} (1,1); 
\end{tikzpicture}
\right\rrbracket 
\right)
\]
for some twist $D^c$, which also strictly increases the index $s$.
\end{proof}

\appendix

\section{Foams and singular Soergel bimodules}\label{s:FoamSSBim}

As is well-known in certain circles, 
the main results of \cite{WebSchur} and \cite{QR} taken together imply that the 
$k\to \infty$ limit of the monoidal $2$-category of ``enhanced $\slnn{k}$ foams''
(i.e. $\glnn{k}$ foams) from \cite{QR} is equivalent to the monoidal $2$-category 
of singular Bott-Samelson bimodules.

We record the bimodule morphisms corresponding 
to the (non-isomorphism) generating foams in \cite[Definition 3.1]{QR}. 
Let $\partial_i\colon \K[x_1,\ldots,x_N] \to \K[x_1,\ldots,x_N]$ be the $i^{th}$
Demazure operator
\[
\partial_i(f) := 
\frac{f(\ldots,x_i,x_{i+1},\ldots) - f(\ldots,x_{i+1},x_i,\ldots)}{x_i - x_{i+1}}
\]
and let $\partial_{a,b} \colon R^{a,b} \to R^{a+b}$ be the Sylvester operator
\[
\partial_{a,b} :=
(\partial_{b}\cdots \partial_{1})(\partial_{b+1} \cdots \partial_{2}) \cdots
	(\partial_{a+b-1} \cdots \partial_a) \, .
\]
We now record
\begin{align*}
\un :=
\begin{tikzpicture} [scale=.45,fill opacity=0.2,anchorbase,tinynodes]
	\path [fill=red] (4.25,-.5) to (4.25,2) to [out=170,in=10] (-.5,2) to (-.5,-.5) to 
		[out=0,in=225] (.75,0) to [out=90,in=180] (1.625,1.25) to [out=0,in=90] 
			(2.5,0) to [out=315,in=180] (4.25,-.5);
	\path [fill=red] (3.75,.5) to (3.75,3) to [out=190,in=350] (-1,3) to (-1,.5) to 
		[out=0,in=135] (.75,0) to [out=90,in=180] (1.625,1.25) to [out=0,in=90] 
			(2.5,0) to [out=45,in=180] (3.75,.5);
	\path[fill=blue] (.75,0) to [out=90,in=180] (1.625,1.25) to [out=0,in=90] (2.5,0);
	\draw [very thick,directed=.55] (2.5,0) to (.75,0);
	\draw [very thick,directed=.55] (.75,0) to [out=135,in=0] (-1,.5);
	\draw [very thick,directed=.55] (.75,0) to [out=225,in=0] (-.5,-.5);
	\draw [very thick,directed=.55] (3.75,.5) to [out=180,in=45] (2.5,0);
	\draw [very thick,directed=.55] (4.25,-.5) to [out=180,in=315] (2.5,0);
	\draw [very thick, red, directed=.75] (.75,0) to [out=90,in=180] (1.625,1.25)
		to [out=0,in=90] (2.5,0);
	\draw [very thick] (3.75,3) to (3.75,.5);
	\draw [very thick] (4.25,2) to (4.25,-.5);
	\draw [very thick] (-1,3) to (-1,.5);
	\draw [very thick] (-.5,2) to (-.5,-.5);
	\draw [very thick,directed=.55] (4.25,2) to [out=170,in=10] (-.5,2);
	\draw [very thick, directed=.55] (3.75,3) to [out=190,in=350] (-1,3);
	\node [blue, opacity=1]  at (1.625,.5) {$a{+}b$};
	\node[red, opacity=1] at (3.5,2.6) {$b$};
	\node[red, opacity=1] at (4,1.75) {$a$};		
\end{tikzpicture}
\longleftrightarrow
\left\{
\begin{aligned}
R^{a,b} \otimes_{R^{a+b}} R^{a,b} &\to R^{a,b} \\
f \otimes g &\mapsto fg
\end{aligned} \right.
\, &, \quad
\col :=
\begin{tikzpicture} [scale=.5,fill opacity=0.2,anchorbase,tinynodes]
	\path[fill=blue] (-.75,-4) to [out=90,in=180] (0,-2.5) to [out=0,in=90] 
		(.75,-4) .. controls (.5,-4.5) and (-.5,-4.5) .. (-.75,-4);
	\path[fill=red] (-.75,-4) to [out=90,in=180] (0,-2.5) to [out=0,in=90] (.75,-4) -- 
		(2,-4) -- (2,-1) -- (-2,-1) -- (-2,-4) -- (-.75,-4);
	\path[fill=blue] (-.75,-4) to [out=90,in=180] (0,-2.5) to [out=0,in=90] 
		(.75,-4) .. controls (.5,-3.5) and (-.5,-3.5) .. (-.75,-4);
	\draw[very thick, directed=.55] (2,-1) -- (-2,-1);
	\path (.75,-1) .. controls (.5,-.5) and (-.5,-.5) .. (-.75,-1); 
	\draw [very thick, red, directed=.65] (.75,-4) to [out=90,in=0] 
		(0,-2.5) to [out=180,in=90] (-.75,-4);
	\draw[very thick] (2,-4) -- (2,-1);
	\draw[very thick] (-2,-4) -- (-2,-1);
	\draw[very thick,directed=.55] (2,-4) -- (.75,-4);
	\draw[very thick,directed=.55] (-.75,-4) -- (-2,-4);
	\draw[very thick,directed=.55] (.75,-4) .. controls (.5,-3.5) and (-.5,-3.5) .. (-.75,-4);
	\draw[very thick,directed=.55] (.75,-4) .. controls (.5,-4.5) and (-.5,-4.5) .. (-.75,-4);
	\node [red, opacity=1]  at (1.25,-1.25) {$a{+}b$};
	\node[blue, opacity=1] at (-.25,-3.4) {$b$};
	\node[blue, opacity=1] at (.25,-4.1) {$a$};
\end{tikzpicture}
\longleftrightarrow
\left\{
\begin{aligned}
R^{a,b} &\to R^{a+b} \\
f &\mapsto \partial_{a,b}(f)
\end{aligned} \right. 
\\
\zip :=
\begin{tikzpicture} [scale=.45,fill opacity=0.2,anchorbase,tinynodes]
	\path [fill=red] (4.25,2) to (4.25,-.5) to [out=170,in=10] (-.5,-.5) to (-.5,2) to
		[out=0,in=225] (.75,2.5) to [out=270,in=180] (1.625,1.25) to [out=0,in=270] 
			(2.5,2.5) to [out=315,in=180] (4.25,2);
	\path [fill=red] (3.75,3) to (3.75,.5) to [out=190,in=350] (-1,.5) 
		to (-1,3) to [out=0,in=135] (.75,2.5) to [out=270,in=180] (1.625,1.25) 
			to [out=0,in=270] (2.5,2.5) to [out=45,in=180] (3.75,3);
	\path[fill=blue] (2.5,2.5) to [out=270,in=0] (1.625,1.25) to [out=180,in=270] (.75,2.5);
	\draw [very thick,directed=.55] (4.25,-.5) to [out=170,in=10] (-.5,-.5);
	\draw [very thick, directed=.55] (3.75,.5) to [out=190,in=350] (-1,.5);
	\draw [very thick, red, directed=.75] (2.5,2.5) to [out=270,in=0] (1.625,1.25) 
		to [out=180,in=270] (.75,2.5);
	\draw [very thick] (3.75,3) to (3.75,.5);
	\draw [very thick] (4.25,2) to (4.25,-.5);
	\draw [very thick] (-1,3) to (-1,.5);
	\draw [very thick] (-.5,2) to (-.5,-.5);
	\draw [very thick,directed=.55] (2.5,2.5) to (.75,2.5);
	\draw [very thick,directed=.55] (.75,2.5) to [out=135,in=0] (-1,3);
	\draw [very thick,directed=.55] (.75,2.5) to [out=225,in=0] (-.5,2);
	\draw [very thick,directed=.55] (3.75,3) to [out=180,in=45] (2.5,2.5);
	\draw [very thick,directed=.55] (4.25,2) to [out=180,in=315] (2.5,2.5);
	\node [blue, opacity=1]  at (1.625,2) {$a{+}b$};
	\node[red, opacity=1] at (3.5,2.6) {$b$};
	\node[red, opacity=1] at (4,1.75) {$a$};		
\end{tikzpicture}
\longleftrightarrow
\left\{
\begin{aligned}
R^{a,b} &\to R^{a,b} \otimes_{R^{a+b}} R^{a,b}  \\
1 &\mapsto \Schur_{b^a}(\leftX_1 - \rightX_2)
\end{aligned} \right.
\, &, \quad 
\ \cre :=
\begin{tikzpicture} [scale=.5,fill opacity=0.2,anchorbase,tinynodes]
	\path[fill=blue] (-.75,4) to [out=270,in=180] (0,2.5) to [out=0,in=270] 
		(.75,4) .. controls (.5,4.5) and (-.5,4.5) .. (-.75,4);
	\path[fill=red] (-.75,4) to [out=270,in=180] (0,2.5) to [out=0,in=270] (.75,4) -- 
		(2,4) -- (2,1) -- (-2,1) -- (-2,4) -- (-.75,4);
	\path[fill=blue] (-.75,4) to [out=270,in=180] (0,2.5) to [out=0,in=270] 
		(.75,4) .. controls (.5,3.5) and (-.5,3.5) .. (-.75,4);
	\draw[very thick, directed=.55] (2,1) -- (-2,1);
	\path (.75,1) .. controls (.5,.5) and (-.5,.5) .. (-.75,1); 
	\draw [very thick, red, directed=.65] (-.75,4) to [out=270,in=180] 
		(0,2.5) to [out=0,in=270] (.75,4);
	\draw[very thick] (2,4) -- (2,1);
	\draw[very thick] (-2,4) -- (-2,1);
	\draw[very thick,directed=.55] (2,4) -- (.75,4);
	\draw[very thick,directed=.55] (-.75,4) -- (-2,4);
	\draw[very thick,directed=.55] (.75,4) .. controls (.5,3.5) and (-.5,3.5) .. (-.75,4);
	\draw[very thick,directed=.55] (.75,4) .. controls (.5,4.5) and (-.5,4.5) .. (-.75,4);
	\node [red, opacity=1]  at (1.4,3.5) {$a{+}b$};
	\node[blue, opacity=1] at (.25,3.4) {$a$};
	\node[blue, opacity=1] at (-.25,4) {$b$};
\end{tikzpicture}
\longleftrightarrow
\left\{
\begin{aligned}
R^{a+b} &\to R^{a,b} \\
1 &\mapsto 1
\end{aligned} \right.
\end{align*}

\section{Some \texorpdfstring{$\Hom$}{Hom}-space computations}

In \cite[Section 5.4]{KLMS}, a basis is computed for certain $\Hom$-spaces in 
the categorified quantum group $\USd(\slnn{2})$.
This implies the following result, by computing the degree of basis elements.

\begin{prop}\label{prop:LDH}
Let $x,y,p \in \N$ and
suppose that $\lambda+y-x+p \geq 0$.
Up to scalar multiple, there is a unique 
lowest degree $2$-morphism in 
$\Hom_{\USd(\slnn{2})}(\F^{(x+p)}\E^{(y+p)} \ooone_\lambda , \F^{(x)}\E^{(y)} \ooone_{\lambda})$
of degree $p(\lambda+y-x+p)$.
\end{prop}

It is known, e.g. from \cite[Theorem 9]{WebSchur}, that the $2$-functor 
$\Phi \colon \USd(\slnn{2}) \to \SSBim$ is full\footnote{In fact, the failure of this $2$-functor 
to be full is due to the fact that 
$\End_{\USd(\slnn{2})}(\ooone_\lambda) \cong \Lambda \xmapsto{\Phi} \Sym(\B - \Fr)$. 
See e.g. \eqref{eq:bubble}. It becomes full after extending scalars in 
$\End_{\USd(\slnn{2})}(\ooone_\lambda) \cong \Lambda$ to $\Lambda \otimes \Lambda$.} 
in lowest degree. 
Thus, Proposition \ref{prop:LDH} has the following implications for 
$\Hom$-spaces between singular Soergel bimodules.

\begin{cor}\label{cor:Choms}
Let $a,b,d \in \N$, then up to scalar
\[
\chi_0^+ \in \Hom_{\SSBim}(\F^{(d-k)}\E^{(b-k)} \oone_{a,b} , \F^{(d-k-1)}\E^{(b-k-1)} \oone_{a,b})
\]
is the unique map of lowest degree. (It has degree $a-d+1$.)
\end{cor}

\begin{cor}\label{cor:HomForSR}
Let $a,b,c,d,k,p \in \N$. 
Suppose that $k+p \leq \min(b,d-1)$, then
\[
\Hom(\qdeg^r \F^{(d-k-1)} \E^{(b-k)} \oone_{a,b}, \qdeg^s \F^{(d-k-p-1)} \E^{(b-k-p)} \oone_{a,b}) \cong 
\begin{cases}
\K & \text{if } r-s = p (a-d+p+1) \\
0 & \text{if } r-s < p (a-d+p+1) \, .
\end{cases}
\]
\end{cor}

\bibliographystyle{alpha}

\end{document}